\documentclass[reqno, 12pt]{amsart}
 
\usepackage{graphicx}
\usepackage{xcolor}

\usepackage[utf8]{inputenc}

\usepackage{tikz-cd}

\usepackage{amsfonts, amsthm, amsmath, amssymb}
\usepackage{hyperref}
\hypersetup{colorlinks=false}
\usepackage[all]{xy}

\usepackage{pifont}

\usepackage[margin=1in]{geometry}

\newcommand{\AAL}{\mathbf{A}_L}

\RequirePackage{mathrsfs} \let\mathcal\mathscr

\usepackage[colorinlistoftodos,french,bordercolor=black,backgroundcolor=red,linecolor=red,textsize=footnotesize,textwidth=0.75in,shadow]{todonotes}

\usepackage{hyperref}
\usepackage{dsfont}
\usepackage{upgreek}
\usepackage{mathabx,yfonts}
\usepackage{enumerate}
\usepackage{multicol}

\numberwithin{equation}{section}

\newtheorem{theorem}{Theorem}[section] 
\newtheorem{lemma}[theorem]{Lemma}
\newtheorem{proposition}[theorem]{Proposition}
\newtheorem{corollary}[theorem]{Corollary}

\theoremstyle{definition}
\newtheorem*{acknowledgements}{Acknowledgements}
\newtheorem{remark}[theorem]{Remark}
\newtheorem{definition}[theorem]{Definition}

\renewcommand{\d}{\mathrm{d}}
\renewcommand{\phi}{\varphi}

\newcommand{\0}{\mathbf{0}}

\newcommand{\FF}{\mathbb{F}}
\newcommand{\ZZ}{\mathbb{Z}}

\newcommand{\NN}{\mathbb{N}}
\newcommand{\QQ}{\mathbb{Q}}
\newcommand{\RR}{\mathbb{R}}

\newcommand{\vF}{\mathbf{F}}
\newcommand{\vG}{\mathbf{G}}

\newcommand{\napoli}{\Upsilon}

\newcommand\cB{\mathcal{B}}

\newcommand{\cO}{\mathcal{O}}

\renewcommand{\leq}{\leqslant}

\renewcommand{\geq}{\geqslant}

\renewcommand{\c}{\mathbf{c}}

\renewcommand{\v}{\mathbf{v}}

\newcommand{\z}{\mathbf{z}}

\renewcommand{\b}{\mathbf{b}}

\renewcommand{\k}{\mathbf{k}}

\renewcommand{\r}{\mathbf{r}}

\renewcommand{\Im}{\mathrm{Im}}

\DeclareMathOperator{\vol}{vol}

\DeclareMathOperator{\res}{Res}

\DeclareMathOperator{\moo}{mod} 
\renewcommand{\bmod}[1]{\,(\moo{#1})}

\let\emptyset\varnothing

\DeclareSymbolFont{bbold}{U}{bbold}{m}{n}
\DeclareSymbolFontAlphabet{\mathbbold}{bbold}

\newcommand{\md}[1]{  \left(\textnormal{mod}\ #1\right)}
 
\renewcommand{\P}{\mathbb{P}}
\newcommand{\Q}{\mathbb{Q}} 

\newcommand{\F}{\mathbb{F}}
\newcommand{\N}{\mathbb{N}}
\newcommand{\R}{\mathbb{R}}
\newcommand{\B}{\mathcal{B}}
\newcommand{\Z}{\mathbb{Z}}
\renewcommand{\l}{\left}

\renewcommand{\r}{\right}
\renewcommand{\b}{\mathbf}
\renewcommand{\c}{\mathcal}
\renewcommand{\epsilon}{\varepsilon}

\renewcommand{\leq}{\leqslant}
\renewcommand{\geq}{\geqslant}

\DeclareMathOperator*{\Osum}{\sum{}^\dagger}

\newcommand{\vn}{\mathbf{n}}
\newcommand{\vm}{\mathbf{m}}
\newcommand{\va}{\mathbf{a}}

\newcommand{\vt}{\mathbf{t}}
\newcommand{\vu}{\mathbf{u}}
\newcommand{\vv}{\mathbf{v}}
\newcommand\eps\varepsilon

\newcommand{\PP}{\mathbb{P}}
\renewcommand{\AA}{\mathbb{A}}

\newcommand\ddet{{\updelta_{\mathrm{det}}}} 
\newcommand\ddetf{{\widehat{\updelta}_{\mathrm{det}}}}

\newcommand\dran{{\updelta_{\mathrm{rand}}}}
\newcommand\dranh{{\widehat{\updelta}_{\mathrm{rand}}}}

\newcommand\hs[2]{\left(#1,#2\right)}
\newcommand\ls[2]{\left(\frac{#1}{#2}\right)}
\newcommand{\vA}{\mathbf{A}}

\newcommand{\vx}{\mathbf{x}}
\newcommand{\vy}{\mathbf{y}}

\newcommand{\vzero}{\mathbf{0}}

\newcommand\where{\ :\ }
\newcommand\one{\mathds{1}}
\newcommand\valpha{\boldsymbol{\alpha}}
\newcommand\vbeta{\boldsymbol{\beta}}
\newcommand\veta{\boldsymbol{\eta}}

\newcommand{\s}{\sigma}
\newcommand{\Sing}{\mathfrak{S}}
\newcommand{\hSing}{\widehat{\mathfrak{S}}}

\newcommand{\hS}{\widehat{S}}

\newcommand{\abs}[1]{\left|#1\right|}

\newcommand{\cF}{\mathcal{F}}

\newcommand{\cFZ}{\mathcal{F}_{\ZZ}}

\newcommand{\cFZELS}{\mathcal{F}_{\ZZ,\operatorname{ELS}}}
\DeclareMathOperator{\lcm}{lcm}

\title
[Random conic bundle surfaces satisfy the Hasse principle]
{Random conic bundle surfaces satisfy the Hasse principle}

\author{Christopher Frei}
\address{Technische Universit\"at Graz\\
Institut f\"ur Analysis und Zahlentheorie\\
Kopernikusgasse 24/II\\
A-8010, Graz\\
Austria}
\email{frei@math.tugraz.at}

\author{Efthymios Sofos} 
\address{Universit\` a di Roma Tor Vergata\\ Dipartimento di Matematica\\00133, Roma, Italy}
\email{efthymios.sofos@uniroma2.eu}

\date{April 8, 2026}

\begin{document}\begin{abstract}We establish the Hasse principle 
for $100\%$ of conic bundles over $\mathbb{P}^1_{\Q}$. 
\end{abstract} 

\maketitle\setcounter{tocdepth}{1}
\tableofcontents \section{Introduction} 
The Hasse principle, if it holds for a given variety $X$ over a 
number field $k$, is the main tool to decide the most fundamental 
arithmetic property of $X$, namely whether $X$ has rational points. 
If $X$ is smooth, projective, geometrically integral and 
geometrically rationally connected,  
a conjecture of Colliot-Th\'{e}l\`ene 
(see \cite[p.174]{MR2011747}) asserts that 
the Brauer-Manin obstruction is the only obstruction to the 
Hasse principle (and weak approximation) for $X$.

Significant effort has been devoted to verifying the conjecture 
for varieties with fibrations in which each of the fibres satisfies
the Hasse principle. The archetypal examples of such varieties are 
conic bundle surfaces over $\Q$, i.e. smooth projective surfaces $X$
over $\Q$ equipped with a dominant morphism $\pi:X\to \P^1_\Q$, all
fibres of which are conics. 

\subsection{Arithmetic of conic bundle surfaces}
Concretely, conic bundle surfaces arise as smooth projective models
of surfaces defined in $\AA_\QQ^1\times\P_\QQ^2$ by equations of the
shape\begin{equation}\label{eq:simpleequation_affine}
f_1(t) x^2+ f_2(t)y^2=f_3(t)z^2,\end{equation}
with polynomials $f_i\in\ZZ[t]$ whose product $f_1f_2f_3$ is  separable. These surfaces occur naturally in 
geometry and their arithmetic has been studied extensively; 
a summary can be found in the work of 
Colliot-Th\'{e}l\`ene~\cite{MR1104699}. Let $d_i$ be the degree of $f_i$. In some cases, the existence of rational points is obvious. This holds, in particular, if one of the $f_i$ has a linear factor over $\Q$, yielding a singular fibre of $\pi$ defined over $\Q$. In general, Colliot-Th\'{e}l\`ene's conjecture is widely open for conic bundle surfaces and has spawned significant activity.

Smooth projective models of~\eqref{eq:simpleequation_affine} with 
$(d_1,d_2,d_3)=(2,2,0)$ are del
Pezzo surfaces of degree $4$, for which the conjecture 
was proved by
Colliot-Th\'{e}l\`ene~\cite{MR1104699}.
The cases with $(d_1,d_2,d_3)=(0,0,4)$
correspond to Ch\^atelet surfaces and were settled by 
Colliot-Th\'{e}l\`ene, Sansuc and 
Swinnerton-Dyer~\cite{MR870307,MR876222}.
Cases with  $(d_1,d_2,d_3)=(0,0,6)$ and $f_3$ being 
a product of a quadratic and a quartic irreducible polynomial 
were studied by Swinnerton-Dyer~\cite{MR1751446}. 
The cases $(d_1,d_2,d_3)=(2,2,2)$ are open; these
correspond to specific types of del Pezzo surfaces of 
degree $2$, see~\cite[Proposition 5.2]{MR3194818}.
Building on 
\textit{descent} ideas of
Colliot-Thélène and Sansuc~\cite{MR667708}
that proved  the 
conjecture conditionally upon Schinzel's hypothesis
and  using
additive combinatorics
results by Green--Tao~\cite{MR2680398}
and Green--Tao--Ziegler~\cite{MR2950773},
Browning--Matthiesen--Skorobogatov~\cite{MR3194818}
and Harpaz--Skorobogatov--Wittenberg~\cite{MR3292295}
proved that the Brauer--Manin obstruction is the only
obstruction to weak approximation 
for arbitrary degrees $d_i$, requiring that
each $f_i$ is a product of
linear factors over $\Q$.

The Hasse principle is not well understood in other cases with $d_1+d_2+d_3>6$. In \cite{MR452704}, Skorobogatov and Sofos studied it from a statistical perspective, ordering conic bundles \eqref{eq:simpleequation_affine} with arbitrary fixed degrees $d_1,d_2,d_3$ by the absolute values of the coefficients of all $f_i$. Their results imply that a positive proportion of conic bundles \eqref{eq:simpleequation_affine} have rational points and thus satisfy the Hasse principle. 

Our results show that the Hasse principle is in fact a typical property of conic bundles, in the sense that the proportion satisfying it is $100\%$.

\begin{theorem}\label{thm:special_affine}
Fix arbitrary strictly positive integers $d_1,d_2,d_3$.\ 
\begin{enumerate}
\item Let $f_1,f_2,f_3\in\Z[t]$ run through all polynomials of respective degrees bounded by $d_1,d_2,d_3$. When ordered by absolute value of the coefficients,  $100\%$ of the equations \eqref{eq:simpleequation_affine} define conic bundle surfaces that satisfy the Hasse principle.
\item Let $f_1,f_2\in\Z[t]$ run through all polynomials of respective degrees bounded by $d_1,d_2$. When ordered by absolute value of the coefficients,  $100\%$ of the equations 
$$f_1(t)x^2+f_2(t)y^2=z^2$$
define conic bundle surfaces that satisfy the Hasse principle.
\end{enumerate}
\end{theorem}
As $100\%$ of polynomials are irreducible,
Theorem~\ref{thm:special_affine} sees only conic bundles in which all $f_i$ are irreducible. 
As will be explained in Remark \ref{rem:obstructions},
counter-examples to the Hasse principle 
are known to occur when other factorisations are allowed. Even then, these counter-examples are
rare, as we show in the following generalisation 
of Theorem \ref{thm:special_affine}. It proves the Hasse principle with probability $1$
for all degrees and all prescribed factorisations, 
i.e. for conic bundle surfaces given by equations of the form
\begin{equation}\label{eq:general_affine}
\Big( \prod_{j=1}^{m_1}f_{1j}(t)\Big)  x^2+
\Big(   \prod_{j=1}^{m_2}f_{2j}(t) \Big) 
y^2=  \Big(     \prod_{j=1}^{m_3}f_{3j}(t) \Big)    z^2,
\end{equation}
where an empty product is understood as $1$. Previously, it was 
known from \cite{MR452704} that the probability is strictly positive.
\begin{theorem}\label{thm:main_affine}
Let $m_1,m_2,m_3\in\ZZ_{\geq 0}$ with $m_1m_2> 0$. 
For $i\in\{1,2,3\}$ and $j\in \{1,\ldots,m_i\}$, let $d_{ij}\in\NN$.
Let $(f_{ij})_{i,j}$ run through all tuples of polynomials in 
$\ZZ[t]$ with $\deg f_{ij}\leq d_{ij}$ for all $i,j$, ordered by 
the maximal absolute value of all coefficients. Then $100\%$ of the 
equations \eqref{eq:general_affine} define conic bundle surfaces 
that satisfy the Hasse principle.\end{theorem}

Note that, by the Lang-Nishimura theorem, the choice of smooth projective model is irrelevant for the validity of the Hasse principle. Triviality of the generic Brauer group was verified
in~\cite[\S 2]{MR4789075}. Therefore, Theorem~\ref{thm:special_affine} and Theorem~\ref{thm:main_affine} are
expected consequences of Colliot-Th\'{e}l\`ene's conjecture. A $100\%$ Hasse principle statement would be empty unless 
a positive percentage of surfaces  
is everywhere locally soluble; in case of our Theorem \ref{thm:main_affine},  
a positive proportion was proved
to have a $\Q$-point (and thus be  
everywhere locally soluble)
in~\cite[Theorem 1.4]{MR452704}. 

There is 
extensive
literature on the local-global
principle for \eqref{eq:simpleequation_affine}.
Hasse's proof of the local-global principle for quadratic forms uses Dirichlet's theorem on primes in arithmetic progressions to pass from three to four variables.
Colliot-Thélène and Sansuc~\cite{MR667708}
realised that Schinzel's hypothesis (H) can play a similar role in other situations. Conditionally on this hypothesis, they proved 
that 
varieties of the form
$$x^2 + ay^2 = f(t)z^2$$ 
over $\QQ$ with $f$ irreducible satisfy the Hasse principle and weak approximation.
This result opened the way for many
subsequent developments.
Serre~\cite[\S II, Annexe]{MR1867431} extended 
their argument to
arbitrary families
of Severi--Brauer varieties over a number field, 
thus in particular to equation 
\eqref{eq:simpleequation_affine} above.
The proof was detailed by Colliot-Thélène
and Swinnerton-Dyer in \cite{MR1285781}. 
The work by Harpaz--Skorobogatov--Wittenberg \cite{MR3292295}
mentioned earlier replaces Schinzel's hypothesis (H) in this approach by the Green--Tao theorem.
Further 
research
on the topic
includes  work by 
Swinnerton-Dyer~\cite{MR1291734}, 
Colliot-Th\'{e}l\`ene--Skorobogatov--Swinnerton-Dyer~\cite{MR1603908},
Wittenberg~\cite{MR2307807},
Wei~\cite{MR3293154} 
and
Harpaz--Wittenberg~\cite{MR3432584}.
 
\begin{remark} \label{rem:obstructions} 
As already mentioned, Colliot-Th\'{e}l\`ene, Sansuc, and Swinnerton-Dyer~\cite{MR870307,MR876222} proved the 
Hasse principle for
$$  x^2+   y^2=  f(t)  z^2$$
when $f$ is quartic, except in the case where $f$ is a product of two irreducible 
quadratics. In that case, Iskovskih~\cite{MR286743} had already produced 
 counterexamples. Work of 
Colliot-Th\'{e}l\`ene--Coray--Sansuc 
\cite{MR592151}, la Bret\`eche--Browning \cite{MR3198755} and 
Rome \cite{MR3976470} shows that in this exceptional case there are 
$\gg H^2$ counterexamples among the $\asymp H^6$ pairs of quadratic polynomials 
of height $H$. 
\end{remark}

\subsection{Statistical approach}
Poonen and Voloch~\cite{MR2029869}  
were the first to propose 
a statistical way of approaching 
the Hasse principle;
they conjectured that random Fano hypersurfaces 
satisfy the Hasse principle, a statement that 
  was  proved in dimension $\geq 3$ by  
Browning, Le Boudec and Sawin~\cite{MR4564262}.
Earlier work of
Br\"{u}dern--Dietmann~\cite{MR3177289}
settled the case of diagonal hypersurfaces of degree $d$
in $n$ variables, when $2[n/2]\geq 3d$.  
As mentioned above, 
Skorobogatov--Sofos~\cite{MR452704,MR4789075} made unconditional `on average'
the Schinzel Hypothesis approach 
of Colliot-Th\'{e}l\`ene and Sansuc~\cite{MR667708}
 to prove the Hasse principle for a positive percentage of conic bundle
 surfaces. They used circle method arguments
 together with Vinogradov-type estimates for exponential sums.
 Browning--Sofos--Teräväinen~\cite{arXiv:2212.10373}
 then
 established the integral Hasse principle for $100\%$ of 
 generalized Ch\^atelet 
 varieties 
 of the form $N_{K/\Q}(\b x)=f(t)$, 
 where   $N_{K/\Q}$ is the norm form of an
 arbitrary number field extension and $f$ is a random integer polynomial with positive 
 leading coefficient.  When $[K:\Q]$ divides $\deg(f)$
this was recently modified   
 to prove the Hasse principle for rational points
 with probability $1$ by Diao
\cite{arXiv:2506.18065}. 
In addition to the corresponding norm-representation functions, these works also apply to the M\"obius, von Mangoldt and Liouville functions.
They do not rely on the circle method,
 instead, they develop
an asymptotic result for averages of 
 arithmetic functions $f:\Z\to \mathbb C$ 
 over the values of random integer polynomials 
 using  multiplicative number theory and zeros 
 of $L$-functions. 
We take a  different route by  injecting summability
kernels directly into 
a circle method argument. 
This enables us to control the averages of a broad 
class of arithmetic functions  $f:\Z^m\to \mathbb C$,
under the sole hypothesis that we know its 
distribution  in  arithmetic progressions 
of small moduli.  \subsection{Main  innovations} 
We achieve our 100\%-results by avoiding arguments using primes. Instead, we develop machinery to deal directly with all fibres, relying on several key innovations:  
 \begin{itemize}\item 
Heat kernels  
 are  used as weights 
 for the coefficients of the random polynomials.    
This leads to a Fourier-analytic 
set-up in which the 
transformation law for the 
Jacobi theta function implies 
super-exponential decay
almost everywhere on the torus. 

\item 
This leads to second moment estimates 
of very general functions 
$f:\Z^m\to \mathbb C$
over values of random polynomials
assuming only  weak     equidistribution in 
arithmetic 
progressions. 
The  results are formulated 
in a way that 
is straightforward to employ in   applications, see 
Corollary~\ref{cor:chiocolato}.

\item We develop a detector function for the existence of
rational points on
conics,
which 
we  decompose into a random and a deterministic part
using Hilbert's reciprocity law. 
The random part satisfies equidistribution properties   required in the previous bullet point.
\item 
To define our detector function, we introduce 
an analytic version of the Hilbert symbol 
which has average $0$ over $\Z_p^2$. 
This construction enables us to bound certain character sums,
thereby reducing the required level of distribution in dispersion arguments. 
\end{itemize}

\subsection{Conic bundle surfaces}\label{sec:def_conic_bundle} 
Throughout, we work with explicit conic bundle surfaces, whose construction we briefly recall here. For details, see \cite[\S 1.3]{MR3851328}. Let $a_1,a_2,a_3\in\ZZ_{\geq 0}$ and $e\in\ZZ$. Let 
$d_i:=2 a_i+e\geq 0$,
and let $G_i\in\ZZ[t_1,t_2]$ be binary forms of degree $d_i$, for $i=1,2,3$, such that $G_1G_2G_3$ is separable.
Then the equation
\begin{equation}\label{eq:def_conic_bundle}
    G_1(t_1,t_2)x^2+G_2(t_1,t_2)y^2 = G_3(t_1,t_2)z^2
\end{equation}
defines a smooth hypersurface $X_\vG$ of bidegree $(e,2)$ in the $\PP^2$-bundle $\FF(a_1,a_2,a_3)$ over $\PP^1_\QQ$ defined as the projectivisation of the vector bundle $\cO_{\P^1}(a_1)\oplus \cO_{\P^1}(a_2)\oplus \cO_{\P^1}(a_3)$.

In more concrete terms, \eqref{eq:def_conic_bundle} is bihomogeneous of bidegree $(e,2)$ with respect to the action
\begin{equation}\label{eq:action}
    (\lambda,\mu)\cdot ((t_1,t_2),(x,y,z)) = ((\lambda t_1, \lambda t_2), (\lambda^{-a_1}\mu x, \lambda^{-a_2}\mu y,\lambda^{-a_3}\mu z)).
\end{equation}
For any field $K\supseteq \Q$, points in $X_\vG(K)$ are represented by orbits $((t_1:t_2),(x:y:z))$ of this action of $(K^\times)^2$ on $(K^2\smallsetminus\{\vzero\})\times (K^3\smallsetminus\{\vzero\})$ that satisfy \eqref{eq:def_conic_bundle}. 

In particular, each point in $X_\vG(\Q)$ is represented by four tuples $((t_1,t_2),(x,y,z))\in\Z^5$ with $\gcd(t_1,t_2)=\gcd(x,y,z)=1$, satisfying \eqref{eq:def_conic_bundle}. The hypersurface $X_\vG$ is a conic bundle surface via the morphism $\pi:X_\vG\to \PP^1_\Q$ given by $((t_1:t_2),(x:y:z))\mapsto (t_1:t_2)$.

If the polynomials $f_i$ and forms $G_i$ satisfy $f_i(t)=G_i(t,1)$, then the preimage under $\pi$ of $\{t_2\neq 0\}$ is isomorphic with \eqref{eq:simpleequation_affine}. Hence, $X_\vG$ is a smooth projective model of \eqref{eq:simpleequation_affine}.

 \subsection{Hasse principle theorems}\label{sec:HP_theorems}
Here we state our main results, precise versions of Theorem \ref{thm:main_affine} formulated in terms of the conic bundle surfaces introduced above.

Let  $m_1,m_2,m_3$ be arbitrary non-negative integers such that $m_1 m_2>0$. For $i=1,2,3$ and $j=1,\ldots, m_i$
we let 
$d_{ij}$ be arbitrary strictly positive integers.
Throughout this paper  we  use the symbol 
$ F_{ij}(t_1,t_2)$ to denote a  binary form of degree $d_{ij}$ and  denote $$\cF:= \{\b F=(F_{ij})\where F_{ij}\in \R[t_1,t_2] \text{ forms with } 
 \deg(F_{ij})=d_{ij} \ \forall i,j \}$$ and 
 $$\cF_\Z:= \{\b F=(F_{ij}) \where F_{ij}\in \Z[t_1,t_2] \text{ forms with } 
 \deg(F_{ij})=d_{ij} \ \forall i,j \}.$$  
Denote 
$$  m:=m_1+m_2+m_3,
\hspace{1cm} 
d:=\sum_{i,j} d_{ij},  \hspace{1cm} 
\text{ and }  \hspace{1cm} 
   d_i:=\sum_{j=1}^{m_i}d_{ij}
\quad\text{ for }1\leq i\leq 3
.$$ 
We will assume that all $d_i$ have the same parity and denote $a_i:=\lfloor d_i\rfloor$, thus writing $d_i=2a_i+e$ for some fixed $e\in\{0,1\}$. Let $X_\vF$ be the hypersurface defined in $\FF(a_1,a_2,a_3)$ by the equation
\begin{equation}\label{eq:conic_bundle_eqn_mult}
  \Big( \prod_{j=1}^{m_1}F_{1j}(\vt)\Big)  x^2+
  \Big(   \prod_{j=1}^{m_2}F_{2j}(\vt) \Big) 
  y^2=\Big(     \prod_{j=1}^{m_3}F_{3j}(\vt) \Big)    z^2,
 \end{equation}
 which is bihomogeneous of bidegree $(e,2)$ with respect to the action \eqref{eq:action}. It is a conic bundle surface whenever $\prod_{i,j}F_{ij}$ is separable.
 Let $\pi:X_\vF\to\PP^1_\QQ$ be the morphism $(\vt,\vx)\mapsto\vt$. 

 For a binary form $F$  we denote  the maximum of the absolute values of its coefficients by $$ h(F)$$  and we set $h(F_1,\ldots,F_N):=
 \max_i\{h(F_i)\}$.  For $H\geq 1$, we let 
  \begin{equation}\label{def:setsFFFZZZHHH}
\cF(H) := \{\b F \in \cF\ :\ h(\b F)\leq H\}
\hspace{1cm} \text{ and }
\hspace{1cm} \cF_\Z(H) := \{\b F \in \cF_\Z\ :\ h(\b F)\leq H\}.
 \end{equation} 

Our main result is a more precise version of Theorem \ref{thm:main_affine}, formulated in terms of binary forms as above.
 
\begin{theorem}\label{thm:main} Fix $m_i$ and $d_{ij}$ 
as above and   any $\alpha \in (0,1)$. 
For all large enough $H\geq 1$,
the proportion of $\b F\in \c F_\Z(H)$ 
for which $X_\b F$ is a conic bundle satisfying the Hasse principle exceeds  
  $1-(\log \log H)^{-\alpha}$.
\end{theorem} 

This follows immediately from the 
following stronger result, providing a lower bound on the number of soluble fibres $(X_\vF)_\vt := \pi^{-1}(\vt)$. 
\begin{theorem}\label{thm:chebychev}  
Fix $\gamma \in (0, \frac{1}{50})$, $\alpha \in (0,1)$, and assume that $H$ is sufficiently large.
Then, for all $\b F\in \c F_\Z(H)$, with the exception of
possibly  $\#\c F_\Z(H)/
(\log \log H)^{\alpha}$ many, the hypersurface $X_\vF$ is a conic bundle surface and satisfies
$$ \#\{\vt\in \P^1(\Q):  
(X_{\b F})_\vt \textrm{ has a }\Q\textrm{-point}\}>
H^{\gamma/d} $$ whenever $X_\b F$ is 
everywhere locally soluble.
\end{theorem} 
Since the number of singular geometric fibres is bounded by $d\ll 1$, Theorem \ref{thm:chebychev}
shows that $100\%$ of 
everywhere locally soluble conic bundles $X_\b F$
have  
rational points on smooth fibres. 
In \S\ref{sec:chebyshev_multiple}, we deduce 
Theorem \ref{thm:chebychev}
from Theorem \ref{thm:L^2theorem}, stated later after introducing the necessary notation.
We will deduce Theorem \ref{thm:main_affine} from Theorem \ref{thm:main} in \S \ref{sec:proof_main_affine}.

\subsection{Sums of arithmetic functions over values of  binary forms}\label
{subsectionsumblty}Let $F_1,\ldots, F_m$ be
integer binary forms of respective degrees 
$d_1,\ldots, d_m$ and $f:\Z^m\to \mathbb C$ be 
any function. We are interested in giving asymptotics 
 for the sum  
\begin{equation}\label{eq:defsums}
\sum_{ \b n\in \Z^2 \cap[-x,x]^2} f(F_1(\b n), \ldots, F_m(\b n) ) .\end{equation}
Special  $f$, such as 
the von Mangoldt or the M\"obius function,
are out of reach for large $d_i$. 
We thus focus on a statistical 
point of view and consider typical $F_i$
by  randomizing their
coefficients. In particular, for arbitrary 
fixed   $d_1,\ldots, d_m$ 
we consider the $L^2$-mean
  $$ \sum_{\substack{ \b F\in \Z[t_1,t_2]^m\\
 \\ h(\b F) \leq H 
 }} \Bigg|\sum_{
\b n\in \Z^2 \cap[-x,x]^2 } 
f(F_1(\b n), \ldots, F_m(\b n) ) 
\Bigg|^2,$$ where the outer sum is over vectors 
of integer forms 
$\b F=(F_1,\ldots, F_m)$ with $\deg(F_i)=d_i$ for all $i$.
Our results show  that  the $L^2$-mean can be bounded non-trivially 
when $f$ has an equidistribution property in arithmetic progressions of small moduli.

We state a very 
special   case with $m=1$ here; 
stronger and more general 
versions are presented in 
\S \ref{s:summability}. 

\begin{theorem}\label{thm:mccoytyner}Fix any $B,C>0$ and let $f:\Z\to \mathbb C$ be any function satisfying $$
|f( n)| \leq B  \begin{cases}
\tau(|n|)^C ,  & n\neq 0 \\
1, & n=0 \end{cases}$$ for all $n\in \Z$,
where $\tau$ is the divisor function.
For any $N>0$ and  any   
 strictly positive integer 
$d$ there exists $\kappa(B,C,d,N)>0$ such that 
for all $H\geq 3$ 
and all $x$ in the range 
 $(\log H)^{\kappa}\leq x \leq H$
 we have 
\begin{align*}
\frac{1}{H^{1+d}} 
\sum_{\substack{ F \in \Z[t_1,t_2]\operatorname{ form}\\\deg(F)=d \\ 
h(F) \leq H  }} \Bigg|\frac{1}{x^2}
\sum_{\substack{  
\b n\in \Z^2 \cap[-x,x]^2 }} 
f( F(\b n) )  \Bigg|^2  
\ll & \frac{(\log H)^{\kappa}x^{4d}}{H^2}
\max_{\substack{ q\in \NN 
\\ q \leq 2x^{2d}  }} 
\c E_f((1+d)x^{d}H;q )^2 
\\ & +  \frac{1}{(\log x)^N},\end{align*}
where the implied constant depends only on 
$B,C,d,N$ and we denote  
$$\c E_f( T ; q )
:= \max_{ r\in  \Z/q\Z}  
\ \sup_{\substack{   v \in  \R  \\  
|v | \leq T } } \
\l|  \sum_{\substack{ \b n \in \Z,
-T \leq n \leq v  \\ n \equiv  r \md {q}  } }
f( n )\r|.$$ \end{theorem}
This bounds explicitly the second moment over values 
of forms in terms of  the distribution of $f$ 
on arithmetic progressions.
The main idea of the proof is to employ 
heat kernels, meaning that, writing $F=\sum_{j=0}^{d} c_{j} 
t_1^{j}t_2^{d-j}$ we use
$$ \mathds 1_{[-H,H]}(c_{j}) \leq \mathrm e^{\pi}
\exp(-\pi c_{j}^2/H^2)$$ for each $j$.
Using Fourier analysis   identities
this leads to an integral of a product 
of Jacobi theta functions multiplied by 
the exponential sum of $f$. The theta  
terms have sharp decaying properties 
that follow from  the   
transformation laws of the Jacobi theta function;
this eliminates   the contribution of
the minor arcs without any Vinogradov type information 
on $f$. The major arcs are dealt with using information 
on $f$ in arithmetic progressions of small moduli.

\begin{remark}[Applications] 
If we know that 
there are large   constants $A_1,A_2>0$ such that 
for all $1\leq q \leq (\log T)^{A_1}$ and all 
$a\in \Z/q\Z$ one has 
\begin{equation}\label{assumption:siegwalf}
\sum_{\substack{ |n|\leq T \\ n\equiv a \md q}} f(n) \ll \frac{T}{(\log T)^{A_2}}
,\end{equation}
then  applying Theorem~\ref{thm:mccoytyner} 
with $x=(\log H)^M$ for some constant  $M(A_1,A_2)$
gives non-trivial bounds for the average of $f$ 
over the values of random   $F$.
The assumption~\eqref{assumption:siegwalf}
is easy to verify in applications as one often 
knows  a Siegel--Walfisz bound in which 
$A_1,A_2$ are allowed to be arbitrarily large.  
\end{remark}Theorem~\ref{thm:mccoytyner}
is the  special case corresponding to taking $m=1$, $d_1=d$ and 
 $a=1$ in 
Corollary~\ref{cor:chiocolato}.
This corollary      regards   
$f:\Z^m \to \mathbb C$ for any positive integer $m$
and gives explicit constants and more accurate bounds.
Corollary~\ref{cor:chiocolato}
is proved  
at the end of \S\ref{s:specialheat!} 
by using Corollary~\ref{cor:heat:kernel}, 
which is proved in 
 \S\ref{s:specialheat!}  via heat kernels
 and  Theorem~\ref{thm:generalkernels}. This theorem 
 is proved for more general 
 summability kernels in \S\ref{s:beginendpproof}.
\subsection{The analytic Hilbert symbol}
To prove the main Hasse principle statements in this paper,
the natural plan of action is 
to apply Theorem \ref{thm:mccoytyner}
with $f=\updelta-\hat{\updelta}$, where 
$\updelta$ is   a   
Hilbert symbol detector function of 
rational points and $\hat{\updelta}$ is a ``model" 
that mimicks $\updelta$ on arithmetic progressions. 
This     furnishes a second moment involving only
 $\hat{\updelta}$ 
that   needs to be dealt with separately.  
This is still a formidable challenge, which we render feasible through the use of a new detector function relying on a modified definition of the Hilbert symbol. This new 
version
has the advantage of having zero average in a suitable sense, which will lead to the vanishing of certain averages in the analysis of $\hat{\updelta}$.

To describe the alternative
detectors we 
recall that for a 
local field $k$ 
and $a,b\in k^\times$,
the Hilbert symbol 
$(a,b)_k 
\in \{\pm 1\}\subseteq\ZZ
$ is defined as $1$ when 
the plane conic $z^2=ax^2+by^2$ has $k$-rational points
and $-1$ otherwise. When $p\neq 2$ and both 
$a,b\in \Q_p^\times$ have even valuation, then $(a,b)_{\Q_p}=1$.
The main observation is that  
if  we ignore  such $(a,b)$ 
then in the rest of $(\QQ_p^\times)^2$
the Hilbert symbol
takes the values $1$ and $-1$ equally often.
 This ``$0$-average" 
Hilbert  symbol retains enough 
properties to be used for detecting solubility and  it has key
cancellation properties for analytic arguments.
Denote $p$-adic valuation by $v_p$.

\begin{definition}
For a prime $p$ and   $t_1,t_2\in\QQ_p$ we
define 
$\hs{t_1}{t_2}_p'\in\{-1,0,1\}\subseteq\ZZ$ by
\begin{align*}\hs{t_1}{t_2}_p'&:= \begin{cases}
0, &\text{ if }t_1=0\text{ or }t_2=0,\\
0, &\text{ if $p$ odd and }v_p(t_1), v_p(t_2)\text{ both even},
\\0, &\text{ if $p=2$, }v_2(t_1), v_2(t_2)
\text{ both even, and }
\frac{t_1}{2^{v_2(t_1)}}
\not\equiv 
\frac{t_2}{2^{v_2(t_2)}}
\bmod 4,
\\(t_1,t_2)_{\Q_p},&\text{ otherwise}.
  \end{cases} \end{align*} For $\b t \in \R^2$
  we let 
  $\hs{t_1}{t_2}_\infty':=0$ when $t_1t_2=0$
  and  we set $ \hs{t_1}{t_2}_\infty':= (t_1,t_2)_\R$ otherwise. 
\end{definition}  
  Throughout, we normalise the Haar measure on 
 $\Q_p$ so that $\Z_p$ has   measure $1$.
\begin{lemma}\label{lem:new_hilb_integral234}
For any prime $p$ and  
$\beta_1,\beta_2\in\ZZ$  we have  \begin{equation*}
\int_{\substack{ \b t \in \QQ_p^2 \\ v_p(t_i)=\beta_i, i=1,2}}
(t_1,t_2)_p'\mathrm{d}\b t= 0.
\end{equation*} \end{lemma} The proof is given in
\S\ref{subsectnprofzerointegrla}.
Next, we show that $(t_1,t_2)_p'$ is flexible 
enough to detect rational points.  
This depends on the key observation, already hinted at above, that 
\begin{equation}\label{eq:hilbert_key}
(t_1,t_2)_{\QQ_p}=1 \text{ whenever } t_1,t_2 \text{ are in $\QQ_p^\times$ with $(t_1,t_2)_p'=0$},
\end{equation}
which can be made from well-known explicit formulas for the Hilbert symbol (see \cite[Theorem 1 in Chapter III]{MR344216}).
For every prime $p$, we consider $\Z$ as a subset of $\Z_p$ via the natural 
embedding, so $(t_1,t_2)_p'$ is well-defined for $\b t \in \Z^2$. We always understand products indexed by the letter $p$ to be running over primes.

\begin{lemma}\label{lem:detectorcoffee}
For every $\b t \in \Z^2$, the product $$\prod_p(1+\hs{t_1}{t_2}_p')$$  
has only finitely many factors different from one. It 
is either $0$ or a power of $2$. 
It is not equal to $0$ if and only if the conic defined by
$t_1x^2+t_2y^2=z^2$ in $\PP^2$ has a rational point.
\end{lemma}

\begin{proof} 
By definition of $\hs{\cdot}{\cdot}_p'$, every factor is either $0,1$ or $2$.
If $t_1t_2=0$, then $(t_1,t_2)'_p=0$ for all $p$,  
and hence all factors are equal to $1$. In this case, the conic is degenerate and thus has at least one rational point.

Now assume $t_1t_2\neq 0$. If $p\nmid 2t_1t_2$ then  $(t_1,t_2)'_p=0$, hence the corresponding factor is equal to $1$. By \eqref{eq:hilbert_key}, the product is non-zero if and only if $\hs{t_1}{t_2}_{\QQ_p}=1$ for all primes $p$. By Hilbert's product formula and the Hasse principle for conics, this is equivalent to the conic
$t_1x^2+t_2y^2=z^2$ having rational points.
\end{proof} 
Hence, for $\vt=(t_1,t_2)\in\ZZ^2$,  we define our detector 
\begin{equation}
\label{eq:def_Nt}\updelta(\vt):=\prod_{p}(1+\hs{t_1}{t_2}'_p) 
\hspace{0.5cm}
\textrm{ and the quantity } 
\hspace{0.5cm}
N_\vt:= \prod_{p:\ \hs{t_1}{t_2}_p'\neq 0}p,
\end{equation} 
where we recall again that an empty product is defined to be $1$. Note that $(t_1,t_2)_p'\neq 0$ implies that $t_1t_2\neq 0$ and $p\mid 2t_1t_2$, so the product defining $N_\vt$ is finite. We can expand 
\begin{equation}\label{eq:hearingboots}
\updelta(\vt)
=\prod_{p\mid N_\b t} (1+\hs{t_1}{t_2}'_p)
=
\sum_{\substack{s\in \N \\s\mid N_\vt}} 
\ \ \prod_{p\mid s}\hs{t_1}{t_2}'_p
=
\sum_{ s\textrm{ square-free}}  
\ \ \prod_{p\mid s}\hs{t_1}{t_2}'_p.\end{equation}
The  oscillation in the values of 
the modified Hilbert symbol $(\cdot,\cdot)'_p$
means that the majority of $s$
in the  right-hand side sum cancel each other.
Reciprocity determines which terms give rise to 
cancellation. 
\begin{lemma}\label{lem:symmetry} 
For all   $\b t \in  (\Z\setminus\{0\})^2$ 
and   $z\geq 1 $, we have 
$$\sum_{\substack{
s\ \mathrm{ square}\textrm{-}\mathrm{free} \\
s \leq z  } } 
\hspace{ 0.2cm }\prod_{ p\mid s} 
(t_1,t_2)'_p
=
(t_1,t_2)_\R\sum_{\substack{s\ \mathrm{ square}
\textrm{-}\mathrm{free}\\
s \geq \frac{N_\b t}{z} }} 
\hspace{ 0.2cm }\prod_{p\mid s} 
(t_1,t_2)'_p.$$\end{lemma}\begin{proof}  
The only $s$ that make a non-zero contribution to the 
left-hand side sum are those that divide $N_\b t$.
Letting $e:=N_\b t/s$ we 
write this sum as 
$$\sum_{\substack{(s,e) \in \N^2\\ se=N_\b t,  s \leq z   } } 
\hspace{ 0.2cm } \prod_{p\mid s} 
(t_1,t_2)'_{p}=\sum_{\substack{(s,e) \in \N^2\\ se=N_\b t,  s \leq z   } } 
\hspace{ 0.2cm } \prod_{p\mid s} 
(t_1,t_2)_{\Q_p}$$ because $(t_1,t_2)'_p=(t_1,t_2)_{\Q_p}$
whenever $p\mid N_\b t $.
By   Hilbert's reciprocity formula  we get 
$$ \prod_{p\mid s} (t_1,t_2)_{\Q_p} \prod_{p\mid e} (t_1,t_2)_{\Q_p}=
\prod_{p\mid N_\b t} (t_1,t_2)_{\Q_p}=(t_1,t_2)_\R \prod_{p\nmid N_\vt} 
(t_1,t_2)_{\Q_p}=(t_1,t_2)_\R,$$
where the last equality holds by \eqref{eq:hilbert_key}.
Hence, the sum on the left-hand side  in the 
lemma can be written as  
$$(t_1,t_2)_\R\sum_{\substack{(s,e) \in \N^2,\  
se=N_\b t \\ e \geq N_\b t/z } } \hspace{ 0.2cm } \prod_{p\mid e} 
(t_1,t_2)_{\Q_p},$$ which equals 
the right-hand side of the equation in the lemma. \end{proof}  
When  $\b t \in  (\Z\setminus\{0\})^2$
satisfies $N_\vt>z^2$,
then 
by \eqref{eq:hearingboots} the detector 
function can be written
\begin{align}
\updelta(\vt)&=
\sum_{\substack{
s\ \mathrm{ square}\textrm{-}\mathrm{free} \\
s \leq z  } } 
\ \ \prod_{p\mid s}\hs{t_1}{t_2}'_p
+
\sum_{\substack{
s\ \mathrm{ square}\textrm{-}\mathrm{free} \\
z<s< N_{\b t}/ z  } } 
\ \ \prod_{p\mid s}\hs{t_1}{t_2}'_p
+
\sum_{\substack{
s\ \mathrm{ square}\textrm{-}\mathrm{free} \\
s\geq  N_{\b t}/ z  } } 
\ \ \prod_{p\mid s}\hs{t_1}{t_2}'_p\nonumber
\\ &= 
\underbrace{
(1+(t_1,t_2)_\R) \sum_{\substack{
s\ \mathrm{ square}\textrm{-}\mathrm{free} \\
s \leq z  } } 
\ \ \prod_{p\mid s}\hs{t_1}{t_2}'_p
}_{\text{Deterministic}}
+
\underbrace{\sum_{\substack{
s\ \mathrm{ square}\textrm{-}\mathrm{free} \\
z<s< N_{\b t}/ z  } } 
\ \ \prod_{p\mid s}\hs{t_1}{t_2}'_p
}_{\text{Random}}
,\label{eq:delta_splitting}\end{align}where the second equality 
comes from Lemma~\ref{lem:symmetry}.
The parameter $z$ will later be 
chosen to go to infinity 
with the main asymptotic parameter $H$,
sufficiently slowly to ensure that pairs $\vt$ with $N_\vt\leq z^2$ are negligible.
The `random' part 
can be interpreted as a sum of 
$\pm 1$-terms with essentially random signs as $z\to\infty$, 
corresponding to the component of $\updelta$
in which the terms nearly cancel each other.
The `deterministic' part  records the influence of $\R$
and the small primes $p\leq z$.

\begin{definition} \label{def:dddddd} 
Let $z\geq 1$. For $\b t=(t_1,t_2) \in \Z^2$ we define
\begin{align*}
 \ddet(\vt)&:= (1+
 (t_1,t_2)'_\infty)\sum_{\substack{ 
s \textrm{ square-free}\\   
 s\leq z  } }   \prod_{p\mid s }(t_1,t_2)'_p,
\\  \dran(\vt)&:= \updelta(\vt)-\ddet(\vt).
\end{align*} 
\end{definition}
In particular, if $t_1t_2=0$ then 
$ \ddet(t_1,t_2)=\updelta(t_1,t_2)=1$ and 
$ \dran(t_1,t_2)=0$. 
We shall show that 
certain averages of $\dran$  
are small
using Heath-Brown's large sieve 
inequality~\cite{MR1347489}
in \S\ref{sec:siegel_walfisz}. An example of 
the kind of averages we are interested in is given by
$$\sum_{\substack{ s_1,s_2,t_1,t_2,t_3,r_1 \in \mathbb N
\\1\leq s_1,s_2,t_1,t_2,t_3,r_1 \leq B}}
\dran\l(  s_1 s_2  r_1 ,t_1t_2t_3
r_1 \r),$$ which is relevant to 
conic bundles
\eqref{eq:conic_bundle_eqn_mult} with $(m_1,m_2,m_3)=(2,3,1)$.
Given any real numbers 
$x_1,\ldots ,x_{m_1},y_1,\ldots, y_{m_2},
z_1,\ldots ,z_{m_3}\geq 1$ we denote   
\begin{equation}\label{def:casanatensenotation}
\c X= \prod_{i=1}^{m_1} x_i, \ \ 
\c Y= \prod_{i=1}^{m_2} y_i, \ \ 
\c Z= \prod_{i=1}^{m_3} z_i.
\end{equation}
The general case is:  
 \begin{theorem}[Randomness law for the
   analytic
   Hilbert symbol]
 \label{thm:pergolesiguglielmo}Let   $m_1,m_2>0 $ 
and $m_3\geq 0 $ be arbitrary  integers.
Fix any $\epsilon>0$ and $\sigma_1,\sigma_2\in \{-1,1\}$.
Assume that $ a:\N^{m_1} \to \mathbb C$, $b:\N^{m_2} \to  \mathbb C$ and
$c:\N^{m_3} \to  \mathbb C$ are arbitrary functions bounded by $1$ in modulus.
For any   $x_1,\ldots ,x_{m_1},y_1,\ldots, y_{m_2},z_1,\ldots ,z_{m_3},
z\geq 1$  we have 
\begin{align*} 
&\sum_{\substack{ \forall i, 1\leq s_i \leq x_i\\
\forall i, 1\leq t_i \leq y_i\\\forall i, 1\leq r_i \leq z_i}}
\dran\bigg(\sigma_1  \prod_{i=1}^{m_1} s_i 
\prod_{i=1}^{m_3} r_i ,\sigma_2 \prod_{i=1}^{m_2} t_i 
\prod_{i=1}^{m_3} r_i \bigg) a(\b s ) b(\b t ) c(\b r)
\\ &  \ll 
  (  \c X\c Y\c Z)^{1+\epsilon}
  \l( \frac{1}{z^{1/9} }+ 
  \frac{z^{1/9} }{\sqrt{\min\{
  \c X,\c Y,\c Z
  \} }} 
  + \frac{z}{\sqrt{
    \c X \c Y \c Z  
  }}
    + \frac{\one_{m_3=0}z^{4/9}}{\min \{  \c X,\c Y\}}
    \r)   ,\end{align*}
  where the implied constant depends only  on $m_1,m_2,m_3$ and $\epsilon$, and 
$\c Z$  is to be ignored 
in case $m_3=0$.
  \end{theorem}  This 
  result can be interpreted as saying that 
  $\dran$ is `orthogonal' to all 
  products of independent
  bounded sequences.
  Indeed, as the trivial bound is 
  $ \ll( \prod_i x_i \prod_i y_i
  \prod_i z_i )^{1+\epsilon}$, the theorem 
  gives a non-trivial saving 
  when $z$    grows like a small power
 of the $x_i,y_i,z_i$. We shall  feed the result into 
 a version of Theorem~\ref{thm:mccoytyner}
 by taking $a,b,c$  to be essentially 
 indicator functions of arithmetic progressions.
  \subsection{Quantitative Hasse principle results}\label{s:quanthasprinc} 
 The main idea of the proof of Theorem \ref{thm:main}  and Theorem \ref{thm:chebychev}
is to  set up a sum $S_\b F$ that essentially 
counts the points $\vt \in \P^1(\Q)$ for which the fibre $(X_\vF)_\vt$ has rational points.
Recall \eqref{eq:conic_bundle_eqn_mult}. For $x\geq 1$ and $\vF\in\cFZ(H)$, we define  \begin{equation}\label{def:Shaf_f(x)007007}
S_\b F(x):= \sum_{\substack{\b n \in \Z^2 \cap x\c B \\ \gcd(n_1,n_2)=1}}
\updelta(\Phi_1(\b n), \Phi_2(\b n)),  \end{equation}
where   \begin{equation}\label{eq:L_growth_bound_mult}
 \cB:=  \l([-1,1]\smallsetminus 
(-1/\log L, 1/\log L)\r)^2
\hspace{1cm} \text{ with } \hspace{1cm} 
L:=  \sqrt{\log H},
\end{equation} 
and 
\begin{equation}
\label{eq:def_phi_i}
\Phi_1 :=\prod_{j=1}^{m_1} F_{1j} \prod_{h=1}^{m_3} F_{3h}, 
\hspace{1.5cm}  
\Phi_2 :=\prod_{j=1}^{m_2} F_{2j} \prod_{h=1}^{m_3} F_{3h}
.\end{equation} 
By Lemma \ref{lem:detectorcoffee},
if $S_\vF(x)>0$ then there is a value of $\vn=(n_1,n_2)$ such that the conic $\Phi_1(\vn)x^2+\Phi_2(\vn)y^2=z^2$ has a rational point. If $\prod_{h=1}^{m_3}F_{3h}(\vn)\neq 0$, then this conic is isomorphic to the fibre $(X_\vF)_{(n_1:n_2)}$. Otherwise, the fibre $(X_\vF)_{(n_1:n_2)}$ is a degenerate conic. In both cases, the fibre, and thus $X_\vF$ has a rational point.  

One cannot show that $S_\b F(x)>0$ for $100\%$ of $\b F$,
because for a positive proportion of $\b F$
there is no $\Q$-point in 
\eqref{eq:conic_bundle_eqn_mult}. The plan is to show 
that for $100\%$ of $\b F$ the counting function 
 $S_\b F(x)$ is close to a product of local densities that is positive 
 and not too small
 if $X_\vF$ is everywhere locally soluble.
For primes $p$, these densities are 
\begin{align*}
\omega_p(\vF) :=& \left(1-\frac{1}{p^2}\right)^{-1}\int_{\ZZ_p^2\smallsetminus p\ZZ_p^2}1+\hs{\Phi_1(\vt)}{\Phi_2(\vt)}_p'\mathrm{d} \vt \\=& 1+\left(1-\frac{1}{p^2}\right)^{-1}\int_{\ZZ_p^2\smallsetminus p\ZZ_p^2}\hs{\Phi_1(\vt)}{\Phi_2(\vt)}_p'\mathrm{d} \vt.
\end{align*}
Moreover, let \begin{equation}\label{eq:def_omega_infty_mult}
  \omega_\infty(\vF) := \int_{\cB}1+\hs{\Phi_1(\vt)}{\Phi_2(\vt)}'_\infty\mathrm d\vt.
\end{equation} For notational convenience, 
we denote the truncated 
product over 
places including $\infty$ by
\begin{equation}\label{eq:def_Sing_mult}
  \Sing(\vF):=\frac{1}{\zeta(2)}
  \omega_\infty(\vF)\prod_{p\leq \sqrt{\log H}}\omega_p(\vF).  
\end{equation} 

 \begin{theorem}\label{thm:L^2theorem} Fix 
$\alpha\in (0,1/100),\ \beta\in (0,1)$ and 
assume that $H^\alpha\leq  x^{d} \leq H^{1/100}$.
Then$$\frac{1}{|\c F_\Z(H)|} \sum_{\b F \in \c F_\Z(H)}
| S_\b F(x)-\mathfrak S(\b F)x^2|^2 
\ll  \frac{x^4}{(\log H)^{\beta/2}}
,$$ where the implied constant depends only on 
$\alpha$, $\beta$, $m_1,m_2,m_3$ and  the $d_{i j}$.
\end{theorem}  

Theorem~\ref{thm:L^2theorem} is 
the main analytic result of this paper. 
Theorems \ref{thm:main}-\ref{thm:chebychev}
will be deduced from it in \S\ref{sec:chebyshev}.
The proof of Theorem~\ref{thm:L^2theorem} is presented in \S\ref{sec:dispersion}. 
The main idea is to use the decomposition 
$\updelta = \ddet + \dran$ to split $S_\b F$ into 
two sums. The $\dran$-part is handled using 
Corollary \ref{cor:heat:kernel} (a version of Theorem~\ref{thm:mccoytyner}) 
and Theorem \ref{thm:pergolesiguglielmo}, while the $\ddet$-part is treated through a level-lowering process. This
level-lowering method appears to be new in the context 
of dispersion arguments. It provides a relatively short
and uniform approach to all factorizations 
in~\eqref{eq:conic_bundle_eqn_mult}, and relies crucially 
on the fact that the modified Hilbert symbol 
$(\cdot,\cdot)_p'$ averages to zero.
A more detailed overview of the proof of 
Theorem~\ref{thm:L^2theorem} is given in  
\S\ref{s:sketching_ideas}.

\begin{acknowledgements}The core
of this research took place when the authors stayed at the 
Max Planck Institute in Bonn during April 2023 and 
April 2025, and when C.F.~visited E.S.~at the University 
of Glasgow in June 2024; we wish to acknowledge 
their support and hospitality.
C.F.~was supported by EPSRC grant EP/T01170X/2.
Finally, we are grateful to Jean-Louis Colliot-Thélène for his comments   on an earlier draft of this manuscript, which  improved the exposition.
\end{acknowledgements}

\section{Summability kernels}\label{s:summability}
The primary result of this section is 
Theorem \ref{thm:generalkernels}, 
a special case of which is
Theorem \ref{thm:mccoytyner}. 
Our main objective is to  
develop second-moment estimates for sums over random binary forms of multivariate functions with zero average, requiring only that the functions be sufficiently equidistributed in residue classes to small moduli. 
The   challenge in achieving this
using the circle method
lies in handling the 
minor arcs. 
These are usually treated using specific arithmetic information about the function under consideration, 
e.g. provided by combinatorial decompositions in case of the von Mangoldt or M\"obius functions. 
To address the lack of such specific information 
in our setup,
we introduce the idea 
that by employing positive summability kernels from 
Fourier analysis, the contribution of the minor arcs can 
be bounded directly.

We review the necessary definitions and terminology 
about  kernels in \S\ref{ss:kernelsss}, where we also 
state Theorem \ref{thm:generalkernels}.
Its proof is given in \S\S\ref{s:enter the dragon}-\ref{s:beginendpproof}. By specializing to the case of heat kernels, we shall obtain
Corollary \ref{cor:heat:kernel}, which is stated and proved 
in \S\ref{s:specialheat!}. We finish this section with Corollary \ref{cor:chiocolato}, a special case of 
Corollary \ref{cor:heat:kernel}
which is simpler to use.

\subsection{Kernels}\label{ss:kernelsss}
We recall some material from
Zygmund's book~\cite[\S 3.2]{MR236587}.
We normalise the Haar measure on $\mathbb T=\R/2\pi\Z$ so that $\mathbb T$ has measure $2\pi$. Hence, we will sometimes identify $\mathbb T$ with the interval $[0,2\pi)$.
\begin{definition}
\label{s:stein}Assume that for   
$H\geq 1$
we are given integrable functions 
$K_H:\mathbb T\to \mathbb [0,\infty)$. The functions $K_H$
are  called  \emph{positive summability kernels} if
\begin{itemize}
\item (Normalisation) For all $H$,
\begin{equation}\label{eq:normalise}
\frac{1}{2\pi} \int_{\mathbb T} 
K_H(\alpha)\mathrm d \alpha =1.  \end{equation} 
\item ($L^1$-concentration)
For every  $0<\delta <\pi$,
\begin{equation}\label{eq:L1concentr}
T_H(\delta):=\int_{\substack{ \|\alpha\|>\delta}}  K_H(\alpha)
\mathrm d \alpha \to 0  \ \ \textrm{as } H\to \infty,\end{equation}
 where $\|\cdot \|$ denotes the distance from $0$ in $\mathbb T$.
\end{itemize} 
\end{definition}  
We also require the Fourier coefficients  
$$\widehat{K}_H(n):=\frac{1}{2\pi}\int_{\mathbb T}K_H(\alpha)e^{-in\alpha}\mathrm d\alpha$$
of $K_H$ to be non-negative real numbers. More precisely, we ask that there exists $c_0>0$      
such that for all $H$ and $n\in \ZZ$  one has 
\begin{equation}\label{eq:positivfouriertransf} 
\widehat{K}_H(n) \geq c_0 \mathds 1_{[-H,H]}(n)\geq 0.
\end{equation}  
Moreover,
we assume  explicit decay of  Fourier coefficients, i.e.,  for fixed  
$\beta_0>0$, $\beta>1$,  
\begin{equation}\label{eq:decaycoeffi} 
\widehat{K}_H(n)\leq \beta_0 
\l(\frac{H}{1+|n|}\r)^\beta . \end{equation}
Assuming, in addition,
that   $K_H$ is continuous,~\eqref{eq:decaycoeffi}
implies that for all $\alpha, H$ one has 
\begin{equation}\label{eq:fourierinversion}K_H(\alpha)=
\sum_{n \in \Z }  \widehat{K}_H(n) \mathrm e ^{i\alpha n }.
\end{equation} 
We observe that \eqref{eq:normalise} and the positivity of $K_H$ imply for all $n\in\ZZ$ that
 \begin{equation}\label{eq:KHhat_bound}
\widehat{K}_H(n) \leq 1.
\end{equation}
Hence, by~\eqref{eq:fourierinversion} and \eqref{eq:decaycoeffi} we get 
\begin{equation}\label{eq:decaycoeffinal} 
K_H(\alpha) \leq  
\sum_{n\in \Z} \widehat{K}_H(n)
\ll \sum_{|n|\leq H } 1+
\sum_{|n|>H}\l(\frac{H}{|n|}\r)^\beta \ll H, 
\end{equation}  
with implicit constants depending only on $\beta_0,\beta$.

For any $m\in \N$,
$f:\Z^m\to\mathbb C$, $\b q\in (\Z\setminus \{0\})^m$
and $x_1,\ldots,x_m \geq 1 $   let 
\begin{equation}\label{def:porporapassio} 
E_f(\b x ;\b q ) :=
\sup_{\b b\in \prod_{k=1}^m (\Z/q_k\Z)}  
\sup_{\substack{ \b v \in  \R^m \\ \forall k,
|v_k| \leq x_k } } \l| 
\sum_{\substack{ \b t \in \Z^m \\
-x_k \leq t_k \leq v_k    \forall k } }
f(\b t ) \mathrm{e}^{2\pi i \sum_{k=1}^m \frac{b_k t_k}{q_k}}\r|.\end{equation} 

 We introduce 
a standard assumption that prevents 
one value of   $f$ from dominating its average:
fix any $B>0$, $C\geq 0$ and assume 
\begin{equation} \label{def:weaksigwalfs}|f(\b n)| \leq B \prod_{\substack{1\leq j \leq m \\ n_j \neq 0 }} \tau(|n_j|)^C 
\ \ \ \textrm{ for all } \b n\in \Z^m,\end{equation}
where $\tau$ is the divisor function and an
empty product is defined to be $1$. 

Next, we fix any $d_1,\ldots, d_m\in \N$ and define 
\begin{equation}\label{pasquini} 
\gamma_0:=\sum_{j=1}^m d_j(1+C(d_j+2) ),\quad
\gamma_1:=\sum_{j=1}^m 2^{2C(d_j+2)+1}.\end{equation}   
In this section, we change notation slightly and denote by $\cFZ(H)$
the set of   
 vectors of integer forms $\b F=(F_{i})$ in $\Z[t_1,t_2]^m$
 such that each $F_i$ has degree $d_i$
 and all of its 
 coefficients lie  in $[-H,H]$. Moreover, we write
\begin{equation*}
 d:=d_1+\cdots+d_m\quad\text{ and }\quad \c D=\max\{d_1,\ldots,d_m\}.
\end{equation*}
 \begin{theorem}\label{thm:generalkernels} 
 Let $m,d_1,\ldots,d_m\in \N$, $B,\beta_0,c_0>0$, $C\geq 0$ and $\beta>1$, and let $K_H$ be positive summability Kernels satisfying \eqref{eq:positivfouriertransf}--\eqref{eq:fourierinversion}. 
 For   any $f:\Z^m\to \mathbb C$ 
 satisfying~\eqref{def:weaksigwalfs},
 any $\delta\in (0,1)$, any    
 $1\leq \xi_0\leq x\leq H$ and any function
 $a:\Z^2\to \{z\in \mathbb C:|z|\leq 1\}$, 
we have  
\begin{align*}
&\frac{1}{\#\c F_\Z(H)} \sum_{\b F \in \c F_\Z(H)} \Bigg|\sum_{\substack{  
\b n\in \Z^2 \cap[-x,x]^2 }} a(\b n)  f(F_1(\b n), \ldots, F_m(\b n) ) 
\Bigg|^2  
\\ & \ll  \l\{
 x^{2d } 
 (\log H)^{m 2^{C+1}} T_H(\delta)  
 \  + \   
 \frac{ (\log H)^{ \gamma_1 } (\log x)^{  2^{2(1+2\gamma_0)} } }{\xi_0^{1/(2\c D)}} 
 \r\}\cdot x^4 \\ &+
\Bigg(\frac{ 
\max\l\{1, \delta \xi_0 H \r\}}{H}\Bigg)^{2m} 
\sum_{\substack{ \b n,\b l \in \Z^2, n_2 l_1\neq \pm n_1 l_2
\\ x/\xi_0^{1/2}<|n_i|,|l_i| \leq x }}  
 E_f(((1+d_j)x^{d_j}H)_{j=1}^m;
( (n_2 l_1)^{d_j}- (n_1 l_2)^{d_j})_{j=1}^m )^2 
,\end{align*}  
where the implied constant depends only on $m,d_1,\ldots,d_m,B,C,c_0,\beta,\beta_0$.
\end{theorem}
Hence, if 
the `tail' function $T_H$ is sufficiently close to $0$, 
$\xi_0$ is suitably large and $E_f$ is 
appropriately small then 
for most tuples $\b F$ the 
corresponding sum $\sum_\b n$ is 
$o(x^2)$. 
\begin{remark}
The error 
term involving $T_H$ comes from the minor arcs, 
the error term with 
$\xi_0^{-  1/(2\c D)  }$
comes from,
essentially,
the diagonal contribution when opening up the square in $\sum_\b F(\sum_\b n)^2$, and the error term involving $E_f$ comes 
from the major arcs.\end{remark}

We now start with the proof of Theorem \ref{thm:generalkernels}, which will be concluded in \S\ref{s:beginendpproof}. Throughout the proof, all implied constants are allowed to depend on the quantities stated in the theorem and nothing else, unless explicitly stated otherwise.

 \subsection{Opening the square}\label{s:enter the dragon}
We start the proof of Theorem \ref{thm:generalkernels} 
by letting \begin{equation}\label{def:III}  \mathtt{I} := \prod_{j=1}^m [-(d_j+1)x^{d_j} H,(d_j+1)x^{d_j}H]\end{equation} and noting that if $|n_1|,|n_2|\leq x$
and $\b F \in \c F_{\Z}(H)$ then 
 $(F_1(\b n ), \ldots, F_m(\b n)) \in \mathtt{I}$. 
Write $$F_j(t_1,t_2)=\sum_{k=0}^{d_j} c_{k,j} t_1^{k} t_2^{d_j-k}.$$
Each   $F_j$ has its coefficients in $[-H,H]$, hence, by~\eqref{eq:positivfouriertransf} the sum over $\b F$ in Theorem~\ref{thm:generalkernels} is  $$ \leq  c_0^{-d_1-\ldots-d_m}\sum_{\substack{ \b F \in \cFZ } }
\Big( \prod_{ \substack{  1\leq j \leq  m\\ 
0\leq k \leq  d_j } }  \widehat{K}_H(   c_{j,k})  \Big)
\Big|\sum_{\substack{ \b n \in \Z^2, \b F(\b n) \in \mathtt{I}
\\  |n_1|,|n_2| \leq x}} 
a(\b n ) f(\b F(\b n) )\Big|^2 
,$$ where $\c F_\Z $  denotes the set of   
 vectors of integer forms $\b F=(F_{i})$ in $\Z[t_1,t_2]^m$
 such that each $F_i$ has degree $d_i$, 
 but having no     restriction 
 on the size of coefficients.  
Opening up the square and inverting the order of summation
turns the sum over $\b F$ into  
$$ \sum_{\substack{ \b n,\b l \in \Z^2 \\ 
|n_i|,|l_i| \leq x}}  
\overline{a(\b n )}a(\b l )  
\sum_{\substack{ \b F \in \c F_\Z \\
\b F(\b n) , \b F(\b l) \in \mathtt{I} } }
\Big( \prod_{ \substack{  1\leq j \leq  m\\ 
0\leq k \leq  d_j } } \widehat{K}_H(   c_{j,k}) \Big) 
\overline{f(\b F(\b n) )}f(\b F(\b l) ) .$$
Here we note that the  
infinite
sum over $\vF$ converges absolutely by \eqref{eq:decaycoeffi}. 
Letting  $t_j=F_j(\b n)$ and $t'_j=F_j(\b l)$ we are led to 
$$ \sum_{\substack{ \b n,\b l \in \Z^2 
\\ |n_i|,|l_i| \leq x}}  
\overline{a(\b n )}a(\b l )
\sum_{ \b t, \b t'\in \Z^{m}\cap \mathtt{I}} 
\overline{f(\b t)}f(\b t ')
\sum_{\substack{ \b F \in \c F_\Z   } } 
\mathds 1_{\{\b t\}}(\b F(\b n) ) 
\mathds 1_{\{\b t'\}}(\b F(\b l) ) 
 \prod_{ \substack{  1\leq j \leq  m\\ 
0\leq k \leq  d_j } }\widehat{K}_H(   c_{j,k}),$$
where $\one_S$ denotes the indicator function of a set $S$.
Thus, we have shown:   
\begin{lemma}\label{lem:starting_point} 
In the setting of  Theorem \ref{thm:generalkernels}
we have 
\begin{equation}\label{eq:disprscircle}
\sum_{\substack{ \b F \in \c F_\Z(\b H)    } }
\Big|\sum_{\substack{ \b n \in \Z^2 \\  |n_1|,|n_2| \leq x}} 
a(\b n ) f(\b F(\b n) )\Big|^2
\ll \sum_{\substack{ \b n,\b l \in \Z^2 \\  |n_i|,|l_i| \leq x}}  
\overline{a(\b n )}a(\b l )
\c J(\b n,\b l) 
 ,\end{equation} where  $$\c J(\b n,\b l) := 
\sum_{ \b t, \b t'\in \Z^{m}\cap \mathtt{I}}
\overline{f(\b t)}f(\b t ')
 \prod_{  j=1}^m \Big( 
\sum_{\substack{ F_j \in \Z[X,Y] \\ \deg(F_j)=d_j  } } 
\mathds 1_{\{(t_j,t'_j)\}}(F_j(\b n),F_j(\b l) ) 
 \prod_{  k=0 }^{d_j}   \widehat{K}_H(   c_{j,k})
\Big).$$
Here, each $F_j$ runs through binary forms of degree $d_j$.
\end{lemma}

 \subsection{Small determinant}

Here we deal with those values of $\b n, \b l$ on the right-hand side in Lemma \ref{lem:starting_point} for which $|(n_2l_1)^{d_j}-(n_1l_2)^{d_j}|$ is small for some $j\in\{1,\ldots,m\}$.
  Let $\tau':\Z\to [1,\infty)$ 
be defined by 
$$
\tau'(0):=1 \quad
\text{ and }\quad \tau'(n):=\tau(|n|) \text{ for } n\neq 0.
$$ 
\begin{lemma} \label{lem:vivaldi23df4}
Fix   $K,K_1\geq 0$ and $M\in\NN$.
Then for all $a\in \Z$,  $q\in \Z\setminus\{0\}$  
and  $z\in\RR$, $w\geq 1$ satisfying 
$|q|(|z|+w)+|a| \leq K_1 w^M $ we have  
$$\sum_{z< n\leq z+w} \tau'(qn+a)^{K} \ll 
 \tau(|q|)^{1+K M}
 w (\log  w )^{2^{K M}}
,$$  
where the implied constant depends at most on $K,K_1$ and $M$. \end{lemma}
\begin{proof}
We use Landreau's inequality~\cite{MR998633}, which shows for every $n\in\NN$ that 
$$ \tau(n)^K \leq M^{M(M-1)K} 
\sum_{\substack { \delta \in \N,\  \delta \mid n 
\\  \delta \leq  n^{1/M}  }} \tau(\delta)^{KM} 
.$$ In particular, for all $n\in \Z$ we have 
$$ \tau'(n)^K \ll \sum_{\substack { \delta \in \N,\  \delta \mid n 
\\  \delta \leq  |n|^{1/M}+1  }} \tau(\delta)^{KM} 
.$$
Hence, for the sum in the lemma we obtain the bound 
$$ \ll
\sum_{ \delta \leq  (|q|(|z|+w)+|a|) ^{1/M}+1 } \tau(\delta)^{KM} 
\sum_{\substack {  z< n\leq z+w 
\\ qn\equiv -a \md{\delta}   } } 
1.$$ The sum over $n$ is  $\ll \frac{w\gcd(\delta,q)}{\delta} + 1 $. Using  our assumption 
$(|q|(|z|+w)+|a|) \ll w^M $, we see that 
$$\delta \leq  (|q|(|z|+w)+|a|) ^{1/M}+1
 \ll w
.$$ Thus, the sum over $n$ is  $\ll 
\frac{w\gcd(\delta,q)}{\delta}$, leading to the 
overall bound  
$$ \ll w
\sum_{ \delta \ll w} \tau(\delta)^{K M} \frac{\gcd(\delta,q)}  {\delta}
.$$  We use the identity 
$\gcd(\delta, q)= \sum_{m  } \phi(m)$, 
where the sum is over $m$ dividing both $\delta $
and $q$. Letting $b=\delta/m$ we infer that
the bound is  \[\ll w
\sum_{m \mid q } \phi(m)
\frac{\tau( m)^{KM}}{m} 
\sum_{ b  \ll  w  } 
\frac{\tau(b )^{KM}  }{b } 
\ll  \tau( q)^{1+KM} w(\log  w )^{2^{KM}}
.\qedhere\] \end{proof}    
\begin{lemma}\label{lem:pulltheplug}
Fix   $K,K_1\geq 0$ and $M\in\NN$.
Then for all $H\geq 1, 
a,a'\in \Z$ and $ q,q'\in \Z\setminus\{0\}$  
 satisfying 
$\max\{|q|H^{2}+|a|, |q'|H^{2}+|a'| \}\leq K_1 H^M $ 
we have$$ 
\sum_{\substack{n\in\Z \\ H<|n|\leq H^{2} }}
\widehat{K}_H(n)\tau'(n q+a)^K 
\tau'(nq'+a')^K\ll 
(\tau(|q|) \tau(|q'|))^{1/2+KM} 
H(\log H)^{2^{2KM}} 
 ,$$  
 where the implied constant depends only on $K,K_1,M,\beta,\beta_0$.
\end{lemma}\begin{proof}
By   \eqref{eq:decaycoeffi}   we obtain the bound
$$ \ll   \sum_{1\leq t \leq H } 
t^{-\beta }
\sum_{tH\leq |n| \leq (t+1 )H}
 \tau'(nq+a)^K 
  \tau'(n q'+a')^K.$$
By Cauchy's inequality the inner sum over $n$ is 
$\leq (T_1 T_2)^{1/2}$ where 
$$ T_1=\sum_{tH\leq |n| \leq (t+1 )H}
 \tau'(nq+a)^{2K}\quad\text{ and }\quad
  T_2=\sum_{tH\leq |n| \leq (t+1 )H}
  \tau'(nq'+a')^{2K}
  .$$The contribution of positive $n$
  in $T_1$ and $T_2$
can be bounded by Lemma~\ref{lem:vivaldi23df4} with
parameters $z=tH$, $w=H$,
while the contribution of negative $n$ is 
treated analogously. We get
 $$  T_1 T_2  \ll  
 (\tau(|q|) \tau(|q'|))^{1+2KM} 
 \l(H(\log H)^{2^{2KM}}\r)^2,$$ uniformly in $t$.
This suffices for the proof.\end{proof}  
\begin{lemma}\label{lem:cimento2} 
With the setup of Theorem \ref{thm:generalkernels}, for all  $\b n, \b l $  as in 
 Lemma~\ref{lem:starting_point}   we have  
$$\c I(\b n,\b l ) \ll \bigg(
\sum_{i,j\in \{1,2\}}
(\tau'(n_i) \tau'(l_j))^{\gamma_0 } \bigg)
 H^{d+m}(\log H)^{ \gamma_1  }. 
 $$ \end{lemma}
\begin{proof} 
By~\eqref{def:weaksigwalfs} 
 we obtain the bound  
\begin{align} \c I(\b n,\b l )
\ll
&\sum_{\substack{ \b F \in \c F_\Z\\ 
\b F(\b n ), \b F(\b l ) \in \mathtt{I}   } }
\Big( \prod_{ \substack{  1\leq j \leq  m \\ 
0\leq k \leq  d_j } } \widehat{K}_H(   c_{j,k}) \Big)
\prod_{\substack{j=1}}^m 
\tau'(F_j(\b n ))^C 
\prod_{\substack{j=1}}^m  
 \tau'(F_j(\b l ))^C  \nonumber
 \\=
 & \prod_{j=1}^m
 \sum_{   \b c \in \Z^{1+d_j}  }
\Big( \prod_{ \substack{  
0\leq k \leq  d_j } } \widehat{K}_H(   c_{k}) \Big)
 \tau'\l( \sum_{k=0}^{d_j} c_k n_1^k n_2^{d_j-k} \r)^C 
  \tau'\l( \sum_{k=0}^{d_j} c_k  l_1^k  l_2^{d_j-k} \r)^C \label{eq:j_sum_bound}
,\end{align} where the sum over $\b c $ is subject to the
additional
condition 
that the  
arguments of $\tau'(\cdot)$
have modulus
at most  
$(1+d_j) H x^{d_j}$.

For the remainder of this proof, we distinguish between a few cases depending on $\b n,\b l$:
\begin{multicols}{3}
\begin{enumerate}[(a)]
\item $n_2l_2\neq 0$,
\item $n_1l_1\neq 0$,
\item $n_1l_2\neq 0$, $n_2=l_1=0$,
\item $n_2l_1\neq 0$, $n_1=l_2=0$,
\item $\b n = \vzero$, $l_2\neq 0$,
\item $\b n=\vzero$, $l_1\neq 0$,
\item $\b l=\vzero$, $n_2\neq 0$,
\item $\b l=\vzero$, $n_1\neq 0$,
\item $\b n=\b l=\vzero$.
\end{enumerate}
\end{multicols}

In case (a), we  
write  the sum over $\b c$ in \eqref{eq:j_sum_bound} as
\begin{equation} \label{eq:pergopergo2} 
\sum_{   (c_1,\ldots, c_{d_j}) \in \Z^{d_j}  }
\Big( \prod_{ \substack{  
 1\leq k \leq  d_j } } \widehat{K}_H(   c_{k}) \Big)
 \sum_{ \substack{ c_0\in \Z \\\eqref{eq:conditionsize}} }    \widehat{K}_H(   c_0)  
 \tau'(  c_0 n_2^{d_j}+N)^C 
  \tau'(  c_0   l_2^{d_j}+N')^C  ,
  \end{equation} where  
  $ N:=\sum_{k=1}^{d_j} c_k n_1^k n_2^{d_j-k}$,
 $N':=\sum_{k=1}^{d_j} c_k l_1^k l_2^{d_j-k}$ and the sum over $c_0$ is subject to the additional conditions
\begin{equation} \label{eq:conditionsize} 
|c_0 n_2^{d_j}+N  | \leq  (1+d_j) H x^{d_j}, \ \ \ 
|  c_0   l_2^{d_j}+N'|\leq (1+d_j) H x^{d_j}.
\end{equation}
 By \eqref{eq:KHhat_bound},
 the sum over $c_0$ is $\ll \Xi_1+\Xi_2+\Xi_3$, where 
 \begin{align*}
\Xi_1&:=\sum_{\substack{ |c_0|\leq H \\ \eqref{eq:conditionsize}} }   
 \tau'(   c_0 n_2^{d_j}+N)^C 
  \tau'(   c_0   l_2^{d_j}+N')^C,
  \\ \Xi_2 &:=
  \sum_{\substack{ H<|c_0|\leq  H^{2} \\ \eqref{eq:conditionsize}} }   
   \widehat{K}_H(   c_0) 
  \tau'(   c_0 n_2^{d_j}+N)^C 
  \tau'(   c_0   l_2^{d_j}+N' )^C, 
 \\ \Xi_3&:=
  \sum_{\substack{ |c_0|> H^{2} \\ \eqref{eq:conditionsize}} }   
   \widehat{K}_H(   c_0) 
  \tau'(   c_0 n_2^{d_j}+N)^C 
  \tau'(   c_0   l_2^{d_j}+N' )^C .
\end{align*}
Using Cauchy's inequality we obtain
\begin{align*}
\Xi_1^2 &\leq \sum_{ |c_0|\leq H }   
 \tau'(   c_0 n_2^{d_j}+N)^{2C } \sum_{ |c_0|\leq H }    
  \tau'(   c_0   l_2^{d_j}+N')^{2C} \\ &\ll
  \tau(|n_2|^{d_j})^{1+2C(d_j+1)} 
    \tau(|l_2|^{d_j})^{1+2C(d_j+1)} 
   \l (H (\log H)^{2^{2C(d_j+1)}} \r)^2
\end{align*}  
due to Lemma \ref{lem:vivaldi23df4} applied with 
$$z=-H,\ w=2H,\ q=n_2^{d_j}\ (\textrm{or } l_2^{d_j}),\  
a=N\ (\textrm{or } N'),\  M=1+d_j.$$ 
We used the bound 
$|N|,|N'| \leq (1+d_j) H x^{d_j}+O(|c_0| x^{d_j})
\ll H x^{d_j}$ by  \eqref{eq:conditionsize}.
Before using Lemma \ref{lem:pulltheplug}  to bound 
$\Xi_2$ we note that \eqref{eq:conditionsize} implies 
$ |N|\ll |c_0n_2^{d_j}|+ Hx^{d_j}\ll H^{2+d_j}$, 
hence
$$|n_2^{d_j}| H^{2}+ |N|  \ll  H^{2+d_j} .
$$Thus, Lemma \ref{lem:pulltheplug} 
with $q=n_2^{d_j}, q'=l_2^{d_j},
a=N, a'=N', M=d_j+2$, gives 
$$\Xi_2\ll 
(\tau(|n_2|^{d_j}) \tau(|l_2|^{d_j}))^{1/2+C(d_j+2)} 
H(\log H)^{2^{2C(d_j +2)}} .$$ Lastly, 
by \eqref{eq:decaycoeffi}  and the bound $\tau'(n)\ll_\epsilon |n|^\epsilon $, valid for all
  $\epsilon>0$ 
  and $n\neq 0$,
we infer that   
$$\Xi_3\ll_\epsilon H^\beta\sum_{\substack{  |c_0|\geq  H^{2}  }} 
|c_0|^{-\beta }   H^{\epsilon} \ll
H^{\epsilon+\beta-2(\beta -1)} \ll H,$$ 
as can be seen by taking  
$\epsilon<\beta-1$. 
Bringing together the bounds for each $\Xi_i$ we 
deduce that the sum over $c_0$ in \eqref{eq:pergopergo2} 
is $$\ll \Xi_1+\Xi_2+\Xi_3 \ll   
(\tau(|n_2|^{d_j}) \tau(|l_2|^{d_j}))^{1/2+C(d_j+2)} 
H(\log H)^{2^{2C(d_j +2)}}.$$
This bound is independent of $(c_1,\ldots,d_{d_j})$, hence using \eqref{eq:decaycoeffinal} the outer sum in \eqref{eq:pergopergo2} adds a factor $H^{d_{j}}$. Taking the product over $j$ in \eqref{eq:j_sum_bound} now suffices to prove the lemma in case (a). Case (b) is analogous. 

In case (c), we proceed similarly: instead of \eqref{eq:pergopergo2}, we 
write the sum over $\b c$ in \eqref{eq:j_sum_bound} as 
\begin{equation} \label{eq:j_sum_bound_c} 
\left(\sum_{   c_1,\ldots, c_{d_j-1} \in \Z  }
\prod_{ \substack{  
 1\leq k \leq  d_j-1} } \widehat{K}_H(   c_{k}) \right)
 \left(\sum_{\substack{c_{0}\in \Z\\ \eqref{eq:conditionsize_2}}} 
 \widehat{K}_H(c_0)\tau'(  c_0 l_2^{d_j})^C\right)\left(\sum_{\substack{c_{d_j}\in \Z\\\eqref{eq:conditionsize_2}}} \widehat{K}_H(c_{d_j})\tau'(c_{d_j}n_1^{d_j})^C\right),
  \end{equation}  
  with the sums over $c_0$ and $d_{d_j}$ subject to the conditions
  \begin{equation}\label{eq:conditionsize_2}
    |c_0 l_2^{d_j}|\ll (1+d_j)Hx^{d_j}\quad\text{ and }\quad |c_d n_1^{d_j}|\ll (1+d_j)Hx^{d_j}.
  \end{equation}
  
Here, the sum over $c_1,\ldots,c_{d_{j}-1}$ is $\ll H^{d_j-1}$ by \eqref{eq:decaycoeffinal}. Similarly as above, we bound the sum over $c_0$ by $\ll \Xi_1+\Xi_2+\Xi_3$, where we formally take $n_1=n_2=l_1=0$, so 
that $N=N'=0$, in particular $\tau'(c_0n_2^{d_j}+N)=1$,
and the conditions \eqref{eq:conditionsize} become \eqref{eq:conditionsize_2}. Forgetting these conditions and
using Lemma \ref{lem:vivaldi23df4}, Lemma \ref{lem:pulltheplug} (with $q=q'=l_2^{d_j}$, $a=a'=0$, $K=C/2$, $M=d_j+2$), and the bound $\tau'(n)\ll_\epsilon |n|^\epsilon$ as above, we estimate
\begin{align*}
    \Xi_1&\ll \tau(|l_2|^{d_j})^{1+C(d_j+1)}H(\log H)^{2^{C(d_j+1)}},\\
    \Xi_2&\ll \tau(|l_2|^{d_j})^{1+C(d_j+2)}H(\log H)^{2^{C(d_j+2)}},\\
    \Xi_3&\ll H. 
\end{align*}
Hence, the sum over $c_0$ in \eqref{eq:j_sum_bound_c} is $\ll \tau(|l_2|^{d_j})^{1+C(d_j+2)}H(\log H)^{2^{C(d_j+2)}}$, and an analogous bound with $l_2$ replaced by $n_1$ holds for the sum over $c_{d_j}$. Bringing these bounds together and taking the product over $j$ in \eqref{eq:j_sum_bound} shows the result in case (c). Case (d) is again analogous.

In case (e), we  write the sum over
$\b c$ in \eqref{eq:j_sum_bound} as \eqref{eq:pergopergo2},
where $\vn=0$ implies that $N=0$, so that $\tau'(c_0n_2^{d_j}+N)=1$ for all $c_0$. We may thus bound the sum over $c_0$ exactly as above in case (c), thus allowing us to estimate \eqref{eq:pergopergo2} by $\ll H^{d_j+1}\tau(|l_2|^{d_j})^{1+C(d_j+2)}(\log H)^{2^{C(d_j+2)}}$. Taking the product over $j$ in \eqref{eq:j_sum_bound} again yields a satisfactory bound. Cases (f), (g), (h) are analogous.

Finally, in case (i) all the terms with $\tau'$ in \eqref{eq:j_sum_bound} are equal to $1$, and hence by \eqref{eq:decaycoeffinal},
$$\c I(\b n,\b l ) \ll \prod_{j=1}^m \sum_{   \b c \in \Z^{1+d_j}  }
\prod_{ \substack{  
0\leq k \leq  d_j } } \widehat{K}_H(   c_{k})\ll H^{d+m}.$$
\end{proof}

Recall the notation $\c D=\max d_i$. 
\begin{lemma}\label{lem:cimento23fa} 
The contribution of $\b n, \b l $ that satisfy
\begin{equation}\label{eq:bounded_size_contrib}
    \min \{|n_1|,|n_2|,|l_1|,|l_2|\} \leq x/\xi_0^{1/2\c D}
\end{equation}
or
\begin{equation}\label{eq:bounded_determinant_contrib}
  |(n_2 l_1)^{d_j} -(n_1 l_2)^{d_j } |\leq   x^{2d_j } /\xi_0 
\end{equation}
for some $1\leq j\leq m$ to the right-hand side of~\eqref{eq:disprscircle} is
$$\ll H^{d+m}(\log H)^{ \gamma_1 } x^4 
 \frac{(\log x)^{  2^{2(1+2\gamma_0)} } }{\xi_0^{1/2\c D}}
.$$
\end{lemma}

\begin{proof}
By Lemma~\ref{lem:cimento2},
the terms with \eqref{eq:bounded_size_contrib}
contribute at most 
$$\ll  H^{d+m}(\log H)^{ \gamma_1  }  \sum_{i,j\in \{1,2\}}
\sum_{\substack{ |n_1|,|n_2|,|l_1|,|l_2| \leq x
\\ \eqref{eq:bounded_size_contrib}
}} 
(\tau'(n_i) \tau'(l_j))^{\gamma_0 }
\ll   H^{d+m}(\log H)^{ \gamma_1 } 
(\log x)^{  2^{1+\gamma_0} } 
\frac{x^4}{\xi_0^{1/(2\c D)}}.$$

Now we fix $1\leq j\leq m$ and consider the contribution of $\vn,\b l$ for which \eqref{eq:bounded_size_contrib} fails and \eqref{eq:bounded_determinant_contrib} holds. These cases satisfy $x/\xi_0^{1/(2d_j)} \leq x/\xi_0^{1/(2\c D)}\leq |n_1|,|n_2|,|l_1|,|l_2| \leq x$. 
Note that when $r\in \R$ and $n\in \N$
then the distance of $r$ from each of 
the points $\mathrm e^{2\pi i k/n}$, $1\leq k<n$, $k\neq n/2$ is strictly positive 
and bounded from below in terms of $n$ only. 
In particular,  $$| r^n -1 |=\prod_{k=1}^n  
|r-\mathrm e^{2\pi i k/n} |
\gg_n |r-1| |r+1|^{\chi(n)}, $$ where $\chi$ is the indicator of even integers. 
Therefore, when    $d_j\equiv 1 \md 2 $ we have 
$$ x^{2d_j}/\xi_0 \geq  |(n_2 l_1)^{d_j} -(n_1 l_2)^{d_j } |
\gg  | n_2 l_1  -n_1 l_2 |  (x/\xi_0^{1/(2d_j)})^{2(d_j-1)}
\Rightarrow  | n_2 l_1  -n_1 l_2 | \leq 
\frac{x^2}{\xi_0^{1/d_j}}.  $$
These cases contribute   $$\ll H^{d+m}(\log H)^{ \gamma_1  }   
\sum_{\substack{    0<|n_1|,|n_2|,|l_1|,|l_2| \leq x
\\  | n_2 l_1  -n_1 l_2 | \leq x^2/\xi_0^{1/d_j}}}\ 
 \sum_{i,j\in \{1,2\}} (\tau'(n_i) \tau'(l_j))^{\gamma_0 } 
.$$ Letting $a=n_2l_1$ and $b=n_1 l_2$,
the sum is  
$$\ll \sum_{\substack{0<|a|,|b|\leq x^2 \\ |b-a| \leq x^2/\xi_0^{1/d_j}}} 
(\tau'(a) \tau'(b))^{1+2\gamma_0 }\leq 
\l(
\sum_{\substack{0<|a|\leq x^2 \\ |b-a| \leq x^2/\xi_0^{1/d_j}}} 1
\r)^{1/2}\l(
\sum_{ 0<|a|,|b|\leq x^2  } 
(\tau'(a) \tau'(b))^{2(1+2\gamma_0)}
\r)^{1/2}
,$$  
which is  
$\ll \frac{x^2}{\xi_0^{1/(2d_j)}} x^2 (\log x)^{2^{2(1+2\gamma_0)}}$, which gives a sufficient overall bound.

When    $d_j\equiv 0 \md 2 $ we similarly obtain
 $   | (n_2 l_1)^2  -(n_1 l_2 )^2| \leq 
x^4/\xi_0^{2/d_j}$,  
and therefore  
 $   | n_2 l_1  - n_1 l_2 | \leq 
x^2/\xi_0^{1/d_j}$ or 
 $   | n_2 l_1  + n_1 l_2 | \leq 
x^2/\xi_0^{1/d_j}$. 
Both cases are treated as above.\end{proof}  

\subsection{Using the circle method identity}\label{s:introkernel}
We   write $$\mathds 1_{\{(t_j,t'_j)\}}(F_j(\b n),
F_j(\b l) ) = \frac{1}{(2\pi)^2} \int_{\mathbb T^2} 
\mathrm e^{ i \l( \alpha_j (F_j(\b n) -t_j) - \beta_j (F_j(\b l)-t'_j) \r) }
\mathrm d\alpha_j \mathrm d\beta_j ,$$ hence, by~\eqref{eq:fourierinversion} the function $\c J$ in \eqref{eq:disprscircle} equals 
\begin{equation}\label{eq:Jcirclemthd} \c J(\b n,\b l) = 
\frac{1}{(2\pi)^{2m} }\int_{\mathbb T^{2m}}  
\overline{S(\boldsymbol \alpha)}S(\boldsymbol \beta)  
 \prod_{  j=1}^m  \prod_{  k=0}^{d_j}  
  K_H(\alpha_j n_1^k n_2^{d_j-k} - \beta_j  l_1^k l_2^{d_j-k}  )
\mathrm d\boldsymbol \alpha \mathrm d\boldsymbol\beta, \end{equation}
where $$ S(\boldsymbol \alpha) := \sum_{ \b t\in   \Z^{m}\cap \mathtt{I}}
f(\b t) \mathrm e^{ i  \boldsymbol \alpha \cdot \b t } $$ and $ \boldsymbol \alpha \cdot \b t  $ stands for the standard inner product. Before proceeding let us use \eqref{def:weaksigwalfs}  to get 
\begin{equation}\label{eq:easybnd} |S(\boldsymbol \alpha)|
\ll \prod_{j=1}^m \sum_{|t|\leq (d_j+1)x^{d_j} H } \tau'(t)^C 
\ll  x^{d} H^m (\log H)^{m2^C}.\end{equation}

\subsection{Minor arcs}\label{s:minor} 
We   define the minor arcs not in the traditional sense 
but as the subset of $\mathbb T^{2m}$ where
some specific kernels
$K_H$ in~\eqref{eq:Jcirclemthd} 
assume a value away from their peak. 
Let $\delta\in(0,1)$ be as in the statement of Theorem \ref{thm:generalkernels}.
Recall that
$\|\cdot \|$ denotes the distance from $0$ in $\mathbb T$. 
We study the contribution towards~\eqref{eq:Jcirclemthd} of  
$\boldsymbol\alpha  ,\boldsymbol\beta $ for which there 
is   $1\leq h \leq m $ such that  \begin{equation}\label{def:minorarcs} 
\|\alpha_h  n_2^{d_h} - \beta_h   l_2^{d_h} \|  >  \delta 
\ \ \textrm{ or } \ \ \|\alpha_h n_1^{d_h} - \beta_h  l_1^{d_h}  \|  >  \delta 
.\end{equation} 
In order to do so, we need a simple auxiliary result.
\begin{lemma}\label{lem:KH_integral_transformation_new}
    Let  $A,B,C,D$  be integers with $AD\neq BC$ and $E\subset \mathbb T^{2}$ measurable. Then 
\begin{align*}
&\int_{ \mathbb{T}^2 } 
 K_H(A \alpha  - B \beta    )
  K_H(C \alpha   - D \beta    )
  \mathds 1_E(A \alpha  - B \beta,
  C \alpha   - D \beta )
  \mathrm d \alpha \mathrm d\beta = 
  \int_{ E }   K_H(\alpha)K_H(\beta)
  \mathrm d \alpha  \mathrm d \beta.\end{align*}
In particular, for $E=\mathbb{T}^2$, the result is equal to $4\pi^2$.
\end{lemma}

\begin{proof}
  As $AD-BC\neq 0$, the map $\Phi:(\alpha,\beta)\mapsto (A\alpha-B\beta,C\alpha-D\beta)$ is a surjective endomorphism of the compact group $\mathbb{T}^2$, and thus preserves the Haar measure. Hence, with $f(\alpha,\beta):= K_H(\alpha)K_H(\beta)\one_{E}(\alpha,\beta)$, the left-hand side is equal to
  \begin{equation*}
     \int_{\mathbb{T}^2}f(\Phi(\valpha))\mathrm{d}\valpha = \int_{\mathbb{T}^2}f(\valpha){\Phi_*}(\mathrm{d}\valpha) = \int_{\mathbb{T}^2}f(\valpha)\mathrm{d}\valpha = \int_E K_H(\alpha)K_H(\beta)\mathrm d\alpha\mathrm d\beta.\qedhere
  \end{equation*}
\end{proof}
With these preparations in place, our estimate for the minor arcs is as follows.
Recall the definition of $T_H(\delta)$ in \eqref{eq:L1concentr}.
\begin{lemma} \label{lem:nondiagminrARCS} 
When $ (n_2 l_1)^{d_j} -(n_1 l_2)^{d_j } \neq 0$ for all $j=1,\ldots, m$, the contribution towards $\c J(\b n, \b l )$ of those $\valpha,\vbeta\in\mathbb{T}^2$ that satisfy~\eqref{def:minorarcs} for some $h\in\{1,\ldots,m\}$ 
 is 
 \begin{equation*}
 \ll x^{2d }  H^{d+m} (\log H)^{m 2^{C+1}} T_H(\delta). 
 \end{equation*}
\end{lemma} 

\begin{proof}
Fix $h\in\{1,\ldots,m\}$ such 
that \eqref{def:minorarcs} holds.
Starting from \eqref{eq:Jcirclemthd}, using~\eqref{eq:easybnd} to bound 
$\overline{S(\boldsymbol \alpha)}$ and $S(\boldsymbol \beta) $,
and using~\eqref{eq:decaycoeffinal} 
for all  $1\leq j\leq m$ and all $k\notin\{0,d_j\}$
to bound $  K_H(\alpha_j n_1^k n_2^{d_j-k} 
- \beta_j  l_1^k l_2^{d_j-k}  )$, we see that
the contribution is 
\begin{equation}\label{def:intermedd} 
 \ll x^{2d} H^{d+m} (\log H)^{m 2^{C+1}} 
\int_{\substack{ \mathbb T^{2m} \\ \eqref{def:minorarcs} }} 
 \prod_{  j=1}^m  \prod_{  k=0,d_j}   
 K_H(\alpha_j n_1^k n_2^{d_j-k} - \beta_j  l_1^k l_2^{d_j-k}  )
\mathrm d\boldsymbol \alpha \mathrm d\boldsymbol\beta
.\end{equation} 
For $j\in \{1,\ldots,m\}\setminus \{h\} $
we use  
Lemma~\ref{lem:KH_integral_transformation_new} with 
$A=   n_2^{d_j}, B=   l_2^{d_j},
C= n_1^{d_j}  ,D=  l_1^{d_j} 
$ and $E=\mathbb{T}^2$ to get $$\int_{\mathbb T^2}  \prod_{  k=0,d_j}   
 K_H(\alpha_j n_1^k n_2^{d_j-k} - \beta_j  l_1^k l_2^{d_j-k}  )
\mathrm d \alpha_j 
\mathrm d\beta_j = 4\pi^2\ll 1
.$$ Hence, \eqref{def:intermedd} becomes 
$$ \ll x^{2d} H^{d+m} (\log H)^{m 2^{C+1}}
\int_{\substack{ \mathbb T^{2} \\ \eqref{def:minorarcs} }} 
  \prod_{  k=0,d_j}   K_H(\alpha_h n_1^k n_2^{d_h-k} - \beta_j  l_1^k l_2^{d_h-k}  )
\mathrm d \alpha_h \mathrm d\beta_h .$$ 
Alluding to Lemma \ref{lem:KH_integral_transformation_new} with 
$E=\{(\alpha,\beta) \in \mathbb T^2\where \max\{\|\alpha\|,\|\beta\|\}>\delta\}$, we see that the integral is equal to
\begin{equation*}
  \int_{\substack{ (\alpha,\beta) \in \mathbb T^2 \\ 
\|\alpha\| \textrm{ or } \|\beta\| > \delta }} 
K_H(\alpha)K_H(\beta) \mathrm d \alpha\mathrm d\beta \ll T_H(\delta).\qedhere
\end{equation*}
\end{proof} 

\subsection{Major arcs} \label{s:majorarcization}
The main idea
in this section 
is to show that 
the $\boldsymbol \alpha, 
\boldsymbol \beta \in \mathbb T^{2m}$ 
left untreated by 
Lemma~\ref{lem:nondiagminrARCS} 
lie near vectors of rationals  with small 
denominator. This will enable us to  
extract savings from the sums 
$S(\boldsymbol \alpha)$ and $S(\boldsymbol \beta)$.  
\begin{lemma} \label{lem:whaatgs6} 
Let $A,B,C,D$ be   integers with $AD\neq BC $
and let    $\alpha,\beta \in \mathbb T$ be such that  
$$ \| A \alpha   - B \beta    \| \leq \delta \textrm{ and } 
\| C \alpha   - D \beta   \| \leq \delta .$$ 
Set $ q:=AD-BC$.
Then there are integers $a,b$
such that 
$$\l|\alpha -2\pi  \frac{a}{q}\r| \ll \delta \frac{|B|+|D|}{ | q|}  
\ \textrm{ and } \ 
\l|\beta -2\pi\frac{b}{q}\r| \ll \delta \frac{|A|+|C|}{ | q|}  
,$$ with absolute implied constants.
\end{lemma} \begin{proof} Let  
$s:= A \alpha   - B \beta $ and $t:= C \alpha   - D \beta$ so that 
$$ \frac{Ds-Bt}{  q} = \alpha  
\ \ \textrm{ and }  \ \ 
\frac{ C s-A t }{  q }= \beta .$$
By assumption there are integers $N,M$ with 
$s= 2\pi N+O(\delta)$ and $t=2\pi M+O(\delta )$.
Hence, $$\alpha = 2\pi
\frac{D N  -B M }{  q} 
+O\l(\delta \frac{|D|+|B|}{ | q|} \r)=2\pi \frac{a}{q}+O\l(\delta \frac{|D|+|B|}{ | q|} \r) $$ for some integer $a$.
Similarly, $\beta= 2\pi \frac{b}{q}+O(\delta \frac{|A|+|C|}{|q|})$ 
for some integer $b$.
  \end{proof}

We use the following higher-dimensional version of summation by parts.

\begin{lemma} \label{lem:partsumation} Let $F:\Z^m\to \mathbb C$ and 
$\b M, \b N \in \Z^m$ 
such that $M_k\leq N_k$ for all $1\leq k\leq m$.
For any $\b v \in \RR^m$ with $v_k\in [M_k,N_k]$, write 
$$A(\b v ):=\sum_{\substack{ \b t \in \Z^m \\ \forall k, 
M_k\leq t_k \leq v_k } } F(\b t ), $$ 
and let 
$\c B:=\max_{\b v}|A(\b v )|$.
Then, for all such $\b v$ and all 
$\boldsymbol \eta \in \mathbb R^m$  
we have  
$$ \Big|\sum_{\substack{ \b t \in \Z^m \\ \forall k, 
M_k\leq t_k \leq v_k } } F(\b t ) 
\mathrm e ^ {i\boldsymbol \eta \cdot \b t }\Big| \leq \c B 
\prod_{k=1}^m (1 + |\eta_k| (N_k-M_k)).$$ 
\end{lemma}

\begin{proof} 
We show by induction over $j\in\{0,\ldots,m\}$ that the bound holds for $\veta\in\RR^j\times\{0\}^{m-j}$.
If $j=0$, i.e. $\veta=\vzero$, this follows immediately from the definition of $\c B$.

For $j>0$, take $\veta\in\RR^{j}\times\{0\}^{m-j}$ and write $\veta':=(\eta_1,\ldots,\eta_{j-1},0,\ldots,0)$.  
Using the Abel sum formula for the sum over $t_j$, we obtain 
\begin{align*}\sum_{\substack{ \b t \in \Z^m \\ \forall k, 
M_k\leq t_k \leq v_k } }\hspace{-0.7cm} (F(\b t ) 
\mathrm e ^{ i  \boldsymbol \eta' \cdot \b t })e^{i \eta_j t_j}
&= \bigg( \sum_{\substack{ \b t \in \Z^m \\ \forall k, 
M_k\leq t_k \leq v_k} }\hspace{-0.7cm}  F(\b t ) 
\mathrm e ^{ i\veta'\cdot \vt} \bigg) 
\mathrm e ^{i \eta_j v_j }
- i \eta_j 
\int_{M_j}^{v_j}
\bigg( \sum_{\substack{ \b t \in \Z^m, M_j \leq t_j\leq u 
\\ \forall k \neq j , M_k\leq t_k \leq v_k  } }\hspace{-0.7cm}  F(\b t ) 
\mathrm e ^{ i\veta'\cdot \vt} \bigg)
\mathrm e ^{ i\eta_j u } \mathrm du.
\end{align*}
With the inductive hypothesis, this is bounded in absolute value by
\begin{equation*}
\left(\cB\prod_{k=1}^{j-1}\left(1+|\eta_k|(N_k-M_k)\right)\right)\left(1+|\eta_j|(v_j-M_j)\right).\qedhere
\end{equation*}
\end{proof}

Recall the definition of $E_f$ in
\eqref{def:porporapassio}.

\begin{lemma} \label{lem:applyinglevelofdistrib}
Let $\b a \in \Z^m$, $\b q\in (\Z\setminus \{0\})^m$ and $\boldsymbol \eta \in \R^m$, and write $\alpha_i := 2\pi\frac{a_i}{q_i}+\eta_i$ for $1\leq i\leq m$. Then
$$ S\l(\valpha\r) \ll    E_f((1+d_1)x^{d_1}H, \ldots,
(1+d_m)x^{d_m}H ;\b q) \prod_{k=1}^m \max\{1, |\eta_k| x^{d_k} H\},$$ 
where the implied constant  depends only on $m$
and $d_1,\ldots,d_m$.
\end{lemma}  
\begin{proof} 
 Let $x_k:= (1+d_k)x^{d_k}H$. Recall the definition of 
 $\mathtt{I}$   in \eqref{def:III}.
 For $\b v \in \R^m \cap  \mathtt{I}$  
we have $$A(\b v ) := 
\sum_{\substack{ \b t \in \Z^m \\ \forall k,
-x_k \leq t_k \leq v_k } }f(\b t ) \mathrm e ^{2 \pi  i  \sum_{k=1}^m 
a_kt_k/q_k}   \ll   E_f(x_1, \ldots,
x_m ;\b q).$$ Using Lemma~\ref{lem:partsumation} with 
$F(\b t ) = f(\b t ) \exp(2 \pi i \sum_{k=1}^m
\frac{a_k}{q_k} t_k )$
we obtain the desired bound. \end{proof}

\begin{lemma} \label{lem:finalmajorarc} 
For each $j=1,\ldots, m$ let $q_j:= (n_2 l_1)^{d_j}- (n_1 l_2)^{d_j}$.
If  $ |q_j |> x^{2d_j } /\xi_0$
for all $j=1,\ldots, m$, then the  $\boldsymbol \alpha, \boldsymbol \beta
\in \mathbb T^{2m}$ for which~\eqref{def:minorarcs} fails 
for every  $1\leq h \leq m$   contribute towards $\c J(\b n ,\b l )$
a quantity that is  
$$\ll H^{d-m} 
\l( \max\l\{1, \delta \xi_0 H \r\}^{2}\r)^m 
E_f(((1+d_j)x^{d_j}H)_{j=1}^m   ;\b q)^2 
 .$$ \end{lemma}
\begin{proof} For each $h\in \{1,\ldots,m\}$,
we  use  
Lemma~\ref{lem:whaatgs6} to find 
$a_h,b_h\in \Z$ such that  \begin{equation}\label{eq:hilariouseq}
 \l|\alpha_h -2\pi  \frac{a_h}{ q_h}\r| 
\ll \delta \frac{x^{d_h }}{ | q_h |}\ll \frac{\delta \xi_0}{x^{d_h }} ,\quad
\l|\beta_h -2\pi\frac{b_h}{ q_h}\r| \ll \delta \frac{x^{d_h }}{ | q_h |}  
\ll \frac{\delta \xi_0}{x^{d_h }}. \end{equation} 
By Lemma~\ref{lem:applyinglevelofdistrib} 
with $\eta_h:=\alpha_h -2\pi   a_h/ q_h$ we get $$ S\l(
\boldsymbol \alpha \r)  \ll    E_f((1+d_1)x^{d_1}H, \ldots,
(1+d_n)x^{d_m}H ;\b q)  
\max\l\{1, \delta \xi_0 H \r\}^m.$$
The same bound is analogously proved   
for $S(\boldsymbol \beta)$. 
The contribution to $\c J$ 
in~\eqref{eq:Jcirclemthd} is 
$$ \ll \c G  
E_f(((1+d_j)x^{d_j}H)_{j=1}^m   ;\b q)^2 
 \max\l\{1, \delta \xi_0 H \r\}^{2m}
,$$ where   $$\c G= \int_{\mathbb T^{2 m}}   \prod_{  j=1}^m  \prod_{  k=0}^{d_j}   
 K_H(\alpha_j n_1^k n_2^{d_j-k} - \beta_j  l_1^k l_2^{d_j-k}  )
\mathrm d\boldsymbol \alpha \mathrm d\boldsymbol\beta 
.$$ We use \eqref{eq:decaycoeffinal} 
to bound each term in $\c G$ corresponding to 
$j\in \{1,\ldots,m\}$ and $k\notin \{0,d_j\}$. Thus, 
   $$\c G\ll H^{d-m} \prod_{  j=1}^m  
\int_{\mathbb T^{2 }}   
 K_H(\alpha   n_2^{d_j}  - \beta    l_2^{d_j}  )
  K_H(\alpha n_1^{d_j}   - \beta   l_1^{d_j}  ) 
\mathrm d \alpha 
\mathrm d\beta 
.$$ 
The integral is $\ll 1$ as can be seen by Lemma~\ref{lem:KH_integral_transformation_new}.\end{proof}

\subsection{Conclusion of the proof of Theorem~\ref{thm:generalkernels}}\label{s:beginendpproof}
Feeding the bounds from  
Lemma~\ref{lem:cimento23fa}, Lemma~\ref{lem:nondiagminrARCS} 
and Lemma~\ref{lem:finalmajorarc} to 
the right-hand side of~\eqref{eq:disprscircle} suffices for the proof.

\subsection{Heat kernels}\label{s:specialheat!}   
To apply Theorem~\ref{thm:generalkernels} we need to choose 
a kernel $K_H$ such that both $K_H$ and $\widehat{K}_H$
decay fast  in the sense of 
\eqref{eq:L1concentr} and \eqref{eq:decaycoeffi}. By Heisenberg's uncertainty principle
the heat kernel is a good candidate. 
It arises when        describing the 
temperature distribution  $u(x,t)$ on a circular ring, where  $2\pi x $ 
is the angle  
of a point and $t>0$ denotes the time, 
see \cite[\S 4.4]{MR1970295}, for example.
Under the initial condition $u(x,0)=g(x)$, the function $u$ satisfies the differential 
equation $$ \frac{\partial u}{\partial t}=c\frac{\partial^2 u}{\partial x^2},$$
where $c$ is a physical constant. For $c=1$
the  solution of the differential equation 
is given by  $u(x,t)=(g\ast G(\cdot,t))(x),$ where $\ast$ is the convolution on 
$\R/\Z$
and $$G(x,t):=
\sum_{n \in \Z}\mathrm e^{- 4 \pi^2 n^2 t}
\mathrm e^{ 2\pi  i  n x}.$$

The heat kernel gives rise to positive positive 
summability kernels that satisfy all the 
requirements of Theorem \ref{thm:generalkernels}.
Define for the rest of this section $$
K_H(\alpha) := G\left(\frac{\alpha}{2\pi},
\frac{1}{4\pi H^2}
\right), \ \ \ 
H\geq 1,\ \alpha \in \mathbb T.
$$
\begin{lemma}\label{lem:heatker}
  The functions $K_H$ for $H\geq 1$ are positive summability kernels satisfying \eqref{eq:positivfouriertransf} with $c_0=e^{-\pi}$, \eqref{eq:decaycoeffi} with $\beta_0=1,
\beta=2
  $, and \eqref{eq:fourierinversion}. Moreover, for any $\delta\in (0,\pi)$, we have 
  $$T_H(\delta)\ll \frac{1}{\delta H\exp((\delta H)^2/(4\pi))}$$
  with an absolute implied constant.
\end{lemma}

Before we prove the lemma, let us apply it with Theorem \ref{thm:generalkernels} to obtain the following result.

\begin{corollary}\label{cor:heat:kernel}
Let $m, d_1,\ldots,d_m\in\N$, $B>0$ and $C\geq 0$. 
For   any $f:\Z^m\to \mathbb C$ 
 satisfying~\eqref{def:weaksigwalfs},
 any $1\leq \xi_0\leq x\leq H$,
 any $1\leq \xi\leq  
 H/(2\pi)
 $ and any  
 $a:\Z^2\to \{z\in \mathbb C:|z|\leq 1\}$, 
we have  
\begin{align*}
&\frac{1}{\#\c F_\Z(H)} \sum_{\b F \in \c F_\Z(H)} \Bigg|\sum_{\substack{  
\b n\in \Z^2 \cap[-x,x]^2 }} a(\b n)  f(F_1(\b n), \ldots, F_m(\b n) )  \Bigg|^2   \\ & \ll  \l\{
\frac{ x^{2d }}{\xi\mathrm e^{\pi \xi^2}}
 (\log H)^{m 2^{C+1}}    \  + \   
 \frac{ (\log H)^{ \gamma_1 } (\log x)^{  2^{2(1+2\gamma_0)} } }{\xi_0^{1/(2\c D)}}   \r\}\cdot x^4 \\ &+ \Bigg(\frac{ 
\xi \xi_0}{H}\Bigg)^{2m} 
\sum_{\substack{ \b n,\b l \in \Z^2, n_2 l_1\neq \pm n_1 l_2
\\ x/\xi_0^{1/2}<|n_i|,|l_i| \leq x  }}  
 E_f(((1+d_j)x^{d_j}H)_{j=1}^m;
( (n_2 l_1)^{d_j}- (n_1 l_2)^{d_j})_{j=1}^m )^2 
,\end{align*} 
where the implied constant depends only on $m,d_1,\ldots,d_m,B,C$.
\end{corollary}

\begin{proof}
    Apply Theorem \ref{thm:generalkernels} with the heat kernel and the bound for $T_H(\delta)$ specified in Lemma \ref{lem:heatker}, taking $\delta=2\pi\xi/H$.
\end{proof}

\begin{proof}[Proof of Lemma \ref{lem:heatker}]
In this proof, we identify $\mathbb T$ with $(-\pi,\pi]$, so any $\alpha\in \mathbb T$ satisfies $|\alpha|\leq \pi$. With the Jacobi theta function
$$ \vartheta(z;\tau) := \sum_{n \in \mathbb Z} 
\exp(\pi i n^2 \tau +2\pi i n z), 
$$
defined for $z,\tau\in \mathbb C$ with $\Im(\tau)>0$,
we have 
$$ K_H(\alpha)=\vartheta(\alpha/2\pi;i/H^2).$$
The modular transformation corresponding to the $\mathrm{SL}_2(\mathbb{Z})$-action 
 $\tau\mapsto -1/\tau$ 
 satisfies the following identity:
\begin{equation}\label{eq:mumford}
\vartheta(z/\tau;-1/\tau) =\exp(-\pi i /4) \tau^{1/2}
\exp(\pi i z^2/\tau)
\vartheta(z;\tau) 
,\end{equation} 
where $\tau^{1/2}$ is chosen to lie in the first quadrant. 
See, for instance, \cite[Theorem 7.1]{MR2352717}. We apply this with $z=\alpha/(2\pi)\in (-1/2,1/2]$ and $\tau=iH^{-2}$
to obtain
\begin{equation}\label{eq:T_H_pw_bd}  
 \sum_{m\in \Z} \exp(-\pi H^2 (m-\alpha/(2\pi))^2)=
 H^{-1} K_H(\alpha).
\end{equation}
This shows that $K_H(\alpha)$ is indeed a positive real function. Its Fourier transform is
$$\widehat{K}_{H}(n)=\mathrm{e}^{-\pi n^2/H^2},$$
which shows in particular that 
$\widehat{K}_{H}(0)=1$ and thus \eqref{eq:normalise}. Moreover, it implies that \eqref{eq:positivfouriertransf} holds with $c_0=e^{-\pi}$, and that $\eqref{eq:decaycoeffi}$ holds with $\beta_0=1,
\beta=1
$. The inversion formula \eqref{eq:fourierinversion} holds by definition of $K_H$. 

For \eqref{eq:L1concentr} and the explicit estimate stated in the lemma, we now proceed to bound the expression on the left-hand side of \eqref{eq:T_H_pw_bd}, noting that 
$$
\frac
{ \exp(-\pi H^2 (m-\alpha/(2\pi))^2) }
{ \exp(-\pi H^2 (\alpha/(2\pi))^2) } 
=\exp(-\pi H^2 
(m^2-m\alpha/\pi)
)
 \leq \exp(-\pi H^2 ( m^2-m)  )
.$$ If $|m| \geq 2 $ we have $m^2 - m \geq |m|/2$,  hence 
$$\exp(-\pi H^2 (m-\alpha/(2\pi) )^2)
\leq \exp(-\pi H^2 (\alpha/(2\pi))^2) 
\exp(- H^2 |m|/2  ).$$
This shows that 
\begin{align*}
 \sum_{|m|\geq 2  }  
 \exp(-\pi H^2 (m-\alpha/(2\pi)  )^2)
& 
\ll
\exp(-\pi H^2 (\alpha/(2\pi) )^2)
 \sum_{m\geq 1 } \exp(- H^2 m/2  ) \\
 &\ll \exp(-\pi H^2 (\alpha/(2\pi))^2) 
.\end{align*}
As $\alpha/(2\pi)\in (-1/2,1/2]$, the terms with $m=-1,1$ 
are bounded by the term with $m=0$. Hence, in total we see from \eqref{eq:T_H_pw_bd} that
\begin{equation*}
K_{H}(\alpha)\ll H \exp(-\pi H^2 (\alpha/(2\pi))^2)
\end{equation*}
with an absolute implied constant. This implies that
\[T_H(\delta) \ll H \int_{\delta}^{\pi} 
\frac{ \mathrm d\alpha}
{\mathrm e ^{ H^2 \alpha^2/4\pi} }  
=    \int_{\delta H}^{\pi H } 
\frac{ \mathrm d\beta}
{\mathrm e ^{ \beta^2/4\pi} }  
\leq  \int_{\delta H }^{\infty } \frac{\beta}{\delta H}
 \frac{ \mathrm d\beta}{\mathrm e ^{\beta^2/4\pi} }  
 \ll \frac{1}{\delta H \mathrm e^{
 (\delta H)^2/(4\pi)}} .\qedhere\]
\end{proof} 
 
We conclude this section with a special case of 
Corollary \ref{cor:heat:kernel}.
For $\b q\in (\Z\setminus \{0\})^m$
and $\b x \in [1,\infty)^m$  define 
\begin{equation} \label{defwalfsjnsf9}
\c E_f(\b x ;\b q )
:= \max_{\b r\in \prod_{j=1}^m (\Z/q_j\Z)}  
\ \sup_{\substack{ \b v \in  \R^m \\ \forall k,
|v_k| \leq x_k } } \
\l|  \sum_{\substack{ \b t \in \Z^m,
-x_k \leq t_k \leq v_k \forall k\\ t_k \equiv  r_k \md {q_k}  \forall k } }
f(\b t )\r|.\end{equation} Thus, 
$f$ has average $0$ over the interval 
$\prod_{j=1}^m[-x_j,x_j]$ and along   arithmetic 
progressions modulo $\b q$ equivalently when  
$\c E_f(\b x ;\b q ) =o((x_1\cdots x_m)/(q_1\cdots q_m))$.
Recall \eqref{def:porporapassio} and note that 
$$\sum_{\substack{ \b t \in \Z^m \\
-x_k \leq t_k \leq v_k    \forall k } }
f(\b t ) \mathrm{e}^{2\pi i \sum_{k=1}^m 
\frac{b_k t_k}{q_k}}
=\sum_{\b r \in \prod_{j=1} ^m (\Z/q_j \Z) }
\mathrm{e}^{2\pi i \sum_{k=1}^m \frac{b_k r_k}{q_k}}
\sum_{\substack{ \b t \in \Z^m,
-x_k \leq t_k \leq v_k \forall k\\ t_k \equiv  r_k \md {q_k}  \forall k } }
f(\b t ),$$ hence, bounding 
$\mathrm{e}^{2\pi i \sum_{k=1}^m \frac{b_k r_k}{q_k}}$ trivially by $1$ yields
\begin{equation}\label{def:paganiniguitara} 
E_f(\b x ;\b q )  \leq |q_1 \cdots q_m|
\c E_f(\b x ;\b q ) 
.\end{equation} 
Recall the definitions of $\gamma_1,\gamma_2,d,\c D$ from \eqref{pasquini}.

\begin{corollary}\label{cor:chiocolato}
Let $m,d_1,\ldots,d_m\in\N$, $N,B>0$ and $C\geq 0$. With
\begin{equation*}
    \kappa_1:= 2\c D\gamma_1\quad\text{ and }\quad\kappa_2:=2\c D( N+ 2^{2(1+2\gamma_0)}),
\end{equation*}
for any function $f:\Z^m\to \mathbb C$ 
 satisfying~\eqref{def:weaksigwalfs},
 any  $a:\Z^2\to \{z\in \mathbb C:|z|\leq 1\}$,
all $H\geq 2$ and all $x$ in the range 
 $(\log H)^{\kappa_1+\kappa_2}\leq x \leq H$,
 we have   \begin{align*}
&\frac{1}{\#\c F_\Z(H)} 
\sum_{\b F \in \c F_\Z(H)} \Bigg|\frac{1}{x^2}
\sum_{\substack{  
\b n\in \Z^2 \cap[-x,x]^2 }} a(\b n)  f(F_1(\b n), 
\ldots, F_m(\b n) )  \Bigg|^2  \ll 
\frac{1}{(\log x)^N} 
\\ & +
(\log H)^{2m\kappa_1}  
(\log x)^{2m\kappa_2+m}  x^{4d}
\l(\max_{\substack{\b q\in (\Z\setminus \{0\})^m 
\\ |q_j| \leq 2x^{2d_j}\forall j }} 
\frac{\c E_f(((1+d_j)x^{d_j}H)_{j=1}^m;\b q )}
{  H^m}\r)^2 
,\end{align*}
where the implied constant depends only on 
$m,d_1,\ldots,d_m,B,C$ and $N$.
\end{corollary}

\begin{proof}
We may assume that $H$ is sufficiently large in terms of $m,d_1,\ldots,d_m,B,C,N$. Choose $\xi_0$ and $\xi$ by 
$$\xi_0:=(\log H)^{\kappa_1} 
(\log x)^{\kappa_2},\quad \xi^2:= \kappa_2\log x + N\log\log x + (\kappa_1+m2^{C+1})\log \log H 
.$$
Then one directly sees that $1\leq\xi_0\leq (\log H)^{\kappa_1+\kappa_2}\leq x$, and the estimate $\xi\ll (\log H)^{1/2}$ shows that $1\leq \xi\leq 
H/(2\pi)
$ for large enough $H$. Hence, we may and apply Corollary \ref{cor:heat:kernel}. The second error term is
$$\frac{ (\log H)^{ \gamma_1 } (\log x)^{  2^{2(1+2\gamma_0)} } }{\xi_0^{1/(2\c D)}}   x^4 =
\frac{x^4}{(\log x)^N},$$
while the first error term is  
$$ \ll  
\frac{ x^{2d }  }{\mathrm e^{ \xi^2}}
 \frac{ (\log H)^{m 2^{C+1}}  }{\mathrm e^{ \xi^2}}\frac{1}{\mathrm e^{ \xi^2}}
 x^4 \leq  \frac{1}{\mathrm e^{ \xi^2}} x^4\leq \frac{x^4}{(\log x)^N}.$$
By \eqref{def:paganiniguitara} 
the last error term is 
$$
\ll   ( \xi \xi_0 )^{2m} 
x^{4+4d} 
\l(\max_{\substack{\b q\in (\Z\setminus \{0\})^m 
\\ |q_j| \leq 2x^{2d_j}\forall j }} 
\frac{\c E_f(((1+d_j)x^{d_j}H)_{j=1}^m;\b q )}
{ H^m}\r)^2, 
  $$ 
  and $(\xi\xi_0)^{2m}\ll (\log H)^{2m\kappa_1}(\log x)^{2m\kappa_2+m}$.
 \end{proof}

\section{Randomness law for the analytic Hilbert symbol}
\label{sec:siegel_walfisz}   We prove
Theorem~\ref{thm:pergolesiguglielmo} 
in~\S \ref{s:occupies} by reducing to 
following lower dimensional analogues: 
\begin{theorem}\label{thm:bitethemedals}
Fix any $\epsilon>0$ and $\sigma_1,\sigma_2\in \{-1,1\}$.
Assume that $ a,b,c:\N \to  \mathbb C$ are arbitrary functions bounded 
by $1$ in modulus. Then for any   $x_1,x_2,x_3,z\geq 1$ we have
$$ \sum_{\substack{ \b t \in \N^3 
\\ t_i \leq x_i \forall i } } \dran(\sigma_1 t_1 t_3, \sigma_2 t_2 t_3) 
a(t_1) b(t_2)c(t_3)\ll (x_1x_2x_3 )^{1+\epsilon }
\l( \frac{1}{z^{1/9} }+ \frac{z^{1/9} }{\min_i \sqrt{ x_i }} + \frac{z}{\sqrt{x_1x_2x_3}}\r) ,$$where the implied constant depends only   on $\epsilon$.  
\end{theorem}
\begin{theorem}\label{thm:olympiadepergolesi}
Fix any $\epsilon>0$ and $\sigma_1,\sigma_2\in \{-1,1\}$.
Assume that $ a,b:\N \to  \mathbb C$ are arbitrary functions bounded 
by $1$ in modulus. Then for any   $x_1,x_2,z\geq 1$
 we have  
$$ \sum_{\substack{ \b t \in \N^2 \\ 
t_i \leq x_i \forall i } } 
\dran(\sigma_1 t_1, \sigma_2 t_2 ) 
a(t_1) b(t_2)\ll (x_1x_2 )^{1+\epsilon } 
\l( \frac{1}{z^{1/9} }
+ \frac{z^{1/9} }{\min_i \sqrt{ x_i }}
+\frac{z}{\sqrt{x_1x_2}}
+\frac{z^{4/9}}{\min_i x_i}
\r) ,$$where the implied constant depends only   on $\epsilon$.
\end{theorem}  
The proof of Theorem~\ref{thm:olympiadepergolesi}
follows along 
similar but simpler lines than
that of 
Theorem~\ref{thm:bitethemedals} 
and is  
briefly outlined  in
\S\ref{s:pereli}. 
The proof of Theorem~\ref{thm:bitethemedals} is in 
\S\S\ref{sec:dealing_with_small_Nt}--\ref{ss:prfthrm}.
\begin{remark}\label{rem:what is this noise?}
The heart of the argument is 
that the terms 
in $\ddet$ give rise to sums involving quadratic characters of small moduli, thus, one can only hope for logarithmic savings by 
Siegel--Walfisz type theorems. In contrast, 
 $\dran$ contains terms that give rise to 
sums    involving quadratic characters of 
large moduli that can be bounded with polynomial savings
by the large sieve for quadratic characters as proved 
by Heath--Brown~\cite[ Corollary 4]{MR1347489}.
\end{remark}
\begin{lemma}[Heath--Brown]\label{lem:rogier}Fix any $\epsilon>0$. 
Then for all  positive integers $M,N$ and all complex numbers 
 $a_1,\ldots, a_M,b_1,\ldots, b_N$ satisfying 
 $|a_m|,|b_n|\leq 1$ we have 
 $$\sum_{\substack{ m\leq M \\ 2\nmid m }}
 \sum_{n\leq N } a_m b_n \left(\frac{n}{m} \right)\ll 
 (MN )^{1+\epsilon} \min\{M,N\}^{-1/2}  ,$$ where the
 implied constant depends only on $\epsilon$. \end{lemma}

\subsection{Proof of Theorem~\ref
{thm:pergolesiguglielmo}}\label{s:occupies}\begin{proof}
First we assume that $m_3>0$.
We can write
the sum as 
$$\sum_{\substack{ 1\leq n_1 \leq x_1\cdots x_{m_1}
\\ 1\leq n_2 \leq y_1\cdots y_{m_2}\\ 1\leq n_3 \leq z_1\cdots z_{m_3}
}}\dran(\sigma_1 n_1 n_3,  \sigma_2 n_2 n_3  )
a'(n_1) b'(n_2) c'(n_3), \ \textrm{ where }   
a'(n_1):= \sum_{\substack{ \forall i, 
1\leq s_i \leq x_i\\ s_1 \cdots s_{m_1} = n_1 }} a(\b s ) $$ and 
$b',c'$ are defined analogously. 
Let $\tau_m(n)$ be the number of ways of writing $n$ 
are a product of $m$ positive integers and recall that 
for every fixed $\epsilon>0$ we have 
$\tau_m(n) \leq C(m,\epsilon) n^\epsilon$
for some $C(m,\epsilon)>0$.
Since $|a'(n_1)|\leq \tau_{m_1}(n_1)$, we note that the function
$$ a''(n_1):= \frac{a'(n_1) }{ C(m_1,\epsilon) (x_1\cdots x_{m_1})^{\epsilon}}
$$ is bounded by $1$ in modulus. Defining $b''$ and $c''$ analogously, 
we write the sum as 
$$\prod_{i=1}^3C(m_i,\eps)\bigg(  \prod_{i=1}^{m_1} x_i \prod_{i=1}^{m_2} y_i
  \prod_{i=1}^{m_3} z_i \bigg)^{\epsilon}\sum_{\substack
  { 1\leq n_1 \leq x_1\cdots x_{m_1}\\ 1\leq n_2 \leq y_1\cdots y_{m_2}
\\ 1\leq n_3 \leq z_1\cdots z_{m_3}}}\dran
(\sigma_1 n_1 n_3,  \sigma_1 n_2 n_3  )a''(n_1) b''(n_2) c''(n_3),$$ 
which we bound by Theorem~\ref{thm:bitethemedals}.
When $m_3=0$ we use Theorem~\ref{thm:olympiadepergolesi} instead.
\end{proof}

\subsection{Dealing with small values of $N_\vt$}\label{sec:dealing_with_small_Nt} 
Let us observe first that, by Definition of $N_\vt$ in \eqref{eq:def_Nt}, for all $\vt\in(\ZZ\smallsetminus\{0\})^2$ we have
\begin{equation}\label{eq:dran_bound}
|\dran(\vt)|\leq |\updelta(\vt)|+|\ddet(\vt)|\ll\tau(N_\vt)\ll_\epsilon (t_1t_2)^{\epsilon}.
\end{equation}
Hence, the statement of Theorem \ref{thm:bitethemedals} is trivial if $z\ll 1$ or $z\geq (x_1x_2x_3)^{1/2}$. We will henceforth assume that $z$ is sufficiently large (in terms of $\epsilon$ only), and that $z\leq(x_1x_2x_3)^{1/2}$.

The analysis in \eqref{eq:delta_splitting} shows that for all $\vt\in\ZZ^2$ with $N_\vt>z^2$, the value of $\dran(\vt)$ 
is equal to\begin{equation}\label{eq:def_delta_randh}
  \dranh(\vt):= \sum_{\substack{ 
 s \textrm{ square-free}\\  
 z <s<  \frac{N_\b t}{z}  } }   
 \prod_{p\mid s }(t_1,t_2)'_p.
\end{equation}
We show first that replacing $\dran$ by $\dranh$ introduces an acceptable error in Theorem \ref{thm:bitethemedals}.

\begin{lemma}\label{lem:delta_rand_hat}
  The sum over $\vt$ in Theorem \ref{thm:bitethemedals} is equal to
  \begin{equation*}
  \sum_{\substack{ \b t \in \N^3 
\\ t_i \leq x_i \forall i } } \dranh(\sigma_1 t_1 t_3, \sigma_2 t_2 t_3) 
a(t_1) b(t_2)c(t_3) + O\left((x_1x_2x_3)^{1/2+\epsilon}z\right),
  \end{equation*}
  with the implied constant depending only on $\epsilon$.
\end{lemma}

\begin{proof}
We have already seen that $\dran(\vt)=\dranh(\vt)$ for all $\vt\in\ZZ^2$ with $N_\vt>z^2$. When $N_\vt\leq z^2$, then $\dranh(\vt)=0$, so \eqref{eq:dran_bound} shows that $|\dran(\vt)-\dranh(\vt)| = |\dran(\vt)|\ll_\epsilon (t_1t_2)^\epsilon$. 
Hence, we can bound the error introduced when replacing $\dran$ by $\dranh$ in Theorem \ref{thm:bitethemedals} by
\begin{equation}\label{eq:Nt_small_estimate}
  \ll_\epsilon (x_1x_2x_3)^{\epsilon}\#\{\b t \in \N^3\where t_i \leq x_i \text{ for all } i \text{ and } N_{(\sigma_1t_1t_3,\sigma_2t_2t_3)}\leq z^2\}.
\end{equation}
We can uniquely write $t_i=a_iv_i^2$ with $a_i\in\NN$ square-free and $v_i\in\NN$. Grouping together the primes according to which of $a_1,a_2,a_3$ they divide, we may further uniquely write
\begin{equation*}
 a_1=u_{123}u_{12}u_{13}u_1,\quad a_2=u_{123}u_{12}u_{23}u_2,\quad a_3=u_{123}u_{13}u_{23}u_3
\end{equation*}
with $u_{123},u_{12},u_{13},u_{23},u_1,u_2,u_3$ square-free and pairwise coprime. From the definition of $N_\vt$, we observe that then   
$u_1u_2u_3$   
divides $N_{(\sigma_1t_1t_3,\sigma_2t_2t_3)}$. This allows us to upper-bound
the quantity in \eqref{eq:Nt_small_estimate} by 
\begin{align*}
  & (x_1x_2x_3)^{\epsilon} 
  \sum_{u_1u_2u_3\leq z^2} 
    \sum_{\substack{
    u_{12},  u_{123}  \leq x_2
    \\
    u_{13},u_{23} \leq x_3  }}
  \prod_{i=1}^3\sum_{v_i^2\leq \frac{x_i}{a_i}}1
  \\ &\ll_\epsilon (x_1x_2x_3)^{1/2+2\epsilon}\sum_{u_1u_2u_3
  \leq z^2}\frac{1}{   \sqrt{u_1u_2u_3}}
  \ll_\epsilon (x_1x_2x_3)^{1/2+3\epsilon} z.\qedhere
\end{align*}\end{proof}
\subsection{Factorisation and reciprocity}\label{s:preli}
\begin{lemma}\label{lem:euridice}For any prime $p$ and all 
$a,b,t_1,t_2\in \Z_p\setminus \{0\}$ we have 
$(a^2t_1,b^2t_2)'_p=(t_1,t_2)'_p$.
\end{lemma} \begin{proof}For $p\neq 2$ the proof follows by noting that $v_p(t_1)\equiv v_p(a^2 t_1)\md 2$.
For $p=2$ we use that all odd squares are $1\md 4$, hence
$2^{-v_2(a^2 t_1)} a^2 t_1 \equiv 2^{-v_2(t_1)}t_1\md 4$.  \end{proof}

\begin{lemma}\label{lem:tentoglou}The sum over $\b t \in \N^3$ in 
Lemma \ref{lem:delta_rand_hat} equals
$$\sum_{\substack{\lambda \in \N \\ \b s \in \N^3} }
\sum_{\substack{\alpha,\beta_0 \in \{0,1\} \\  \alpha\leq  \beta_0}}
\sum_{\substack{ 
\beta_1,\beta_2, \beta_3, \beta_{12},\beta_{13}, \beta_{23} \in \{0,1\}
\\ \beta_1+\beta_2+\beta_3+\beta_{12}+\beta_{23}+\beta_{13}\leq 1 
} }\sum_{\substack{ k_{12}, k_{13}, k_{23} \in \N\\ 
v_2(k_{ij} )=\beta_{ij} }}
\sum_{\substack{ e_{12}, e_{13}, e_{23} \in \N\\ 
e_{ij}\mid k_{ij} ,\ 2\nmid e_{ij}} }
 \c C\l(\frac{x_1}{2^{\beta_1} s_1^2 \lambda k_{12} k_{13} },\ldots,
\frac{x_3}{2^{\beta_3} s_3^2 \lambda k_{13} k_{23} }\r )
,$$  where 
\begin{align*}\c C(\b y )  :=&
\Osum_{\substack{ e_1,e_1^*,e_2,e_2^*, e_3,e_3^* \in \N 
\\ e_i e_i^* \leq y_i  } }  
a(\lambda s_1^2  k_{12} k_{13}  2^{\beta_1 }  e_1 e_1^*)
b(\lambda s_2^2  k_{12} k_{23}  2^{\beta_2 }  e_2 e_2^*)
c(\lambda s_3^2  k_{13} k_{23}  2^{\beta_3 }  e_3 e_3^*) \\
\times & \prod_{p\mid 2^\alpha e_1e_2e_3 e_{12} e_{13} e_{23}  }
(\sigma_1 2^{\beta_1+\beta_3 } k_{12} k_{23}  e_1e_1^* e_3 e_3^* ,
\sigma_2 2^{\beta_2+\beta_3 } k_{12} k_{13}  e_2e_2^* e_3 e_3^*  
)_{\Q_p},\end{align*}   
where $\Osum$ is moreover subject to the conditions 
    \begin{equation}\label{eq:common}
   e_1 e_2 e_3 > \frac{z}{2^\alpha e_{12} e_{13} e_{23}},
   \ \ \  e_1^*  e_2^* e_3^* > \frac{z 2^\alpha   e_{12} e_{13} e_{23}}
 {k_{12} k_{13} k_{23} 2^{\beta_0-\sum_{i,j} \beta_{ij}} }
  ,  \end{equation}
 
 and   \begin{equation}\label{eq:commonivy2}
\begin{cases}   |(\sigma_1 k_{12}  k_{23} 2^{\beta_1+\beta_3 } 
e_1 e_1^* e_3 e_3^* ,\sigma_2 k_{12}  k_{13} 2^{\beta_2+\beta_3 } 
e_2 e_2^* e_3 e_3^*)'_2|=\beta_0 , \\
\mu(k_{12} k_{13} k_{23} 
e_1e_1^*e_2e_2^* e_3e_3^*)^2=1, \ \ \ 
2\nmid e_1 e_1^* e_2e_2^* e_3 e_3^* , \\
 \gcd( s_1 k_{12} k_{13} 2^{\beta_1} e_1 e_1^*,
 s_2 k_{12} k_{23} 2^{\beta_2} e_2 e_2^*,
 s_3 k_{13} k_{23} 2^{\beta_3} e_3 e_3^* )=1 
\end{cases}\end{equation}\end{lemma}
\begin{proof} 
From \eqref{eq:def_delta_randh} 
and the definition of $N_\vt$ in \eqref{eq:def_Nt}, we see that
\begin{equation}\label{eq:dran_alt}
\dranh(t_1,t_2)=\sum_{\substack{s\mid N_{\vt}\\z< s<N_\vt/z}}\prod_{p\mid s}\hs{t_1}{t_2}_{\QQ_p}.    
\end{equation}
We factor  $t_i$ to make explicit the number  $N_\b t$.  
Remove common factors of the $t_i$ by letting 
$\lambda:=\gcd(t_1,t_2,t_3)$ and let $t_i=\lambda n_i$
where $\gcd(n_1,n_2,n_3)=1$. Next, we write
$n_i=  s_i^2 f_i$, where  $f_i$ is square-free. 
By Lemma~\ref{lem:euridice} we then see that  
$$N_{(\sigma_1 t_1 t_3,\sigma_2 t_2 t_3)}
=N_{(\sigma_1 \lambda^2 n_1 n_3,\sigma_2 \lambda^2 n_2 n_3)}
=N_{(\sigma_1   n_1 n_3,\sigma_2  n_2 n_3)}
=N_{(\sigma_1 f_1 f_3,\sigma_2 f_2 f_3)}
.$$ Let $\beta_0:= |(\sigma_1 f_1 f_3,\sigma_2 f_2 f_3)'_2|
\in \{0,1\}$ so that 
$v_2(N_{(\sigma_1 f_1 f_3,\sigma_2 f_2 f_3)})= \beta_0  $.  
When $p$ is odd, we note that $p\nmid N_{(\sigma_1 f_1 f_3,\sigma_2 f_2 f_3)} $ equivalently when $v_p(f_1 f_3)\equiv v_p(f_2f_3)\equiv 0\md 2$.
Since each $f_i$ is square-free, this   happens 
exactly when both $v_p(f_1 f_3), v_p(f_2f_3)$
are in $\{0,2\}$. 
If one of them is $2$
then the other is positive, hence, equals  $2$.
This contradicts the fact that $\gcd(n_1,n_2,n_3)=1$.
Therefore, $p\nmid N_{(\sigma_1 f_1 f_3,\sigma_2 f_2 f_3)} $ 
equivalently  when $v_p(f_1 f_3)=0=v_p(f_2f_3)$, i.e. 
when $p\nmid f_1f_2f_3$. Hence   $ N_{(\sigma_1 f_1 f_3,\sigma_2 f_2 f_3)}
= 2^{\beta_0} \prod_{p\mid f_1 f_2f_3, p\neq 2}p $.
For $  i \neq j  $ let  $k_{ij}:=\gcd(f_i,f_j)$ and 
$$m_1=\frac{f_1}{k_{12} k_{13} }, 
m_2=\frac{f_2}{k_{12} k_{23} }, 
m_3=\frac{f_3}{k_{13} k_{23} }.$$ In particular,
$ m_1m_2m_3k_{12} k_{13} k_{23}$ is square-free.
Define  $\beta_i:=v_2(m_i)$, $\beta_{ij}:=v_2(k_{ij})$ 
so that   $\beta_1+\beta_2+\beta_3
 + \beta_{12}+\beta_{13}+\beta_{23} \leq 1$. 
 We infer that $$N_{(\sigma_1 f_1 f_3,\sigma_2 f_2 f_3)}
= 2^{\beta_0} \frac{m_1m_2m_3k_{12} k_{13} k_{23}}
{2^{\beta_1+\beta_2+\beta_3
+\beta_{12}+\beta_{13}+\beta_{23}} } .$$ 
Every divisor   $s\mid N_{(\sigma_1 f_1 f_3,\sigma_2 f_2 f_3)}$
therefore takes the shape  $s=2^\alpha e_1 e_2 e_3 e_{12} e_{13} e_{23}$
where  $$ 0\leq \alpha \leq \beta_0,\quad
 e_i \mid  m_i/ 2^{\beta_i},\quad
e_{ij} \mid k_{ij}/2^{\beta_{ij} }.$$
Define  $e_1^*,e_2^*, e_3^*$ 
 via $e_ie_i^*= m_i/ {2^{\beta_i} }$ 
and note that $e_{12} e_{13} e_{23} e_1e_1^* e_2e_2^*e_3 e_3^*$ is odd.
Making the substitutions $s=2^\alpha e_1 e_2 e_3 e_{12} e_{13} 
 e_{23}$ and $t_i= \lambda s_i^2  k_{ij} k_{ih} 2^{\beta_i } e_i e_i^* $,
 where $\{1,2,3\}=\{i,j,h\}$ and $k_{ij}:=k_{ji}$ in case $i>j$,
concludes the proof.\end{proof}

\begin{lemma}\label{eq:albinoni}
The product over $p$ in the definition of
$\c C(\b y)$  
in Lemma \ref{lem:tentoglou}
equals 
\begin{align*}
& \!\!\!\!\! (-1)^{\frac{1}{4} \sum_{i<j}(e_i-1)(e_j-1) }
 \l(\frac{\sigma_2 2^{\beta_2+\beta_3 } k_{12} k_{13}   e_2^*   e_3^*
}{e_1}\r)\l(\frac{\sigma_1 2^{\beta_1+\beta_3 } k_{12} k_{23}   e_1^*   e_3^*
}{e_2}\r)\l(\frac{-\sigma_1\sigma_2 2^{\beta_1+\beta_2 } k_{13} k_{23}   e_1^*   e_2^*
}{e_3}\r) \\ \times &  \l( \frac{-\sigma_1\sigma_2 2^{\beta_1 + \beta_2} k_{13} k_{23} 
e_1 e_1^* e_2 e_2^* }{e_{12}}\r)
\l( \frac{\sigma_1 2^{\beta_1 + \beta_3 } k_{12} k_{23} 
e_1 e_1^* e_3 e_3^*}{e_{13}}\r)
\l( \frac{\sigma_2  2^{\beta_2 + \beta_3}
k_{12} k_{13} e_2 e_2^* e_3 e_3^* }{e_{23}}\r)\c F_2
,\end{align*} 
where $\c F_2=(\sigma_1 2^{\beta_1+\beta_3 } k_{12} k_{23} 
e_1e_1^* e_3 e_3^* ,\sigma_2 2^{\beta_2+\beta_3 } k_{12} k_{13}  
e_2e_2^* e_3 e_3^*)_{\Q_2}$ 
if $\alpha =1$ and else $\c F_2=1$.
  \end{lemma}
\begin{proof}  By~\eqref{eq:commonivy2} and the explicit formulas for the Hilbert symbol in \cite[Theorem 1 in Chapter III]{MR344216},
the contribution of primes $p\mid e_1$ equals  
$$\prod_{p\mid  e_1}
(\sigma_1 2^{\beta_1+\beta_3 } k_{12} k_{23}  e_1e_1^* e_3 e_3^* ,
\sigma_2 2^{\beta_2+\beta_3 } k_{12} k_{13}  e_2e_2^* e_3 e_3^*)_{\Q_p}=
\l(\frac{\sigma_2 2^{\beta_2+\beta_3 } k_{12} k_{13}  e_2e_2^* e_3 e_3^*}
{e_1}\r),$$ 
and   a symmetric  expression    holds  for $e_2$. 
The  
primes dividing $e_3$ contribute 
$$ \l(\frac{-\sigma_1 \sigma_2 2^{\beta_1+\beta_2  }    k_{13}    k_{23}  
e_1e_1^* e_2e_2^* }{e_3}\r) .$$  
Putting the contribution    from primes 
$p\mid e_1e_2e_3$ together yields 
$$ (-1)^{\frac{1}{4} \sum_{i<j}(e_i-1)(e_j-1) }
 \l(\frac{\sigma_2 2^{\beta_2+\beta_3 } k_{12} k_{13}   e_2^*   e_3^*
}{e_1}\r)\l(\frac{\sigma_1 2^{\beta_1+\beta_3 } k_{12} k_{23}   e_1^*   e_3^*
}{e_2}\r)\l(\frac{-\sigma_1\sigma_2 2^{\beta_1+\beta_2 } k_{13} k_{23}   e_1^*   e_2^*
}{e_3}\r) 
$$   
by quadratic reciprocity.
The   primes dividing $e_{12} e_{13} e_{23} $ contribute 
$$ \l( \frac{-\sigma_1\sigma_2 2^{\beta_1 + \beta_2} k_{13} k_{23} 
e_1 e_1^* e_2 e_2^* }{e_{12}}\r)
\l( \frac{\sigma_1 2^{\beta_1 + \beta_3 } k_{12} k_{23} 
e_1 e_1^* e_3 e_3^*}{e_{13}}\r)
\l( \frac{\sigma_2  2^{\beta_2 + \beta_3}
k_{12} k_{13} e_2 e_2^* e_3 e_3^* }{e_{23}}\r)
.$$  
Finally, the prime $p=2$  contributes  $\c F_2$. 
\end{proof}

\subsection{Using the large sieve} \label{ss:invoking} 
\begin{lemma}  \label{lem:1stcase}
 Fix any $\epsilon>0$ and let $\c C(\b y)$ be as in Lemma \ref{lem:tentoglou}. For any $y_1,y_2,y_3,\napoli\geq 1$, 
 the contribution of those $(e_1,e_1^*,e_2,e_2^*,e_3,e_3^*)$ that satisfy
\begin{equation}\label{eq:commsieveranges}
e_i^*\leq\napoli \text{ and }e_j\leq\napoli\quad\text{ for some }\quad i\neq j\in\{1,2,3\}
\end{equation}   towards the sum defining $\c C(\b y)$ is 
 $ \ll (y_1 y_2 y_3)^{1+\epsilon}  
 \max_i (\napoli/y_i)^{1/2} 
 $, where the implied constant depends only on $\epsilon$.  \end{lemma}
  \begin{proof}  
For ease of notation we consider here those $(e_1,\ldots,e_3^*)$ that satisfy $e_1^*,e_2\leq \napoli$, all other cases being analogous.
They contribute
  $$\ll    \sum_{\substack{    e_1^*,e_2\leq \napoli \\ e_3 e_3^* \leq y_3 }}
\sum_{s,t\in (\Z/8\Z)^*}   \bigg|
\sum_{\substack{   e_1\leq y_1/e_1^*, e_1\equiv s \md 8
\\  e_2^* \leq y_2/ e_2,   e_2^*\equiv t \md 8 
} }   a'(e_1) b'(e_2^*   ) \l(\frac{  e_2^*   }{e_1}\r) 
\bigg|,$$ where $a',b'$ are functions 
bounded in modulus by $1$,
which may depend, in addition, on $e_1^*,e_2,e_3,e_3^*,s,t$, as well as the values of $\lambda,\b s,\alpha$ and the $\beta_i,\beta_{ij},k_{ij},e_{ij}$ appearing in the definition of $\c C(\vy)$ in Lemma \ref{lem:tentoglou}. The crucial point is that $a'$ is independent of $e_2^*$ and $b'$ is independent of $e_1$. Indeed, 
the conditions in~\eqref{eq:common}-\eqref{eq:commonivy2}
can be written as separate conditions on $e_1$ and $e_2^*$
by using the fact that that   
$e_1,e_2^*$ are in fixed classes modulo $8$, odd, 
and their coprimality is ensured by the Kronecker symbol $\l(\frac{e_2^*}{e_1}\r)$.
The terms $a(\cdot), b(\cdot)$ in the definition of $\c C$
as well as various quadratic symbols 
from Lemma \ref{eq:albinoni}
that are separate functions 
of $e_1$ and $e_2^*$ can also be absorbed in the functions $a',b'$.
Lastly, the term $\c F_2$ depends only on $s,t$, and
$(-1)^{\frac{(e_1-1)(e_2-1)}{4}}$ is independent of $e_2^*$.
Absorbing the conditions 
$ e_1\equiv s \md 8$ and $ e_2^*\equiv t \md 8 $ into $a',b'$
allows us to apply Lemma~\ref{lem:rogier}.  This   yields 
the bound   $$\ll    \sum_{\substack{  
e_1^*,e_2\leq \napoli \\ e_3 e_3^* \leq y_3 }}
\bigg(\frac{y_1y_2 }{e_1 ^* e_2} \bigg)^\epsilon 
\bigg( \frac{y_1}{e_1^*}\frac{y_2^{1/2}}{e_2^{1/2}}
+\frac{y_1^{1/2}}{{e_1^*}^{1/2}}\frac{y_2}{e_2}
\bigg),$$ which is sufficient as the sum over $e_3,e_3^*$ is 
$\leq \sum_{m\leq y_3} \tau(m)\ll y_3^{1+\epsilon}$.\end{proof}

 \begin{lemma}  \label{lem:2ndcase} Fix any  
 $\epsilon>0$
 and let $\c C(\b y)$ be as in Lemma \ref{lem:tentoglou}. 
 For $y_1,y_2,y_3,\napoli\geq 1$, the contribution of those $(e_1,e_1^*,e_2,e_2^*,e_3,e_3^*)$ that satisfy
\begin{equation}\label{eq:commsieMONSTERA}
e_i^*>\napoli \text{ and }e_j>\napoli\quad\text{ for some }\quad i\neq j\in\{1,2,3\}
\end{equation}   towards the sum defining $\c C(\b y)$ is 
 $ \ll (y_1y_2y_3)^{1+\epsilon} \napoli^{-1/2+\epsilon} $, 
 where the implied constant depends only on $\epsilon$.\end{lemma}
\begin{proof} 
This is similar to the proof of Lemma \ref{lem:1stcase}, so we will be brief. Again 
we deal with the case $e_1^*,e_2> \napoli $, the other cases 
being similar. From the conditions inherent in the definition of $\c C(\b y)$ 
we have $ e_1\leq y_1/e^*_1<y_1/\napoli$ and $e_2^*<y_2/\napoli$. Thus, the 
contribution is 
$$\ll \sum_{\substack{ e_1<y_1/\napoli, e^*_2<y_2/\napoli 
\\ e_3 e_3^* \leq y_3 }} \sum_{s,t\in (\Z/8\Z)^*}   \bigg|
\sum_{\substack{ e^*_1\leq y_1/e_1, e_1^*\equiv s \md 8
\\  e_2\leq y_2/e^*_2, e_2\equiv t \md 8 } }  
a''( e_1^*  ) b''(  e_2  )  
\l(\frac{e_1^*}{e_2}\r)\bigg|
,$$ 
where the functions $a'',b''$ are again bounded by 
$1$ in modulus and  
capture the information from the definition of $\c C(\vy)$ and Lemma \ref{eq:albinoni} that depends on only one of $e_1^*,e_2$, as well as the conditions $e_1^*,e_2>\napoli$.
Alluding to Lemma~\ref{lem:rogier} leads to  the bound\[\ll  
\sum_{\substack{ e_1<y_1/\napoli, e^*_2<y_2/\napoli \\ e_3 e_3^* \leq y_3 }} 
\bigg(\frac{y_1y_2 }{e_1  e_2^*} \bigg)^\epsilon \bigg(
\frac{y_1}{e_1 }\frac{y_2^{1/2}}{{e_2^*}^{1/2}}
+\frac{y_1^{1/2}}{{e_1 }^{1/2}}\frac{y_2}{e_2^*}
\bigg).\qedhere\]\end{proof}
Before proceeding, we note that the   terms remaining in 
the sum defining $\c C(\b y )$ after excluding 
every case in~\eqref{eq:commsieveranges} 
and~\eqref{eq:commsieMONSTERA} satisfy  
 \begin{equation} \label{eq:cwhat is left!}   
e_1^*,e_2^*,e_3^*\leq \napoli \ \ \textrm{ or } \ \ 
e_1,e_2,e_3\leq \napoli.  \end{equation} 

\subsection{Proof of Theorem~\ref{thm:bitethemedals}} \label{ss:prfthrm}  
By Lemma \ref{lem:delta_rand_hat}, we need to estimate the sum in Lemma \ref{lem:tentoglou}.

We first truncate the sum over $k_{ij}$ in 
Lemma~\ref{lem:tentoglou}. Let $\c K\geq 1 $.
Then, for every fixed $\epsilon>0$ 
the contribution of terms with $k_{12}>\c K$
is   
\begin{equation}\label{eq:casanatense}
\ll    \sum_{\substack{\lambda \in \N \\ \b s \in \N^3} } 
\sum_{\substack{k_{23}, k_{23} \in \N\\ 
 k_{12}>\c K  }}  
 \frac{(x_1x_2x_3)^{1+\epsilon/2} \tau(k_{12}) \tau(k_{23})\tau(k_{23}) }{ (s_1 s_2s_3k_{12} k_{13} k_{23})^2 \lambda^3  }  \ll \frac{(x_1x_2x_3)^{1+\epsilon}}{\c K^{1-\epsilon}} 
\end{equation}  
and the same bound holds for the terms 
 with $\max \{ k_{13},k_{23} \}>\c K$.
To facilitate our notation, we tacitly assume that 
$\{i,j,h\}=\{1,2,3\}$ whenever these indices 
appear, and $k_{ij}=k_{ji}$ when $i>j$. 
By Lemma~\ref{lem:1stcase}, 
the terms $e_i,e_i^*$ satisfying one of the cases 
in~\eqref{eq:commsieveranges}
contribute  the following towards the sum, 
\begin{align*}
&\ll  \sum_{\substack{\lambda \in \N \\ \b s \in \N^3} } 
\sum_{\substack{ k_{12}, k_{23}, k_{23} \in \N\\  }}
\tau(k_{12}) \tau(k_{23})\tau(k_{23})
 \l(\frac{x_1x_2x_3}{  (s_1 s_2 s_3 k_{12} k_{13}k_{23} )^2 \lambda ^3  } \r)^{1+\epsilon} \sum_{i=1}^3 
\frac{(\napoli s_i^2 \lambda k_{ij} k_{ih})^{1/2} }{  x_i   ^{1/2} } \\
&\ll (x_1x_2x_3)^{1+\epsilon }
   \frac{\napoli^{1/2} }{\min_i\{x_i^{1/2}\} } 
.\end{align*}

By Lemma~\ref{lem:2ndcase}
the terms  satisfying one of the cases 
in~\eqref{eq:commsieMONSTERA}
contribute\begin{align*} &\ll   \napoli^{-1/2+\epsilon} 
 \sum_{\substack{\lambda \in \N \\ \b s \in \N^3} } 
\sum_{\substack{ k_{12}, k_{23}, k_{23} \in \N\\  }}
\tau(k_{12}) \tau(k_{23})\tau(k_{23})
 \l(\frac{x_1x_2x_3}{  (s_1 s_2 s_3 k_{12} k_{13}k_{23} )^2 \lambda ^3  } \r)^{1+\epsilon}  
 \ll  \frac{(x_1x_2x_3)^{1+\epsilon} }{   \napoli^{1/2-\epsilon}  }
.\end{align*}  Recalling~\eqref{eq:cwhat is left!}
we infer that the left-over terms   satisfy
$$\max\{k_{12},k_{13}, k_{23}  \} \leq \c K
\ \ \textrm{ and } \ \ 
\min\{ e_1^* e_2^* e_3^*, e_1e_2e_3\}\leq \napoli^3.$$
By~\eqref{eq:common}  there are no   left-over terms 
as long as $\c K$ and $\napoli$ are  
are chosen suitably.
Indeed, if 
$e_1^* e_2^* e_3^* \leq \napoli^3$ then by the second assertion
in~\eqref{eq:common} we deduce  
$$ \frac{ z  }{2\c K^3  }\leq \frac{z 2^\alpha   e_{12} e_{13} e_{23}}
{k_{12} k_{13} k_{23}2^{\beta_0-\sum_{i,j} \beta_{ij}} }<  \napoli^3.$$
Similarly, if $e_1  e_2  e_3  \leq \napoli^3$ 
then by $e_{ij}\leq k_{ij}\leq \c K$ and
the first assertion
in~\eqref{eq:common} we get   
$$ \frac{  z }{ 2 \c K^3 }\leq   
\frac{z}{2^\alpha e_{12} e_{13} e_{23}}<  \napoli^3.$$
We now       define $\c K=\c K(z,\napoli)$  through  
$   2(\c K \napoli)^3 =z$
. Then the last two inequalities 
cannot hold, thus, there are indeed no left-over terms. The proof 
concludes by noting  that the resulting bound 
with this particular choice of $\napoli,\c K$  becomes 
$$\ll \frac{x_1x_2x_3}{(z^{1/3}/\napoli )^{1-\epsilon}} +
(x_1x_2x_3)^{1+\epsilon }
   \frac{\napoli^{1/2} }{\min x_i^{1/2} } 
   +\frac{(x_1x_2x_3)^{1+\epsilon} }{   \napoli^{1/2-\epsilon}  }
.$$ Setting $\napoli=z^{2/9}$ furnishes the error term claimed in 
Theorem~\ref{thm:bitethemedals}. \qed

\subsection{Proof of Theorem~\ref{thm:olympiadepergolesi}} 
\label{s:pereli} It is straightforward  to  modify
the statements and proofs of  
Lemmas \ref{lem:delta_rand_hat}, and
Lemmas \ref{lem:tentoglou}-\ref{lem:2ndcase}
by omitting the terms  
$x_3,t_3,\lambda,s_3,\beta_3,\beta_{i3},k_{i3},e_{i3},e_3,e^*_3$.
In conclusion, 
we may pass from $\dran$ to $\dranh$ at the cost of an error $\ll (x_1x_2)^{1/2+\eps} z$, 
the terms satisfying  
$e^*_1,e_2\leq \napoli$ or 
$e^*_2,e_1\leq \napoli$
contribute   
$ \ll (y_1 y_2)^{1+\epsilon} \max_i (\napoli/y_i)^{1/2} $ 
to the modified $\c C(\b y)$,  
and the terms satisfying 
$e^*_1,e_2>\napoli$ or
$e^*_2,e_1>\napoli$ 
contribute at most
 $ \ll (y_1y_2)^{1+\epsilon} \napoli^{-1/2+\epsilon} $.

With only four variables $e_1,e_1^*,e_2,e_2^*$, we can not conclude immediately that the analogue of \eqref{eq:cwhat is left!} holds in all the remaining cases, as it may also happen, e.g., that $e_1,e_1^*\leq\napoli$ and $e_2,e_2^*>\napoli$. Hence,
let us bound the contribution of the cases with
$e_1,e_1^* \leq \napoli$ or, 
analogously,
$e_1,e_2^*\leq \napoli$. The former 
makes a contribution towards the modified
$\c C(\b y)$ that is  
$$\ll \sum_{\substack{e_1,e_1^*,e_2,e_2^* \in \N \\ e_1,e_1^* \leq \napoli\ ,e_2 e_2^* \leq y_2 } }1 \ll \napoli^{2} y_2^{1+\epsilon},
$$
while the latter similarly makes a 
contribution of modulus  
$\ll \napoli^2 y_1^{1+\epsilon}$.

The   terms remaining in 
the  modified $\c C(\b y )$ after excluding  
all the above cases
satisfy  $e_1^*,e_2^*\leq \napoli$
or $e_1,e_2\leq \napoli$, 
analogously to \eqref{eq:cwhat is left!}.

The argument in \eqref{eq:casanatense} can be carried out similarly and gives an error term bounded by  $\ll (x_1x_2)^{1+\epsilon}/\c K^{1-\epsilon}$.
The analogue of  Lemma~\ref{lem:1stcase}
gives a bound 
$\ll (x_1x_2)^{1+\epsilon }
\napoli^{1/2}/\min_i\{x_i^{1/2}\}  $.
Furthermore, the analogue of  
Lemma~\ref{lem:2ndcase} results in a contribution $\ll (x_1x_2)^{1+\epsilon } 
\napoli^{-1/2+\epsilon}$. 
Finally, the newly excluded terms satisfying 
$e_1,e_1^* \leq \napoli$ or $e_1,e_2^* \leq \napoli$ 
contribute at most   
$$\ll \napoli^{2}\sum_{\substack{ k_{12} \leq \c K \\s_1,s_2\in \N } } \tau(k_{12})\left( \left( 
\frac{x_1}{s_1^2 k_{12} }
\right)^{1+\epsilon}
+\left(
\frac{x_2}{s_2^2 k_{12} }\right)^{1+\epsilon}\right)
\ll \napoli^{2}  (\max_i x_i)^{1+\epsilon }
.$$

In the remaining cases with
$e_1^*,e_2^*\leq \napoli$
or $e_1,e_2\leq \napoli$,  
the analogue of \eqref{eq:common}
can be used to deduce that 
$z<2\c K \napoli^2$.
Setting
$\c K:=z/(2\napoli^2)$ 
renders these cases impossible and
gives the overall bound 
$$ \ll (x_1x_2)^{1+\epsilon}\l\{
\frac{z}{\sqrt{x_1x_2}} +
\frac{\napoli^{2-\epsilon}}{z^{1-\epsilon}}+
\frac{\napoli^{1/2}}{\min_{i}\{x_i^{1/2}\}}+
\frac{1}{\napoli^{1/2-\epsilon}}  
+\frac{\napoli^{2}}{\min_{i}\{x_i\}}
\r\}.$$ 
Taking 
$\napoli:=z^{2/9}$ concludes the proof.\qed

\section{$L^2$-estimate via lowering moduli}\label{sec:dispersion}  The main goal of this section is to prove Theorem~\ref{thm:L^2theorem}.
  \begin{itemize}\item In \S\ref{s:passinL^2mean} we pass from $\updelta$ to 
a model
$\ddetf$ in $L^2$-mean. 
\item In \S\ref{s:passfromFtolocdens} we pass from sums 
over $\b F$
to character sums  involving the symbol 
$(\cdot,\cdot)_p'$.
\item In \S\ref{s:charsums} we study the 
character sums.
\item In \S\ref{s:levellower}
we lower the level and match sum conditions.  
\item In \S\ref{s:passnnntolocals} we pass from sums over 
$\boldsymbol{n}, \boldsymbol{n}'$ to integrals.  
\item In \S\ref{s:smoothings} we   use   anatomy of  integers in an adelic setting 
to recover $\Sing(\vF)$.  
\end{itemize}

\subsection{Sketching the ideas}\label{s:sketching_ideas}
Recall from Definition \ref{def:dddddd}
that $$ \ddet(t_1,t_2)= 
(1+(t_1,t_2)_{\infty}')\sum_{\substack{ s\leq z \\ s \textrm{ square-free}}}\prod_{p\mid s }(t_1,t_2)'_p.$$ 
 When $s$ is fixed the function $\prod_{p\mid s }(t_1,t_2)'_p$ is not
periodic in $t_i$, however, it is periodic for
$t_i$ with  
fixed
$p$-adic valuations at primes $p\mid s$. We therefore restrict the sum to those terms with small
valuations: 
for $\b t \in (\Z\setminus \{0\})^2,z,T\geq 1 $ we let 
$$\ddetf(t_1,t_2) :=  (1+(t_1,t_2)_\infty')
\sum_{\substack{s\leq z \\ \eqref{def:PPPstt}} }  \mu(s)^2 
\prod_{p\mid s }(t_1,t_2)'_p ,$$ where 
 the sum over $s\leq z$ is subject to the condition
 \begin{equation}  
\label{def:PPPstt}\prod_{p\mid s} 
p^{\max\{v_p(t_1), v_p(t_2) \}}\leq T .\end{equation} 
We rewrite this definition as follows: take  
$ r_i:=\prod_{p\mid s}p^{v_p(t_i)}$ so that~\eqref{def:PPPstt} 
becomes $[r_1,r_2]\leq T$, where we use the notation $[r_1,r_2]:=\lcm(r_1,r_2)$. Thus, 
\begin{equation}\label{eq:onthetraintothegreekembassy}
\ddetf( \b t )= (1+\hs{t_1}{t_2}_\infty')
  \sum_{s\leq z}  \mu(s)^2 
\sum_{\substack{\b r \in \N^2, [r_1,r_2]\leq T 
\\p\mid r_1r_2\Rightarrow p\mid s}} 
\prod_{p\mid s }(t_1,t_2)'_p
\mathds 1_{\forall i=1,2:\ v_p(t_i)=v_p(r_i)}.\end{equation}
This formula is also well defined in case $t_1t_2=0$, where it gives $\ddetf(\b t)=1$. Recalling the definition of $S_\b F(x)$ 
in~\eqref{def:Shaf_f(x)007007}, the analogous sum for  
$\ddetf$ is   \begin{equation}\label{def:Shaf_f(x)} 
\hS_\b F(x):= \sum_{\substack{\b n \in \Z^2 \cap x\c B \\ \gcd(n_1,n_2)=1}}
\ddetf(\Phi_1(\b n), \Phi_2(\b n)).\end{equation} In \S\ref{s:passinL^2mean}
we will use the tools developed in \S\S\ref{s:summability}-\ref{sec:siegel_walfisz} 
to bound $\sum_\b F| S_\b F(x)-\hS_\b F(x)|^2$. 
After that, the next goal  is to 
bound  
$\sum_\b F| \hS_\b F(x)-x^2\hSing(\b F)|^2$, where 
\begin{equation}\label{deg...singularseries...}
\hspace{-0.2cm}\hSing(\b F):=\frac{\omega_{\infty}(\b F)}{\zeta(2)}\hspace{-0.2cm}
\sum_{\substack{s\in \N\\P^+(s)\leq L}} \hspace{-0.3cm}
\mu(s)^2\hspace{-0.3cm}\sum_{\substack{[r_1,r_2]\leq T_0 
\\p\mid r_1r_2\Rightarrow p\mid s}} \ \  
\hspace{-0.2cm}\int \limits_{\substack{\vt=(\vt_p)_p\in
\prod_{p\mid s}\Z_p^2\smallsetminus 
p\ZZ_p^2\\v_p(\Phi_i(\vt_p))=v_p(r_i)}}
\prod_{p\mid s }(\Phi_1(\vt_p),\Phi_2(\vt_p))'_p
\left(1-\frac{1}{p^2}\right)^{-1}  
 \mathrm d \b t, 
 \end{equation} 
 $L$ is as in \eqref{eq:L_growth_bound_mult},  
$T_0$ will be chosen to grow with $H$ significantly slower than $T$ and $x$,
$\omega_\infty(\b F)$ is 
defined in \eqref{eq:def_omega_infty_mult},
and $P^+$ denotes the largest prime divisor. To this end,
we open   
the square and use~\eqref{eq:onthetraintothegreekembassy} 
to get   expressions roughly of shape 
\begin{align*}
&\sum_\b F \hS_\b F(x)^2 &=&  &\sum_{\b n}\sum_{\b n'} 
&\sum_{\substack{ s\leq z \\ s' \leq z 
\\ \textcolor{white}{P^+(s')\leq L} }}  
&\sum_{ r_i,r'_i} 
  &\sum_{\b F}\prod_{p\mid s} (\Phi_1(\b n), \Phi_2(\b n) )'_p
\prod_{p\mid s'} (\Phi(\b n'), \Phi_2(\b n') )'_p,
\\ 
&\sum_\b F\hS_\b F(x)x^2\hSing(\b F) &=& &x^2 
\sum_{\b n}\int\limits_{\b t'}
&\sum_{\substack{ s\leq z \\P^+(s')\leq L
\\ \textcolor{white}{P^+(s')\leq L} }} 
&\sum_{ r_i,r'_i} 
 &\sum_{\b F}\prod_{p\mid s} (\Phi_1(\b n), \Phi_2(\b n) )'_p
\prod_{p\mid s'} (\Phi(\b t'), \Phi_2(\b t') )'_p,
\\
&\sum_\b F x^4\hSing(\b F)^2 &=& &x^4\int \limits_{\b t}\int \limits_{\b t'}  
&\sum_{\substack{P^+(s)\leq L\\P^+(s')\leq L}} 
&\sum_{ r_i,r'_i}  
 &\sum_{\b F}\prod_{p\mid s} (\Phi_1(\b t), \Phi_2(\b t) )'_p
\prod_{p\mid s'} (\Phi(\b t'), \Phi_2(\b t') )'_p.
\end{align*}
The coefficients of $\b F$ range through an interval of size 
comparable to $H$ and,
due to the fixed $p$-adic valuations in the Hilbert symbols,
the function $\prod_{p\mid s} (\Phi_1(\b n), \Phi_2(\b n) )'_p
\prod_{p\mid s'} (\Phi(\b n'), \Phi_2(\b n') )'_p$
will be 
periodic in the coefficients of 
$\b F$ with a 
modulus $K$ of size roughly $ss'r_1r_2r'_1r'_2$. Due to the size 
bounds on $s,s',r_i,r'_i$, the modulus $K$ is smaller than the interval size $H$.
In \S\ref{s:passfromFtolocdens} we use this to replace each $\sum_\b F$ in the 
right-hand side by a corresponding 
local
sum $\c X$ modulo $K$ involving
the 
analytic
Hilbert symbol $(\cdot,\cdot)_p'$.

Up to acceptable error terms the expressions thus become,
roughly,
$$\begin{aligned}
&\sum_{\substack{ s\leq z \\ s' \leq z 
\\ \textcolor{white}{P^+(s')\leq L} }} 
&\sum_{ r_i,r'_i} \hspace{0.5cm} 
&\sum_{\b n}\sum_{\b n'}&\c X(\b r;\b s;\b n,\b n'),
\\ 
&\sum_{\substack{ s\leq z \\P^+(s')\leq L
\\ \textcolor{white}{P^+(s')\leq L} }} 
&\sum_{ r_i,r'_i}\hspace{0.5cm} 
&x^2\sum_{\b n}\int\limits_{\b t'}&\c X(\b r;\b s;\b n,\b t'),
\\
&\sum_{\substack{P^+(s)\leq L\\P^+(s')\leq L}} 
&\sum_{ r_i,r'_i} \hspace{0.5cm} 
&x^4\int \limits_{\b t}\int \limits_{\b t'} &\c X(\b r;\b s;\b t,\b t').\end{aligned}$$
We must now   show that these three expressions match up asymptotically. This would be straightforward
 if we could use periodicity modulo $K$ to replace the sums over $\b n,\b n'$ 
by the corresponding   integrals. The problem is that  
\S\S\ref{s:summability}-\ref{sec:siegel_walfisz} 
 require $z$ to be
substantially larger than $x$, and since $K$ exceeds $s$ (whose typical size is 
$z$), the interval size $x$ is much smaller than the modulus $K$.

It is at this point that we make use of the fact that the
analytic Hilbert symbol 
has  
average zero.
In \S\ref{s:charsums} we will use it to show that
the character sums $\c X$  
vanish in many cases.
This allows  us to dispose  of 
most  $s,s', r_i, r'_i$ and only keep those for which the corresponding modulus $K$ 
is lower than $x$. Furthermore, it enables us to move from conditions of type 
$s\leq z$ to $P^+(s)\leq L$. Both of these steps will be carried out 
in \S\ref{s:levellower}.  Then in \S\ref{s:passnnntolocals} we use the new lower 
modulus to replace sums over $\b n,\b n'$ by integrals.
Finally, in \S\ref{s:smoothings} we 
develop
adelic analogues 
of 
anatomy of integers estimates    to bound 
$\sum_\b F| \mathfrak S(\b F)-\hSing(\b F)|^2$.

\subsection{Passing from $\updelta$ to $\ddetf$ in $L^2$-mean}
\label{s:passinL^2mean} We   first 
prove a variant of Theorem~\ref{thm:pergolesiguglielmo} in which $\dran$ is replaced by 
$\updelta-\ddetf$. The first step is the following lemma, in which we denote 
 $ s':=s_1 \cdots s_{m_1}$, $t':= t_1 \cdots t_{m_2}$, $r':= r_1 \cdots r_{m_3}$ and
 $$\c P_s(a,b):=\prod_{p\mid s}p^{\max\{v_p(a),v_p(b)\}}.$$

\begin{lemma}\label{lem:lostbaggage}
Let $m_1,m_2>0$ and $m_3\geq 0$ be integers. Fix any $0<\epsilon<1$. For any $x_1,\ldots, x_{m_1}, y_1,\ldots, y_{m_2}, z_1,\ldots, z_{m_3}\geq 1$ and
$z,T \geq 1$ we have   $$
\sum_{\substack{ \forall i, 1\leq s_i \leq x_i\\
\forall i, 1\leq t_i \leq y_i\\\forall i, 1\leq r_i \leq z_i}} 
\sum_{\substack{s\leq z\\ s\mid  2 s' t'r'\\  
\c P_s( s' r' , t'r' ) > T }} \mu(s)^2 \ll 
\bigg(  \prod_{i=1}^{m_1} x_i \prod_{i=1}^{m_2} y_i
  \prod_{i=1}^{m_3} z_i \bigg)  \frac{z}{T^{1-\epsilon}}
 ,$$ where the implied constant depends only on $\epsilon$ and $m_1,m_2,m_3$.
\end{lemma} 
\begin{proof}  Factor $s_i=a_i b_i$, $t_i=a'_i b'_i$, 
$r_i=a''_i b''_i$, where $b_i,b'_i,b''_i $ are coprime to $s$ 
and   all prime divisors of $a_i a'_ia''_i$ divide $s$.
Using $\tau_m$ to denote the $m$-fold divisor function,
we obtain the upper bound 
\begin{align*}
&\sum_{\substack{ s\leq z } }\mu(s)^2 
\sum_{\substack{ \b a \in \N^{m_1},
\b a' \in \N^{m_2}, \b a'' \in \N^{m_3} \\ 
 p\mid 2s \iff  p\mid  2\prod a_i a'_i a''_i  \\
  \c P_s(\prod a_i a''_i,\prod a'_i a''_i  ) >T
  }} \sum_{\substack{ \forall i, 1\leq b_i \leq x_i/a_i\\
\forall i, 1\leq b'_i \leq y_i/a'_i\\
\forall i, 1\leq b''_i \leq z_i/a''_i }} 1 
\\  \leq &\bigg(  \prod_{i=1}^{m_1} x_i \prod_{i=1}^{m_2} y_i
  \prod_{i=1}^{m_3} z_i \bigg)
\sum_{\substack{ s\leq z } }\mu(s)^2 \hspace{-0.7cm}
\sum_{\substack{a, a', a'' \in \N \\ 
 p\mid 2s \iff  p\mid   2a  a' a''   \\
  \c P_s( a  a'' ,  a'  a''   ) >T }} \hspace{-0.5cm}
  \frac{\tau_{m_1}(a)\tau_{m_2}(a')\tau_{m_3}(a'') }
  {   a  a'  a''  }= :\bigg(  \prod_{i=1}^{m_1} x_i \prod_{i=1}^{m_2} y_i
  \prod_{i=1}^{m_3} z_i \bigg)\Xi, 
\end{align*}
say, where we took $a:=\prod a_i$, $a':=\prod a'_i$ and $a'':=\prod a''_i$.
Clearly,
$$   \c P_s( a  a'' ,  a'  a''   ) \text{ divides } 
\prod_{p}
p^{\max\{v_p(a     a'' ),v_p(a'  a'' )\}}, \text{ which divides } a a' a''.$$ 
Let $g:=\tau_3\cdot \prod_{i=1}^3 \tau_{m_i}$ and $n:= a a' a''$  so that 
$$\Xi \leq \sum_{ s\leq z }\mu(s)^2
\sum_{ \substack{ n>T \\ p\mid 2n \iff p\mid 2s }}  \frac{g(n) }{n}
=\sum_{ n>T }  \frac{g(n) }{n} 
\sum_{ \substack{ s\leq z \\ p\mid 2n \iff p\mid 2s }}   \mu(s)^2\leq 2
\sum_{ \substack{ n>T  \\\textrm{radical}(n)\leq 2z  }}  \frac{g(n) }{n} 
.$$Letting $r:=\textrm{radical}(n)$ we use Rankin's trick to 
obtain $$\Xi \leq \sum_{r\leq 2z}\mu(r)^2
\sum_{ \substack{ n>T  \\ \textrm{radical}(n)=r  }}  
\frac{g(n) }{n} \leq T^{-1+\eps}\sum_{r\leq 2z}\mu(r)^2
\sum_{\substack{ 
n\in\NN\\
\textrm{radical}(n)=r}}
\frac{g(n) }{n^{\eps}}.$$
Since $g(n) \ll n^{\epsilon/2}$, the sum over $n$ in the right-hand side is  
$$\ll \sum_{\substack{
n\in\NN\\
\textrm{radical}(n)=r}} n^{-\eps/2} =\prod_{p\mid r }
\frac{p^{-\epsilon/2}}{1-p^{-\epsilon/2}}\ll 1,$$ thus, $\Xi \ll 
T^{-1+\eps} z$. 
\end{proof} 
Recall the notation \eqref{def:casanatensenotation}.
\begin{corollary}\label{cor:euredc}
Let $m_1,m_2>0$ and $m_3\geq 0$ be integers. 
Fix any $\epsilon\in (0,1)$ and 
$\sigma_1,\sigma_2\in \{-1,1\}$. 
Assume that 
$ a:\N^{m_1} \to \mathbb C$, $b:\N^{m_2} \to  \mathbb C$ and $c:\N^{m_3} \to  \mathbb C$
are arbitrary functions bounded in modulus by $1$.
For $x_1,\ldots ,x_{m_1},
y_1,\ldots, y_{m_2},z_1,\ldots ,z_{m_3}\geq 1$ and $z,T\geq 1 $ we have
\begin{align*} 
&\sum_{\substack{ \forall i, 1\leq s_i \leq x_i\\
\forall i, 1\leq t_i \leq y_i\\ \forall i, 1\leq r_i \leq z_i}}
(\updelta-\ddetf)\bigg(
\sigma_1  \prod_{i=1}^{m_1} s_i \prod_{i=1}^{m_3} r_i ,
\sigma_2 \prod_{i=1}^{m_2} t_i \prod_{i=1}^{m_3} r_i \bigg)
a(\b s ) b(\b t ) c(\b r) \ll \\ & 
 (\c X \c Y \c Z) ^{1+\epsilon}
\l( \frac{1}{z^{1/9} }+ \frac{z^{1/9} }{\sqrt{
\min\{
\c X, \c Y, \c Z\}}} + \frac{z}
{\sqrt{
\c X \c Y \c Z}}
    + \frac{\one_{m_3=0}z^{4/9}}{\min \{  \c X,\c Y\}}
+\frac{z}{T^{1-\epsilon}} \r)   ,
\end{align*}
where the implied constant depends only   on $\epsilon, m_1,m_2,m_3$, and 
  $\c Z $ 
  is to be ignored 
  in case $m_3=0$.
\end{corollary}

\begin{proof}
The proof follows by 
combining Theorem~\ref{thm:pergolesiguglielmo}, Lemma~\ref{lem:lostbaggage}  and  
the estimate
\[(\updelta-\ddetf)(\b t )- \dran(\b t )=\ddet(\vt)-\ddetf(\vt)\ll \sum_{\substack{ s\leq z\\ s\mid 2t_1 t_2\\  \c P_s(\b t ) > T }} \mu(s)^2  .\qedhere\] 
\end{proof}

Recall the setum from \S\ref{sec:HP_theorems} and \S\ref{s:sketching_ideas}.
\begin{proposition}
\label{prop:circle_method result}
Fix $\lambda,\eps\in (0,1)$.
For any $H\geq z,T\geq x\geq H^\lambda$, 
we have  
$$\frac{1}{|\cFZ(H)|} \sum_{\b F \in \cFZ(H)}
|S_\b F(x)-\hS_\b F(x)|^2
 \ll \frac{x^4}{(\log H)^2}+H^{\eps} x^{2d+4}
 \max \l\{
 z^{-2/9},\frac{z^{2/9}}{ H}, 
  \frac{z^2}{T^{2} } 
 \r\},$$
where the implied constant depends at most on $\eps,\lambda,m_1,m_2,m_3$ and the $d_{ij}$. 
\end{proposition}

\begin{proof}
The statement is clear if $H\ll 1$, so  we may assume that $H$ is sufficiently large.
We employ Corollary~\ref{cor:heat:kernel} with $m=m_1+m_2+m_3$, the $d_i$ taken to be the values of the $d_{ij}$,
$$f(\b s,\b t, \b r)=(\updelta-\ddetf)
\bigg(
  \prod_{i=1}^{m_1} s_i \prod_{i=1}^{m_3} r_i ,
 \prod_{i=1}^{m_2} t_i 
 \prod_{i=1}^{m_3} r_i \bigg),\quad a(\b n )=
 \mathds 1_{x\c B}(\b n )\mathds 1_{\gcd(n_1,n_2)=1}.
$$ 
Due to
$$ |(\updelta-\ddetf)(\vt)|\leq |\updelta(\vt)| + |\ddetf(\vt)|\leq 2\sum_{s\mid N_{\b t } }\mu(s)^2 
\ll
\tau(N_\vt),
$$ 
we can take 
$C=1$ in~\eqref{def:weaksigwalfs}. We bound
the size of $E_f(\b x,\b q)$ defined in \eqref{def:porporapassio}
by 
splitting into cases according to the signs of 
$s_i,t_i,r_i$ and in each case using  
Corollary~\ref{cor:euredc}
with suitable $\sigma_1,\sigma_2$
and the functions $a,b,c$ involving the exponentials 
$\mathrm e^{\pm 2\pi i b_k t_k/q_k}$ and bounds $t_k\leq v_k$ in the definition of $E_f$.
This yields the bound 
$$  E_f(((1+d_{ij})x^{d_{ij}}H)_{i,j};\b q ) 
\ll H^{m+\eps} x^d 
\max\l\{  \frac{1}{z^{1/9} }, \frac{z^{1/9} }{\sqrt{H }}
,\frac{z}{H^{m/2}},
\frac{z^{4/9}}{H},
\frac{z}{T } \r\}=:H^{m+\eps} x^d \c M.$$ 
Note that the cases where one of the $s_i,t_i,r_i$ is zero trivially make a harmless contribution $\ll H^{m-1+\eps}x^d$ to this bound. The total error term from Corollary \ref{cor:heat:kernel} is
$$ \ll    (\log H)^{4m} 
\frac{x^{4+2d }
}{\xi \mathrm e^{\pi \xi^2}}
+  (\log H)^{\gamma_1 +  2^{2(1+ 2\gamma_0)}}
\frac{x^4}{\xi_0^{  1/(2\c D)}}
+(\xi \xi_0)^{2m}  x^4 
(H^{\eps} x^d \c M)^2.$$
Taking $\xi=\xi_0=(\log H)^N$, for a sufficiently large fixed $N$,
shows that the error term is 
\[\ll \frac{x^4}{(\log H)^2}+H^{3\eps} x^{2d+4}\c M^2. \qedhere\]
  \end{proof}

\subsection{
Passing
from sums over $\b F$ to local densities} \label{s:passfromFtolocdens}

For square-free $s\in\NN$, we define the adelic sets
$$\Omega_{s}^0:=\prod\limits_{p\mid s}
\left(\Z_p^2\smallsetminus p\ZZ_p^2\right),\quad \Omega_s := (\RR^2\smallsetminus \{\vzero\})\times\Omega_{s}^0,\quad\text{ and }\quad \Omega_s^\cB := \cB\times\Omega_s^0\subseteq \Omega_s,
$$
writing elements of $\Omega_s$ in the form $\vt=(\vt_\infty,\vt_0)$, with $\vt_\infty\in\RR^2\smallsetminus\{\vzero\}$ and $\vt_0=(\vt_p)_{p\mid s}\in\Omega_s^0$. Then every $\vn=(n_1,n_2)\in\ZZ^2\smallsetminus\{0\}$ with $\gcd(n_1,n_2)=1$ can be considered naturally as an element of $\Omega_s^0$ and of $\Omega_s$ by embedding it diagonally.

For square-free $s,s'\in\NN$ and $r_1,r_2,r_1',r_2'\in\NN$ satisfying $p\mid r_1r_2\Rightarrow p\mid s$ and $p\mid r_1'r_2'\Rightarrow p\mid s'$, we define the modulus 
\begin{equation}\label{def:Klcm} K:=K(\b r; \b s)=
4^{\max\{v_2(s),v_2(s')\}}\prod_{p\mid ss'}
p^{\max\{v_p(r_1), v_p(r_2), v_p(r_1'), v_p(r_2')\}+1}
.\end{equation}
It has the crucial property that for all $p\mid ss'$ and $\vt_p=(t_1,t_2)\in\ZZ_p^2$ with fixed valuations $v_p(t_i)=v_p(r_i)$ for $i=1,2$, the value of the Hilbert 
symbols $\hs{t_1}{t_2}_{\QQ_p}$ and $\hs{t_1}{t_2}'_p$
depends only on $\vt_p\bmod{p^{v_p(K)}}$. Hence, with 
$$\cF_{\Z/K\Z}:= \{\b F=(F_{ij})\where F_{ij} \in (\Z/K\Z)[t_1,t_2] \textrm{ form of degree } d_{ij}\ \forall i,j\},$$
the value of the product 
\begin{equation}\label{eq:HS_product}
\prod_{p\mid s}\hs{\Phi_1(\vt_p)}{\Phi_2(\vt_p)}_p'\prod_{p\mid s'}\hs{\Phi_1(\vt_p')}{\Phi_2(\vt_p')}_p'
\end{equation}
is well defined for all $\vF\in \cF_{\Z/K\Z}$ (yielding $\Phi_1,\Phi_2$ by \eqref{eq:def_phi_i}), $\vt_0\in\Omega_s^0$ and $\vt_0'\in\Omega_{s'}^0$ that satisfy 
\begin{equation}\label{eq:Sigma_conditions}
v_p( \Phi_i(\b t_p) )=v_p(r_i), \ \ \ v_{p'}( \Phi_i(\b t'_{p'}) )=
v_{p'}(r'_i)\quad\text{ for } i=1,2\ \text{ and primes }\ p\mid s,\ 
\ p'\mid s'.\end{equation} 
This allows us to define for $\b s=(s,s')$, $\b r=(r_1,r_2,r_1',r_2')$ as above, $\vt_0\in\Omega_s^0$ and $\vt'_0\in\Omega_{s'}^0$ the local sum
\begin{equation}\label{eq:nnaabbvvsdjjdf}
\c X(\b r; \b s; \b t_0,\vt'_0):=
\sum_{\substack{ \b F \in \cF_{\Z/K\Z}\\
\eqref{eq:Sigma_conditions} }}  
\prod_{p\mid s}(\Phi_1(\b t_p),\Phi_2(\b t_p))_p'  
\prod_{p\mid s'}(\Phi_1(\b t_p'),\Phi_2(\b t_p'))_p'.
\end{equation}  
Moreover, for $\b t_\infty,\b t_\infty'\in \R^2\smallsetminus\{\vzero\}$,
let
\begin{equation}\label{eq:nnaVOLUMERE}V(\vt_\infty,\vt_\infty';H):=
\mathrm{vol}\{\b F \in \c F(H): \max_{i=1,2} 
\{\Phi_i(\b t_\infty)\} \geq 0,\ \max_{i=1,2}  \{\Phi_{i }(\b t_\infty')\} \geq 0   \},
\end{equation} where       
$\c F$ is identified with $\R^{d+m}$ 
via 
the coefficients of all $F_{ij}$. 
The following lemma is the main result of this subsection. By definition, $\vn\sim x$ means that $\vn=(n_1,n_2) \in \ZZ^2\cap x\cB$ with $\gcd(n_1,n_2)=1$. Moreover, we write
\begin{equation}\label{subtrahended in a shore far away with only this paper to read}
\phi^\dagger(s):=\prod_{p\mid s}(1-p^{-2})^{-1}.
\end{equation}

\begin{lemma}\label{lem:goalofsection1}
Fix $\eta\in (0,\frac{1}{10})$, 
let $H,z\geq 1$,
let $1\leq
T_0\leq T$, and assume that 
 $z^4T^2 \leq H^{9/10}$.
Then the differences
\begin{align}
&\sum_{\b F\in\cFZ(H)} \frac{\hS_\b F(x)^2}{|\cFZ(H)|}-
\sum_{\substack{ s, s' \leq z}} 
\sum_{\substack{ [r_1,r_2]\leq T \\ [r'_1,r'_2]\leq T }} \ 
\sum_{\substack{\vn,\vn'\sim x } } 
\frac{4V(\vn,\vn';H)}{|\cFZ(H)|}
\frac{\c X(\b r;\b s;\vn,\vn')}{K^{d+m}},
\label{eq:jimi jendrix 1}
\\ 
&\sum_{\b F\in\cFZ(H)}
\frac{ \hS_\b F(x) x^2  \hSing(\b F) }{|\cFZ(H)|}-
\hspace{-0.4cm}\sum_{\substack{ s\leq z \\P^+(s')\leq L}}
\hspace{-0.2cm}\phi^\dagger(s)\hspace{-0.2cm} 
\sum_{\substack{  [r_1,r_2]\leq T \\ [r'_1,r'_2]\leq T_0}}
\sum_{\vn\sim x} \int \limits_{\Omega_{s'}^\cB }
\frac{4V(\vn,\b t'_\infty;H) x^2}{\zeta(2)|\cFZ(H)|}
\frac{\c X(\b r;\b s;\vn,\b t'_0)}{K^{d+m}} \mathrm d  \b t',
\label{eq:jimi jendrix 2}\\ 
&\sum_{\b F\in\cFZ(H)} \frac{x^4\hSing(\b F)^2}{|\cFZ(H)|}-
\sum_{\substack{P^+(ss')\leq L}} 
\phi^\dagger(s)\phi^\dagger(s')
\sum_{\substack{[r_1,r_2]\leq T_0\\ [r'_1,r'_2]\leq T_0}} 
\int  \limits_{\Omega_s^\cB \times \Omega_{s'}^\cB} 
\frac{4V(\b t_\infty,\b t'_\infty;H) x^4}{\zeta(2)^2|\cFZ(H)|}
\frac{\c X(\b r;\b s;\b t_0,\b t'_0)}{K^{d+m}} 
\mathrm d  \b t\mathrm d  \b t' \nonumber
\end{align}
are all of size $O(x^4 H^{-\eta})$, with the implied constant depending only on $\eta,m_1,m_2,m_3$ and the degrees $d_{ij}$. 

In the expressions above,
the sums run over square-free $s,s'$, and
the integers
$r_i,r'_i$ satisfy $p\mid r_1r_2\Rightarrow p\mid s$ and  
$p\mid r'_1r'_2\Rightarrow p\mid s'$ 
for all primes $p$
.
\end{lemma}  
We prove Lemma \ref{lem:goalofsection1} below, after some setup. 
For fixed $\b s, \b r$ as above, $\vt\in\Omega_s$ and $\vt'\in\Omega_{s'}$, we define the sum $\Sigma(\b r;\b s;\vt,\vt';H)$ as 
\begin{equation*}
\sum_{\substack{ \b F\in\cFZ(H)\\ \eqref{eq:Sigma_conditions}  }}\hspace{-0.3cm}(1+\hs{\Phi_1(\vt_\infty)}{\Phi_2(\vt_\infty)}_\infty')(1+\hs{\Phi_1(\vt'_\infty)}{\Phi_2(\vt'_\infty)})_\infty'\ 
\hspace{-0.1cm}\prod_{p\mid s}(\Phi_1(\vt_p),\Phi_2(\vt_p)  )_p'
\prod_{p\mid s'}(\Phi_1(\vt'_p),\Phi_2(\vt'_p) )_{p}'.
\end{equation*} 
\begin{lemma}\label{lem:skilaki} 
Let $H\geq 1$,  
and let $\b s, \b r,\vt,\vt'$ be as above,
such that $ss'[r_1,r_2][r_1',r_2']\leq H$.
Then
$$\Sigma(\b r;\b s;\vt,\vt';H)
=4V(\vt_\infty,\vt_\infty';H)\frac{\c X(\b r;\b s;\vt_0,\vt_0')}{K^{d+m}}  +  
O(H^{d+m-1}  [r_1,r_2] [r'_1,r'_2] [s, s'] ),$$ 
where the implied constant depends only on the $m_i$ and
$d_{ij}$.
\end{lemma} 
\begin{proof}
We identify $\cF(H)$ with $[-H,H]^{d+m}$
via the coefficients,
then the condition
$$\Phi_1(\vt_{\infty})\Phi_2(\vt_{\infty})\Phi_1(\vt_\infty')\Phi_2(\vt_\infty')=0$$
cuts out a family of semialgebraic subsets $Z_{\vt_\infty,\vt_\infty'}\subseteq [-H,H]^{d+m}$, depending only on the $m_i,d_{ij}$ and parameterised by $\vt_\infty,\vt_\infty',
H
$. As $\vt_\infty,\vt_\infty'\neq\vzero$, all of these sets have volume $0$. 

Outside of $Z_{\vt_\infty,\vt_\infty'}$, the expression $(1+\hs{\Phi_1(\vt_\infty)}{\Phi_2(\vt_\infty)}_\infty')(1+\hs{\Phi_1(\vt'_\infty)}{\Phi_2(\vt'_\infty)})_\infty'$ takes the value $4$ if and only if 
$$\max_{i=1,2}  \{\Phi_{i }(\b t_\infty)\} \geq 0\quad\text{ and }\quad\max_{i=1,2}  \{\Phi_{i }(\b t_\infty')\} \geq 0,$$
and $0$ otherwise. The latter conditions also cut out a family of semialgebraic sets $S_{\vt_\infty,\vt_\infty'}\subseteq \cF(H)$, depending only on the $m_i,d_{ij}$ and parameterised by the values of $\vt_\infty,\vt_\infty',
H$.

As explained after the definition of $K$ in \eqref{def:Klcm}, condition \eqref{eq:Sigma_conditions} and therefore also the value of \eqref{eq:HS_product} depend only on $\vF$ modulo $K$. Splitting in congruence classes, we find that $\Sigma(\b r;\b s;\vn,\vn';H)$ is equal to
\begin{equation*}
    4\sum_{\substack{\b F\in \cF_{\Z/K\Z}\\ \eqref{eq:Sigma_conditions}}}\prod_{p\mid s}\hs{\Phi_1(\vt_p)}{\Phi_2(\vt_p)}_p'\prod_{p\mid s'}\hs{\Phi_1(\vt_p')}{\Phi_2(\vt_p')}_p'\left|(\vF+K\cFZ)\cap S_{\vt_\infty,\vt_\infty'}\right| + O(|\cFZ\cap Z_{\vt_\infty,\vt_\infty'}|).
\end{equation*}
We can count lattice points in the sets $S_{\vt_\infty,\vt_\infty'}-\vF$ and $Z_{\vt_\infty,\vt_\infty'}$ with error terms uniform in $\vt_\infty,\vt_\infty',\vF,
H,K
$ using \cite{MR3264671}, yielding
\begin{equation*}
    \left|(\vF+K\cFZ)\cap S_{\vt_\infty,\vt_\infty'}\right| = \frac{\vol S_{\vt_\infty,\vt_\infty'}}{K^{d+m}} + O\left(\left(\frac{H}{K}\right)^{d+m-1}+1\right)
\end{equation*}
and $|\cFZ\cap Z_{\vt_\infty,\vt_\infty'}| = O(H^{d+m-1})$. As $\vol S_{\vt_\infty,\vt_\infty'}=V(\vt_\infty,\vt_\infty';H)$, the result follows by observing that  $K\ll [s,s'][r_1,r_2][r_1',r_2']
\leq H$.
\end{proof}  

We need the following lemma 
to bound the error term when applying Lemma~\ref{lem:skilaki}.
\begin{lemma} \label{lem:easy} Fix any $\epsilon>0$ and $k \in \N$. 
Then for any  
$z,T\geq 1 $ we have $$\sum_{s\leq z} |\{ \b r\in \N^2: 
[r_1,r_2]\leq T,p\mid r_1r_2\Rightarrow p\mid s \}|^k \ll
(zT)^\epsilon z,$$ where the implied constant only depends on $\epsilon $ and $k$.\end{lemma}\begin{proof}By Rankin's trick we 
bound the sum by  
$$ \sum_{s\leq z} \bigg(
\sum_{\substack{\b r\in \N^2   
\\p\mid r_1r_2\Rightarrow p\mid s }} \frac{T^{\epsilon/k} }
{[r_1,r_2]^{\epsilon/k}}\bigg)^k
= T^{\epsilon} \sum_{s\leq z}\left(\prod_{p\mid s }  \sum_{\alpha,\beta\geq 0} \frac{1}
{p^{ \max\{\alpha,\beta\}\epsilon/k}}\right)^k .$$ 
Letting $\gamma:=\max\{\alpha,\beta\}$,
letting $\omega(\cdot)$ denote the number of distinct prime factors,
and using $p\geq 2 $ we 
bound this further by
\[\leq T^{\epsilon} \sum_{s\leq z} \bigg( 
 \sum_{\gamma \geq 0} \frac{1+2\gamma }{2^{\epsilon \gamma/k}} \bigg)^{k\omega(s)}
= T^{\epsilon} \sum_{s\leq z} C(\epsilon,k)^{\omega(s)} 
\ll T^{\epsilon} z^{1+\eps}.\qedhere\]\end{proof}

\begin{proof}[Proof of Lemma \ref{lem:goalofsection1}]
We first bound the differnce \eqref{eq:jimi jendrix 1}. Opening up the square 
and using~\eqref{eq:onthetraintothegreekembassy}-\eqref{def:Shaf_f(x)}, we obtain 
$$\frac{1}{|\cFZ(H)|}\sum_{\b F\in\cFZ(H)} \hS_\b F(x)^2 =   
\sum_{\substack{\vn,\vn'\sim x } } 
\sum_{s,s' \leq z}\mu(s)^2 \mu(s')^2
\sum_{\substack{ [r_1,r_2],[r'_1,r'_2]\leq T 
\\ p\mid r_1 r_2\Rightarrow p\mid s\\
p\mid r'_1 r'_2\Rightarrow p\mid s'}}
\frac{\Sigma( \b r;\b s;\vn,\vn';H)}{|\cFZ(H)|}, $$
with $\Sigma( \b r;\b s;\vn,\vn';H)$ as defined before Lemma \ref{lem:skilaki}.
Applying 
Lemmas~\ref{lem:skilaki}-\ref{lem:easy}
with sufficiently small $\epsilon$
yields the claimed main term and 
error term of size $$\ll \frac{x^4}{H} \sum_{s,s' \leq z}[s, s']
\sum_{\substack{ [r_1,r_2],[r'_1,r'_2]\leq T 
\\ p\mid r_1 r_2\Rightarrow p\mid s\\
p\mid r'_1 r'_2\Rightarrow p\mid s'}}  [r_1,r_2] [r'_1,r'_2] 
\ll
\frac{x^4}{H} z^2 T^2 ((zT)^\eps z)^2 
=
\frac{x^4}{H^{1/10}} \frac{z^4T^2}{H^{9/10}} (zT)^{2\eps}
< \frac{x^4}{H^{\eta}}.$$

Let us now estimate the second difference \eqref{eq:jimi jendrix 2}.
By~\eqref{eq:onthetraintothegreekembassy}-\eqref{deg...singularseries...}  
we can write  
the sum over $\b F$ in \eqref{eq:jimi jendrix 2}
as
$$
x^2\sum_{\substack{ s\leq z \\ P^+(s')\leq L }} \mu(s)^2  \mu(s')^2 
\phi^\dagger(s')
\sum_{\substack{  [r_1,r_2]\leq T \\p\mid r_1r_2\Rightarrow p\mid s}} 
\sum_{\substack{  [r'_1,r'_2]\leq T_0 \\p\mid r'_1r'_2
\Rightarrow p\mid s'}}  
\sum_{\vn\sim x}
\int \limits_{\Omega_{s'}^{\c B}}
\frac{\Sigma(\b r;\b s; \vn,\vt';H)}{\zeta(2)|\cFZ(H)|}
\mathrm d  \vt'.$$Note that a square-free $s'$ with 
$P^+(s') \leq L=\sqrt{\log H}$ satisfies  
$s'\ll_\epsilon H^{\epsilon}$ for any $\epsilon>0$.
Therefore, 
employing Lemmas \ref{lem:skilaki}-\ref{lem:easy}
gives the desired main term and an error term  
$$ \ll  \frac{x^4}{H} \sum_{\substack{s\leq z\\ s' \ll_\eps H^\eps}}[s,s']
  \sum_{\substack{  [r_1,r_2]\leq T \\p\mid r_1r_2\Rightarrow p\mid s}} 
\sum_{\substack{  [r'_1,r'_2]\leq T_0 \\
p\mid r'_1r'_2\Rightarrow p\mid s'}}
[r_1,r_2][r_1',r_2']
\ll\frac{x^4}{H}z H^\epsilon T T_0
(z T H^\epsilon T_0)^\epsilon z H^\epsilon
.$$ 
In light of $T_0\leq T$ and $(zT)^2 \leq H^{9/10}$,
this is $\ll x^4 H^{-\eta}$, if $\eps$ was chosen sufficiently small.
 
Similarly, we estimate the remaining difference in Lemma \ref{lem:goalofsection1}.
By~\eqref{deg...singularseries...} we can express  
$\sum_{\b F} x^4\hSing(\b F)^2/|\cFZ(H)|$ as
$$x^4\sum_{P^+(ss')\leq L } 
\mu(s)^2\mu(s')^2\phi^\dagger(s)\phi^\dagger(s')
\sum_{\substack{  [r_1,r_2],[r'_1,r'_2]\leq T_0 \\
p\mid r_1r_2\Rightarrow p\mid s\\
p\mid r'_1r'_2\Rightarrow p\mid s'}}  \ 
\int \limits_{\Omega_s^\cB \times \Omega_{s'}^\cB }  
\frac{\Sigma(\b r,\b s;\vt,\vt';H)}{\zeta(2)^2|\cFZ(H)| }
\mathrm d\vt\mathrm d\vt'.$$ By Lemmas 
\ref{lem:skilaki}-\ref{lem:easy}, we again obtain the desired main term and, using that $s\ll_\eps H^\eps$ holds for all square-free $s$ with $P^+(s)\leq L$, an error term bounded by
\[
\ll
\frac{x^4}{H}\sum_{ s,s'\ll_\eps H^\eps} [s,s']  \sum_{\substack{  [r_1,r_2],[r'_1,r'_2]\leq T_0 \\
p\mid r_1r_2\Rightarrow p\mid s\\
p\mid r'_1r'_2\Rightarrow p\mid s'}}[r_1,r_2][r'_1,r'_2]
\ll \frac{x^4}{H}H^{2\eps} T_0^2((H^\eps T_0)^{\eps}H^\eps)^2
\ll \frac{x^4}{H^{\eta}}.\qedhere\]
\end{proof}

\subsection{Character sums}\label{s:charsums} In this section 
we give vanishing lemmas and  bounds
for the character sum $\c X$. Most results will 
emanate from Lemma~\ref{lem:new_hilb_integral234} whose proof we give here. 
\subsubsection{Proof of Lemma~\ref{lem:new_hilb_integral234}}\label{subsectnprofzerointegrla}

Write $t_i=p^{\beta_i}u_i$ with $u_i\in\ZZ_p^\times$ for $i=1,2$. First we assume that $p\neq 2$ and recall 
from \cite[Theorem 1 in Chapter III]{MR344216}
that in this case
$$ (t_1,t_2)_{\Q_p}= 
\l(\frac{-1}{p}\r)^{\beta_1\beta_2}
\l(\frac{u_1}{p}\r)^{\beta_2}
\l(\frac{u_2}{p}\r)^{\beta_1},$$
where $(\frac{\cdot}{\cdot})$ is the Legendre symbol. 
The integral over $\b t $ in Lemma \ref{lem:new_hilb_integral234} vanishes by definition of $\hs{\cdot}{\cdot}'_p$ when $\beta_1,\beta_2$ are both
even. Otherwise, the integral is equal to
\begin{equation*}
\ls{-1}{p}^{\beta_1\beta_2}\int_{\substack{ \b t \in \QQ_p^2 \\ v_p(t_i)=\beta_i, i=1,2}}\ls{u_1}{p}^{\beta_2}\ls{u_2}{p}^{\beta_1}
\mathrm{d}\b t,
\end{equation*}
which by Fubini and change of variables is equal to
\begin{equation*}
     \ls{-1}{p}^{\beta_1\beta_2}\left(p^{-\beta_1}\int_{\ZZ_p^\times}\ls{u_1}{p}^{\beta_2}\mathrm{d} u_1 \right)\left(p^{-\beta_2}\int_{\ZZ_p^\times}\ls{u_2}{p}^{\beta_1}\mathrm{d} u_2 \right) = 0.
\end{equation*}
Note that under our hypotheses on $\beta_i$, at least one of the Legendre symbols $\ls{u_i}{p}$ appears with odd exponent, whence the corresponding integral vanishes.

Now consider the case $p=2$, in which we have
$$ (t_1,t_2)_{\Q_2}= (-1)^{\frac{(u_1-1)(u_2-1)}{4}+\beta_2\frac{u_1^2-1}{8} + \beta_1\frac{u_2^2-1}{8}}.$$
If both $\beta_i$ are even, then the integral in Lemma \ref{lem:new_hilb_integral234} is by definition of $\hs{\cdot}{\cdot}'_2$ and change of variables equal to
$2^{-\beta_1-\beta_2}$ times
  \begin{equation*}
  \int_{\substack{(\ZZ_2^\times)^2}}\mathds{1}_{u_1\equiv u_2\bmod 4}(-1)^{\frac{(u_1-1)(u_2-1)}{4}}
\mathrm{d}\b u = \int_{(\ZZ_2^\times)^2}\one_{u_1\equiv u_2\equiv 1\bmod 4}-
\one_{u_1\equiv u_2\equiv 3\bmod 4}\mathrm{d}\b u = 0.
\end{equation*} 

If at least one of $\beta_1,\beta_2$ is odd, then $(t_1,t_2)_2'=(t_1,t_2)_{\Q_2}$. In this case, we may conclude by splitting into congruence classes $u_i\equiv a_i\bmod 4$ and observing that $(-1)^{(u_1-1)(u_2-1)/4}$ is constant on each such class, while 
\begin{equation*}
    \int_{\substack{u\in\ZZ_2^\times\\ u\equiv a\bmod 4}}(-1)^{\frac{u^2-1}{8}}\mathrm{d} u = 0
\end{equation*}
for all $a\in(\ZZ/4\ZZ)^\times$.\qed

\begin{lemma}\label{lem:count_forms_modulo_p_m}
Let $p$ be a prime, $d,l\in \NN$, and 
$u,n_1,n_2\in\ZZ/p^l\ZZ$ with $p\nmid \vn=(n_1,n_2)$. Then there are exactly 
$p^{dl}$ forms $g\in (\ZZ/p^l\ZZ)[t_1,t_2]$ of degree $ d$, such that 
$g(\vn)\equiv u \md{p^l}$. \end{lemma}
\begin{proof}
Assume without loss 
of generality that $p\nmid n_2$ and write $g:=\sum_{j=0}^{d}c_{j}
t_1^jt_2^{d-j}$. Then, as $n_2$ is invertible modulo $p^l$, 
the condition $g(\vn)\equiv u \md{p^l}$ is equivalent to
\begin{equation*}
c_{0} \equiv n_2^{-d}\big(u-\sum_{j=1}^{d}c_{j}n_1^j
n_2^{d-j}\big)\md{p^l},
\end{equation*}
which yields a unique value of 
$c_{0}$ for each choice of all the other coefficients $c_{j}$, 
$1\leq j\leq d$. Hence, the number of forms modulo $p^l$ satisfying 
this condition is $p^{ld}$.
\end{proof}

In the following lemmas, we consider square-free $s,s'\in\NN$, $r_1,r_2,r_1',r_2'\in\NN$ satisfying $p\mid r_1r_2\Rightarrow p\mid s$ and $p\mid r_1'r_2'\Rightarrow p\mid s'$, $\b t_0\in \Omega_s^0$, $\b t_0'\in \Omega_{s'}^0$, and the local sum $\c X(\b r;\b s;\b t_0, \b t_0')$ defined in \eqref{eq:nnaabbvvsdjjdf}. We show that these sums vanish in many cases.

\begin{lemma}\label{lem:pergomassD}
If $s\neq s'$, then 
$\c X(\b r;\b s;\vt_0,\vt_0')=0$.
\end{lemma}

\begin{proof}
With no 
loss of generality  there is a prime $p$ that divides $s$ but not $s'$.
By the Chinese remainder theorem we can split off its contribution into  
$$  \sum_{\substack{\b F:
v_p( \Phi_i(\vt_p) )=\rho_i \forall i }}  (\Phi_1(\vt_p),\Phi_2(\vt_p))_p' ,$$ 
where the sum is over $\b F \in \c F_{\Z/p^{\rho+\lambda}\Z}$,
$\rho_i=v_p(r_i)$, $\rho=\max\{\rho_1,\rho_2\}$ and $\lambda$ is $1$ or $3$ 
respectively when $p$ is odd or $2$.
Writing $u_{ij}= F_{ij}(\vt_p)$ and $U_i=\prod_{j=1}^{m_i} u_{ij}$,
this is equal to
$$p^{d(\rho+\lambda)}\sum_{\substack{ (u_{ij}) \in 
(\Z/p^{\rho+\lambda}\Z)^m \\ v_p(U_i U_3)=\rho_i \forall i}} 
(U_1 U_3 ,U_2 U_3)_p'$$ by Lemma \ref{lem:count_forms_modulo_p_m}.
Let us show that the  sum over $(u_{ij})$ vanishes.
First, 
\begin{equation*}
    p^{-2(\rho+\lambda)}\hspace{-0.8cm}\sum_{\substack{u_1,u_2\in\ZZ/p^{\rho+\lambda}\ZZ\\v_p(u_i)=\alpha_i}}\hspace{-0.5cm}\hs{c_1u_1}{c_2u_2}'_p=\int_{\substack{u_1,u_2\in\QQ_p\\v_p(u_i)=\alpha_i}}\hs{c_1u_1}{c_2u_2}'_p\mathrm d \b u= p^{v_p(c_1)+v_p(c_2)}\int_{\substack{v_1,v_2\in\QQ_p\\v_p(v_i)=\alpha_i+v_p(c_i)}}\hspace{-0.3cm}\hs{v_1}{v_2}'_p\mathrm d \b v
\end{equation*}
holds for all $c_1,c_2\in\ZZ_p$ and  $\alpha_1,\alpha_2\in\NN_0$ with $\alpha_i+v_p(c_i) \leq \rho$. The latter integral vanishes by Lemma \ref{lem:new_hilb_integral234}. For fixed admissible values of $(u_{1j})_{j=2}^{m_1}$, $(u_{2j})_{j=2}^{m_2}$,
$(u_{3j})_{j=1}^{m_3}$ we can apply this
with $u_i=u_{i1}$, $c_i= U_3\prod_{j=2}^{m_i}u_{ij} $ and $\alpha_i=
\rho_i-v_p(c_i)$ to deduce that 
the sum over $(u_{ij})$ vanishes.
\end{proof}  
 In the remaining cases with $s=s'$ the  sum $\c X$ still vanishes for many of  
the pairs $\b n, \b n'$.  
\begin{lemma}\label{lem:count_pairs_of_forms_modulo_p_m}Let $p$ be a 
prime, let $d,l\in\NN$, and let $(u,u'),\vn,\vn'\in(\ZZ/p^l\ZZ)^2$, such that  
$p\nmid n_1n_2'-n_1'n_2$. Then there are 
exactly $p^{l(d-1)}$ forms $g\in(\ZZ/p^l\ZZ)[t_1,t_2]$ of degree $ 
d$, such that $g(\vn)\equiv u\md{p^l} $ and $g(\vn')\equiv u'\md{p^l} $.
\end{lemma}
\begin{proof}
Write $g=\sum_{j=0}^{d} c_j t_1^jt_2^{d-j}$. 
Assume first that $p\nmid n_2n_2'$. We   fix   $c_j$ for all 
$j=2,\ldots, d$ so that $g(\vn)\equiv u$, $g(\vn')\equiv u'$ is equivalently written as 
\begin{align*} c_0n_2^d+c_1 n_1n_2^{d-1} &\equiv
u-\sum_{j=2}^{d} c_j n_1^jn_2^{d-j},\\
c_0n_2'^{d}+c_1 n_1'n_2'^{d-1} &\equiv u'-\sum_{j=2}^{d} c_j 
{n_1'}^j{n_2'}^{d-j}.  \end{align*}  This can be viewed as a 
system of $2$ linear equations in $c_0$ and $c_1$. The determinant of 
this system is $(n_2n_2')^{d-1}(n_1'n_2-n_1n_2')$, which is invertible in
$\ZZ/p^l\ZZ$ by hypothesis. Hence, the system has a unique solution 
$(c_0,c_1)$, and the total number of forms $g$ is $p^{l(d-1)}$.
  
In the remaining case, 
$p$ divides exactly one of
$n_1'n_2$ and $n_1n_2'$. 
Here,
we fix the coefficients $c_j$
for $j=1,\ldots,d-1$. Then the conditions $g(\vn)\equiv u$ and 
$g(\vn')\equiv u'$ 
give
the following   
system for $(c_0,c_d)$:    \begin{align*} c_0n_2^d + c_dn_1^d &\equiv 
u-\sum_{j=1}^{d-1} c_j n_1^jn_2^{d-j},\\
c_0n_2'^{d}+c_d n_1'^d &\equiv  u'-\sum_{j=1}^{d-1} 
c_j {n_1'}^j{n_2'}^{d-j}.  \end{align*} 
As $p$ does not divide the determinant $(n_1'n_2)^d-(n_1n_2')^d$, 
there is   a unique solution. 
\end{proof}

For $\vt_0=(\vt_p)_p\in\Omega_s^0$ and $i\in\{1,2\}$, we write $\vt_i=(t_{p,i})_p\in\prod_{p\mid s}\ZZ_p$.
\begin{lemma}\label{lem:scarlattiharpsichord}
If $s=s'$ and  
$s\nmid \vt_1\vt_2'-\vt_1'\vt_2$
in $\prod_{p\mid s}\ZZ_p$, then $\c X(\b r;\b s;\b t_0,\b t_0')=0$. 
\end{lemma} 
\begin{proof}
Our 
assumptions ensure that there is a prime  
 $p\mid s$ such that $t_{p,1}t_{p,2}'-t_{p,1}'t_{p,2}\in\ZZ_p^\times$. Using the 
 Chinese remainder theorem we can separate 
the $p$-part  and write it as  $$ \sum_{\substack{ \b F:
v_p( \Phi_i(\vt_p) )=\rho_i \forall i \\ 
v_p( \Phi_i(\vt_p') )=\rho'_i \forall i  }} (\Phi_1(\vt_p),\Phi_2(\vt_p))_p'  
(\Phi_1(\vt_p'),\Phi_2(\vt_p'))_p',$$ where the sum is over 
$\b F \in \c  F_{\Z/p^{\rho+\lambda}\Z}$, $\rho_i=v_p(r_i)$,
$\rho'_i=v_p(r'_i)$,  
$\rho=\max\{\rho_1, \rho_2, \rho'_1,\rho'_2\}$ and 
$\lambda$ is as in the proof of Lemma~\ref{lem:pergomassD}.
Letting   $u_{ij}= F_{ij}(\vt_p)$, 
$U_i=\prod_{j=1}^{m_i} u_{ij}$ and similarly for $u'_{ij}, U'_i$,
we can use Lemma~\ref{lem:count_pairs_of_forms_modulo_p_m} to turn the 
sum into  $$ p^{(d-m)(\rho+\lambda)}\sum_{\substack{ 
(u_{ij}), (u'_{ij}) \in (\Z/p^{\rho+\lambda}\Z)^m 
\\ v_p(U_i U_3)=\rho_i, v_p(U'_i U'_3)=\rho'_i \forall i}} 
(U_1 U_3 ,U_2 U_3)_p'(U'_1 U'_3 ,U'_2 U'_3)_p'.$$ 
The variables in the vector $ (u_{ij})$ are  independent from those in
$(u'_{ij})$. Hence, since we showed that the sum over $u_{ij}$ vanishes 
 in the proof of Lemma~\ref{lem:pergomassD}, the proof is complete.
\end{proof}  

Finally, we show that even when $\c X$ does not  vanish, 
it has small modulus.\begin{lemma}\label
{lem:bound_pairs_of_forms_modulo_p_m}
Let $p$ be a prime, $d,l,e\in\NN$  with  $e\leq l$, and 
$\vn\in(\ZZ/p^l\ZZ)^2$, such that $p\nmid \vn$. Then there are exactly 
$p^{l(d+1)-e}$ forms $g\in(\ZZ/p^l\ZZ)[t_1,t_2]$ of degree $ d$, such 
that $v_p(g(\vn))\geq e$.\end{lemma}

\begin{proof}
Sum the result of Lemma \ref{lem:count_forms_modulo_p_m} over all $p^{l-e}$ values of 
$u\in\ZZ/p^l\ZZ$ with $v_p(u)\geq e$.
\end{proof}

\begin{lemma}\label{lem:pergomasinDmaj}If $s'=s$, then
 $$  \frac{|\c X(\b r;\b s;\vt_0,\vt_0')|}{K^{d+m}}\leq 
\tau(K)^{2m}   \prod_{p\mid s }  
p^{-\max\{v_p(r_1),v_p(r'_1), v_p(r_2), v_p(r'_2)\}} .$$    
Moreover, if $s2^{-v_2(s)}$ does not divide 
both $r_1r_2$ and $r'_1r'_2$, 
then $\c X(\b r;\b s;\b t_0,\b t_0')=0$.
\end{lemma}

\begin{proof}
From the 
Chinese remainder theorem, we see that $ \c X(\b r;\b s;\vt_0,\vt_0') 
K^{-d-m}$ equals\begin{equation}\label{eq:Z_bound_splitting}
\prod_{p\mid s}p^{-(d+m)v_p(K)} \sum_{\substack{ \b F: 
v_p( \Phi_i(\vt_p) )=v_p(r_i) \forall i \\ 
v_p( \Phi_i(\vt_p') )=v_p(r'_i) \forall i  }} 
(\Phi_1(\vt_p),\Phi_2(\vt_p))_p'  (\Phi_1(\vt_p'),\Phi_2(\vt_p'))_p',
\end{equation} where the sum is over $\b F \in \c F_{\Z/p^{v_p(K) }\Z}$. We bound the factor corresponding to each $p\mid s$ individually, 
letting
\begin{equation}\label{eq:naming_valuations}
v_{ij}=v_p(F_{ij} (\b t_p) ),\quad v'_{ij}=v_p(F_{ij} (\b t_p') ).
\end{equation}
From this, we infer that \begin{equation}\label{eq:sepultura_arise} 
\sum_{\substack{ i=1,3 \\ 1\leq j \leq m_i }} v_{ij} = v_p(r_1),
 \sum_{\substack{ i=2,3 \\ 1\leq j \leq m_i }} v_{ij} = v_p(r_2),
  \sum_{\substack{ i=1,3 \\ 1\leq j \leq m_i }} v'_{ij} = v_p(r'_1),
 \sum_{\substack{ i=2,3 \\ 1\leq j \leq m_i }} v'_{ij} = v_p(r'_2)
.\end{equation} By Lemma~\ref{lem:bound_pairs_of_forms_modulo_p_m},
the number of binary forms $F_{ij}\md{p^{v_p(K)}}$ of degree $d_{ij}$ satisfying \eqref{eq:naming_valuations} 
is $\leq p^{v_p(K) (1+d_{ij}) - \max\{v_{ij},v'_{ij}\}}$. Hence,  
using the trivial estimate $|(\cdot,\cdot)'_p|\leq 1$ we bound the 
factor for every $p\mid s$
in~\eqref{eq:Z_bound_splitting} by $$  \sum_{\substack{
(v_{ij}),(v'_{ij})  \in [0,v_p(K))^m \\ \eqref{eq:sepultura_arise} }}
 p^{  - \sum_{i,j}\max\{v_{ij},v'_{ij}\}} \leq v_p(K)^{2m} p^{-  M},$$
 where $  M$ is smallest value that  $ \sum_{i,j}\max\{v_{ij},v'_{ij}\}$
 can take subject to~\eqref{eq:sepultura_arise}.  
Since $\max\{v,v'\} $ 
is at least
$v$,
we have $$ M \geq \sum_{\substack{i=1,2,3 \\1\leq j \leq m_i  }}
v_{ij} \geq \max\{v_p(r_1), v_p(r_2)\},$$ and similarly, 
$M   \geq \max\{v_p(r'_1), v_p(r'_2)\}$.  
Moreover, $\prod_{p\mid s} v_p(K)^{2m}\leq \prod_{p\mid K}v_p(K)^{2m} \leq 
 \tau(K)^{2m}$, which is sufficient for the proof of the 
 first claim.
 
 To prove the last claim we assume that $s2^{-v_2(s)}$ does not divide 
both $r_1r_2$ and $r'_1r'_2$. Then  without loss of generality there 
is an odd prime $p\mid s$ with $p\nmid r_1r_2$.  In the factor for $p$ in 
\eqref{eq:Z_bound_splitting}, we then have
$v_p(\Phi_1(\vt))=v_p(\Phi_2(\vt))=0$, 
which implies by definition of our  
analytic
Hilbert symbol 
$(\cdot,\cdot)'_p$ that $(\Phi_1(\vt),\Phi_2(\vt))'_p=0$.
\end{proof}  

\subsection{Level lowering and matching sum conditions}
\label{s:levellower}
Recall that the obstacle in estimating the sums in the first 
display in Lemma \ref{lem:goalofsection1} is that
  $\c X$, as a function of $\b n,\b n'$ is periodic with 
  period of size roughly $ss'[r_1,r_2][r_1,r'_2]$. The period 
has  typical size   $z^2 T^4$,  which 
  far exceeds the length of summation $x$. Thus,
  there is no obvious way to 
  estimate the sum over $\b n,\b n'$.
Our \textit{level lowering} trick uses the strong cancellation properties of the character sum $\c X$ from the previous subsection to discard most large values of $s,s',r_i,r'_i$. Recall that $L=\sqrt{\log H}$.

\begin{proposition}\label{lem:middles}
Assume $\omega\in (0,1),\eps\in (0,\omega)$, $H\geq x\geq H^{\omega}$,
$H\geq T 
\geq
T_0\geq H^\omega$, and $H\geq  z\geq 3^L$. Then:
\begin{enumerate}
\item The following changes to the outermost sums in the subtrahend in \eqref{eq:jimi jendrix 1} change the subtrahend by at most $O(x^4 L^{-1+\eps})$: replacing the conditions $s,s'\leq z$ by $P^+(ss')\leq L$, and replacing $T$ by $T_0$.

\item The following changes to the outermost sums in the subtrahend in \eqref{eq:jimi jendrix 2} change the subtrahend by at most $O(x^4 L^{-1})$: replacing the condition $s\leq z$ by $P^+(s)\leq L$, and replacing $T$ by $T_0$.
\end{enumerate}
The implicit constants depend only on $\eps, \omega, m_1, m_2, m_3$ and the degrees $d_{ij}$.
\end{proposition} 

The proof uses a series of  lemmas, which 
we state here but postpone their proofs until after the proof of Proposition~\ref{lem:middles}. For a prime $p$ and for $r_1,r_2,r_1',r_2'\in \N$,
denote 
\begin{equation*}
\mu_p(\b r):= \max\{v_p(r_1), v_p(r_2), v_p(r_1'), v_p(r_2')\}.
\end{equation*}

\begin{lemma}\label{lem:forappendix} 
For any   $0<\eps<1$, $t\geq 0$ and square-free positive integer $s$ we have $$  
\sum_{\substack{r_1,r_2 \in \N,\ 
p\mid r_1 r_2\Rightarrow p\mid s
\\ s2^{-v_2(s)}\mid r_1r_2 }}
\ \ \sum_{\substack{r_1',r_2' \in \N,\    
p\mid r'_1 r'_2\Rightarrow p\mid s\\s2^{-v_2(s)}\mid r'_1r'_2}}
\  \prod_{\substack{p\mid s  }}  
\frac{(1+\mu_p(\b r) )^t}{p^{\mu_p(\b r)}}
\ll s^{-1+\eps},$$ where 
the implied constant 
depends only on $\eps$ and $t$. \end{lemma}

 \begin{lemma}\label{lem:forappendixgs6s6s6n26}
For 
$\eps>0$, $t\geq 0$, 
$T_0\geq 1$,
$\lambda\in (0,1)$ and any 
square-free positive integer $s$,  
we have $$
\sum_{\substack{r_1,r_2,r'_1,r'_2 \in \mathbb{N} 
\\ [r_1,r_2]>T_0}}
 \prod_{ p\mid s  } 
 \frac{(1+\mu_p(\b r) )^t}{p^{\mu_p(\b r)}}
 \ll T_0^{-\lambda}
 s^{\lambda-1+\epsilon},$$ where sum over $r_1,r_2,r_1',r_2'$ 
 is subject to the further conditions that are present in 
 the sums in Lemma \ref{lem:forappendix}, and the
 the implied constant depends only on 
$\epsilon,t$ and $\lambda$.
\end{lemma}

\begin{lemma}\label{lem:auxiliary casanatense}
Fix any $\epsilon \in (0,1)$. Then for any 
$x,z,\Lambda\geq 1 $ we have 
$$ \sum_{\substack{s  \leq z \\P^+(s)>\Lambda} } 
\frac{\mu(s)^2}{s^{1-\epsilon}}\#
\left\{ \b n,\b n' \in \Z^2 : 
\begin{array}{l}
|\b n|,|\b n'|\leq x, \\
\gcd(n_1,n_2)=1= \gcd(n'_1,n'_2),  \\
n_1 n'_2\equiv n'_1n_2 \md s
\end{array}\right\}
\ll \frac{x^4 }{\Lambda^{1-2\epsilon}}+ x^3 z^\epsilon,$$ where the implied constant depends only on $\epsilon$. \end{lemma}

\begin{proof}[Proof of Proposition \ref{lem:middles}]
By Lemmas~\ref{lem:pergomassD} and~\ref{lem:scarlattiharpsichord}
the subtrahend in  \eqref{eq:jimi jendrix 1} is
\begin{equation}\label{eq:dipk}
\sum_{\substack{ s  \leq z } } \mu(s)^2
\sum_{\substack{ [r_1,r_2],[r'_1,r'_2]\leq T\\p\mid r_1 r_2 r'_1,r'_2
\Rightarrow p\mid s}} \ 
\sum_{\substack{\vn,\vn'\sim x\\s\mid n_1n_2'-n_1'n_2 } } 
\frac{4V(\vn,\vn';H)}{|\cFZ(H)|}
\frac{\c X(\b r;(s,s);\vn,\vn')}{K^{d+m}}.\end{equation}
Note that the condition $P^+(s)\leq L$ implies that
\begin{equation}\label{eq:sboundpi}
s\leq \prod_{p\leq L} p \leq 3^L \leq z
\end{equation} for all large enough $H$ by the prime
number theorem in the form $\sum_{p\leq L} \log p \sim L$. 
Using Lemma~\ref{lem:pergomasinDmaj} and the obvious 
estimate $V(\vn,\vn';H)\ll |\cFZ(H)|$,
we see that the terms in \eqref{eq:dipk} failing $P^+(s) \leq L$ contribute 
$$\ll\sum_{\substack{\vn,\vn'\sim x } }  \ 
\sum_{\substack{s  \leq z, P^+(s)>L \\ 
s\mid n_1n_2'-n_1'n_2} } \mu(s)^2
\sum_{\substack{ [r_1,r_2],[r'_1,r'_2]\leq T
\\p\mid r_1 r_2 r'_1r'_2
\Rightarrow p\mid s}} 
\tau(K)^{2m}   \prod_{p\mid s }  
p^{-\mu_p(\b r)} ,$$
subject to the further condition 
$s2^{-v_2(s)}\mid  (r_1r_2,r'_1r'_2)$. 
Recalling the definition of $K$ in~\eqref{def:Klcm}
and using that $s'=s$, 
we have 
\begin{equation}\label{i love tourists:more tourists please}
\tau(K) \ll \tau(s) \prod_{p\mid s}
(1+\mu_p(\b r)).
\end{equation}
Hence,  applying Lemma~\ref{lem:forappendix} we 
get $$\ll\sum_{\substack{\vn,\vn'\sim x } }  \ 
\sum_{\substack{s  \leq z, P^+(s)>L \\ 
s\mid n_1n_2'-n_1'n_2} } 
\frac{\mu(s)^2} {s^{1-\epsilon/2}}.$$
By Lemma~\ref{lem:auxiliary casanatense}
this is $$ \ll \frac{x^4}{L^{1-\epsilon}}
+x^3 z^{\epsilon/2} \ll \frac{x^4}{L^{1-\epsilon}}, $$ due to our assumptions 
$z\leq H$, $x\geq H^\omega$ and $\epsilon<\omega$, which ensure that 
$$ z^{\epsilon/2} \leq H^{\epsilon/2} \leq x^{\epsilon/(2\omega)} \leq x^{1/2}\ll x L^{-1+\epsilon}.$$

This was the bottleneck. 
Let us now consider the contribution of the terms satisfying $P^+(s)\leq L$ and $T\geq [r_1,r_2] >T_0$
towards~\eqref{eq:dipk}. 
Note that $K \ll [r_1,r_2][r'_1,r'_2]s\ll zT^2$,
hence, 
\begin{equation}\label{eq:tau_K_bound}
\tau(K)^{2m}  \ll (zT^2)^{\epsilon/12} 
\leq
H^{\epsilon/4}.
\end{equation}
Using this together with  
Lemmas~\ref{lem:pergomasinDmaj} and~\ref{lem:forappendixgs6s6s6n26} with $\lambda:=\eps/\omega\in(0,1)$  
and $t=0$
yields the crude bound
$$ \ll H^{\epsilon/4}\sum_{\substack{\vn,\vn'\sim x}} T_0^{-\lambda}
\sum_{s\leq 3^L} s^{\lambda-1+\epsilon}\ll H^{\epsilon/4}3^{L(1+\epsilon)}T_0^{-\lambda}x^4 \ll x^4 H^{\epsilon/2-\omega\lambda}\ll \frac{x^4}{H^{\eps/2}}\ll \frac{x^4}{L}.$$    

It remains to prove the proposition's second assertion. Consider the subtrahend in \eqref{eq:jimi jendrix 2}.
By Lemma~\ref{lem:pergomassD}, only the terms with $s=s'$ are relevant, and since $P^+(s')\leq L$
we infer that $P^+(s)\leq L$.
Hence, the subtrahend
equals $$\sum_{\substack{ P^+(s)\leq L }}\mu(s)^2
\phi^\dagger(s) \sum_{\substack{  [r_1,r_2]\leq T 
\\ [r'_1,r'_2]\leq T_0 \\ p\mid r_1r_2r'_1r'_2\Rightarrow p\mid s }}\sum_{\vn\sim x} \int \limits_{\Omega_{s}^\cB }
\frac{4V(\vn,\b t_\infty';H)x^2 }{\zeta(2)|\cFZ(H)|}
\frac{\c X(\b r;(s,s);\vn,\b t_0')}{K^{d+m}} \mathrm d  \b t'.$$ To finish the proof we   only need to 
bound the contribution of the terms with 
$[r_1,r_2]>T_0$. Since 
$\phi^\dagger$ is bounded, the contribution is 
$$\ll x^2  \sum_{\substack{ P^+(s)\leq L }}\mu(s)^2
  \sum_{\substack{  [r_1,r_2]> T_0 
\\ [r'_1,r'_2]\leq T_0 \\ p\mid r_1r_2r'_1r'_2\Rightarrow p\mid s }}\sum_{\vn\sim x} \int \limits_{\Omega_{s}^0}
\frac{|\c X(\b r;(s,s);\vn,\b t')|}{K^{d+m}} 
\mathrm d  \b t'.$$ By  
Lemma~\ref{lem:pergomasinDmaj}, Lemma ~\ref{lem:forappendixgs6s6s6n26} with $\lambda:=\epsilon/\omega$, and the bounds ~\eqref{eq:sboundpi},\eqref{eq:tau_K_bound}, we again obtain the estimate
$$\ll x^2 H^{\epsilon/4} \sum_{\vn\sim x}\int\limits_{\Omega_s^0}T_0^{-\lambda}\sum_{\substack{ s\leq 3^L }}
s^{\lambda-1+\epsilon}\mathrm d \vt\ll x^4 H^{\epsilon/4}3^{L(1+\epsilon)}T_0^{-\lambda} \ll x^4 H^{\epsilon/2-\omega\lambda}\ll \frac{x^4}{H^{\eps/2}}\ll \frac{x^4}{L}.$$ 
 \end{proof}
\begin{proof}[Proof of Lemma \ref{lem:forappendix}]
The sum over $\b r , \b r'$ factorises 
as $\prod_{p\mid s} c_p,$ where 
$c_2 $ is at most$$
\sum_{\substack{ k_1,k_2\geq 0\\ k'_1,k'_2\geq 0}} 
(1+\max \{ k_1,  k'_1, k_2, k'_2\})^t
2^{- \max \{ k_1,  k'_1, k_2, k'_2\}}\leq 4
\sum_{\mu \geq 0} \frac{(1+\mu)^t}{2^\mu} (1+\mu)^3
\ll 1.$$ 
For an odd prime $p$ that divides $s$, 
the value of $c_p$ equals $$\sum_{\substack{ k_1,k_2,k'_1,k'_2\geq 0\\k_1+k_2,k'_1+k'_2\geq 1  }} 
(1+\max\{k_1,k'_1,k_2,k'_2\})^t
p^{-\max\{k_1,k'_1,k_2,k'_2\}} 
\leq 4\sum_{\mu\geq 1} \frac{(1+\mu)^t}{p^\mu}(1+\mu)^3
 \leq \frac{C}{p},$$ with a constant   $C
 =C(t)
 >1$. 
Since $s$ is square-free, we get  $ \prod_{p\mid s} c_p \ll 
 C^{\omega(s)} s^{-1}\ll  s^{-1+\eps}$.
\end{proof}

\begin{proof}[Proof of Lemma
\ref{lem:forappendixgs6s6s6n26}]
We use  Rankin's trick  by multiplying
 the summand by $([r_1,r_2] /T_0)^\lambda$ and obtain the upper bound  
$   T_0^{-\lambda} \prod_{p\mid s }H_p,$  where 
\begin{equation}\label{def:H_p} H_p=
\sum_{\substack{k_1,k_2\geq 0   \\ p\neq 2\Rightarrow k_1+k_2\geq 1 }}
 \sum_{\substack{k'_1,k'_2\geq 0   \\ p\neq 2\Rightarrow k'_1+k'_2\geq 1 }}(1+\max \{ k_1,  k'_1, k_2,  k'_2\})^t
   p^{-\max \{ k_1,  k'_1, k_2,  k'_2\}+\lambda\max\{k_1,k_2\}}.  
   \end{equation} 
   Letting $\mu:=\max\{k_1,k_2,k_1',k_2'\}$, 
we get  
$$H_2\leq  
4
\sum_{\mu\geq 0 }(1+\mu)^{t+3} 2^{-\mu+\lambda \mu}\ll 1.$$For   $p\neq 2 $, we have    
$$H_p \leq 4 \sum_{\mu\geq 1 } (1+\mu)^{t+3}  p^{-\mu+\lambda \mu}
 \leq C p^{\lambda-1} $$ for some constant 
$ C=C(\lambda,t)>1$. This is sufficient due to 
$C^{\omega(s)}
\ll s^\epsilon$.
\end{proof}

\begin{proof}[Proof of 
Lemma \ref{lem:auxiliary casanatense}]
Define $ s_1:=\gcd(s,n_1)$ and 
$s_2:=\gcd(s,n_2)$, so that
$\gcd(s_1,s_2)=1$. Then $s_1s_2$ divides $s$, 
hence, $s_0:= s/(s_1s_2)$ is an integer. 
As $s_i\mid n_i$, we get $ s_1,s_2\leq  x$, and
furthermore, $s_1$ is coprime to $n_2$. 
But $n_1 n'_2\equiv n'_1n_2 \md {s_1}$,
hence $s_1$ divides $n_1'$. Similarly $s_2\mid (n_2,n_2')$.
Writing $(n_1,n_1')= s_1 (m_1,m_1')$ and 
$(n_2,n_2')= s_2 (m_2,m_2')$, we obtain the upper bound  
$$ \sum_{\substack{s  \leq z \\P^+(s)>\Lambda} } 
\frac{\mu(s)^2}{s^{1-\epsilon}}
\sum_{\substack{ s_0, s_1, s_2 \in \N \\
s_0 s_1 s_2 = s \\s_1,s_2 \leq  x }}
\#\left\{ \b m,\b m' \in \Z^2 : 
\begin{array}{l}
|m_1|,|m'_1|\leq  \frac{x}{s_1},
|m_2|,|m'_2|\leq \frac{x}{s_2} \\
\gcd(s_0, m_2)=1, \\
m_1 m'_2\equiv m'_1m_2 \md {s_0}
\end{array}\right\}
.$$ Using the property 
 $\gcd(s_0, m_2)=1$, we note that 
 for each fixed 
$m_1,m_2,m'_2$ there exists a unique 
$m'_1 \in \Z/{s_0}\Z$ satisfying $m_1 m'_2\equiv m'_1m_2 \md {s_0}$.
Thus we get the bound $$ \ll 
\sum_{\substack{s  \leq z \\P^+(s)>\Lambda} } 
\frac{\mu(s)^2}{s^{1-\epsilon}}
\sum_{\substack{ s_0, s_1, s_2 \in \N \\
s_0 s_1 s_2 = s \\s_1,s_2 \leq  x }} 
\frac{x}{s_1} \frac{x^2}{s_2^2} \l(\frac{x}{s_1s_0}+1\r)
\ll
\sum_{\substack{ P^+(s)>\Lambda } }  
\frac{x^4 }{s^{2-2\epsilon} }
+  \sum_{\substack{s_0, s_1, s_2 \in \N \\ 
s_i  \leq z \forall i} } 
\frac{x^3}{s_0^{1-\epsilon} s_1^{2-\epsilon} 
s_2^{3-\epsilon} } 
,$$ where we used the fact that the number of 
$(s_0,s_1,s_2)\in \mathbb N^3$ with $s_0 s_1 s_2=s$
is at most $\tau(s)^2\ll s^\epsilon$. 
The $s$ in the first sum in the right-hand side 
satisfy $s>\Lambda$ hence the sum is 
$$\ll x^4 \sum_{s>\Lambda} \frac{1}{s^{2-2\epsilon}}\ll 
\frac{x^4 }{\Lambda^{1-2\epsilon}}.$$
The second sum in the right-hand side is   
\begin{equation*}
\ll x^3\sum_{1\leq s_0 \leq z }
\frac{1}{s_0^{1-\epsilon}  } \ll x^3 z^\epsilon. \qedhere
\end{equation*}
\end{proof}   

\subsection{Passing from sums over $\b n,\b n'$ to integrals}
\label{s:passnnntolocals}After Proposition~\ref{lem:middles} the 
three right-hand side main terms in Lemma~\ref{lem:goalofsection1}
completely agree, save for the sums over $\b n,\b n'$ that
differ from the corresponding integrals weighted by $\phi^\dagger(\cdot)$. 
The main result of this section shows that, when the appearing 
moduli are small, the sums asymptotically approach the 
integrals. For fixed $\b s,\b r$, denote  $$\begin{aligned}
&\Delta_1:=
\sum_{\substack{\vn,\vn'\sim x } } 
\frac{4V(\vn,\vn';H)}{|\cFZ(H)|}
\frac{\c X(\b r;\b s;\vn,\vn')}{K^{d+m}},\ 
\Delta_2:=
\phi^\dagger(s)
\sum_{\vn\sim x} \int \limits_{\Omega_{s'}^\cB }
\frac{4V(\vn,\b t'_\infty;H) x^2}{\zeta(2)|\cFZ(H)|}
\frac{\c X(\b r;\b s;\vn,\b t'_0)}{K^{d+m}} \mathrm d  \b t',
\\ 
&\Delta_3:=
\phi^\dagger(s)\phi^\dagger(s') 
\int \limits_{\Omega_s^\cB \times \Omega_{s'}^\cB} \hspace{-0.2cm}
\frac{4V(\b t_\infty,\b t'_\infty;H) x^4}{\zeta(2)^2|\cFZ(H)|}
\frac{\c X(\b r;\b s;\b t_0,\b t'_0)}{K^{d+m}} 
\mathrm d  \b t\mathrm d  \b t'.\end{aligned}$$

Recall that $L=\sqrt{\log H}$.
\begin{proposition}\label{lem:spacanapoli_molto_Knoblauch}
Assume  $H\geq T_0\geq 1$, $x^{1/12}\geq T_0$, and $\log H \leq (\log x)^{3/2}$.
Then $$
\sum_{\substack{ P^{+}(s s') \leq L }} \mu(s)^2\mu(s')^2
\sum_{\substack{ [r_1,r_2], [r'_1,r'_2]\leq T_0 \\ p\mid r_1r_2\Rightarrow 
p\mid s \\ p\mid r'_1r'_2\Rightarrow p\mid s'  }} 
(\Delta_1- 2\Delta_2+\Delta_3) 
\ll   
x^{4-1/4},$$
where the implied constant 
depends only on 
$m_1,m_2,m_3$ and the $d_{ij}$.
\end{proposition}

For the proof we requre a preliminary lemma. Recall the definition 
of $V$ in~\eqref{eq:nnaVOLUMERE}.
\begin{lemma}\label{lem:V_lipschitz_estimate}
  Let $\vt_1,\vt_1',\vt_2,\vt_2'\in\RR^2\smallsetminus\{\vzero\}$. Then
  \begin{equation*}   \abs{V(\vt_1,\vt_1';H)-V(\vt_2,\vt_2';H)}
  \ll H^{d+m}\max\left\{\frac{\abs{\vt_1-\vt_2}}{\max
  \{\abs{\vt_1},\abs{\vt_2}\}},\frac{\abs{\vt_1'-\vt_2'}}
  {\max\{\abs{\vt_1'},\abs{\vt_2'}\}}\right\},
\end{equation*} 
with the implicit constant depending only on $m_1,m_2,m_3$ and the $d_{ij}$.
\end{lemma}\begin{proof} We first use 
Lemma~\ref{lem:linear_forms_volume} 
in the appendix
to deal with all 
$F_{ij}$ with $h(F_{ij})\leq H$ 
such that $F_{ij}(\b t_1)$ and  $F_{ij}(\b t_2)$ have a different 
sign. Identifying $F_{ij}$ with its coefficient vector in $\R^{1+d_{ij}}$,
we consider the linear forms $L_1(F_{ij}):=F_{ij}(\b t_1)$ and 
$L_2(F_{ij}):=F_{ij}(\b t_2)$. We have $$ h(L_1-L_2)=
\max_{r=0,\ldots, d_{ij} }|t_{11}^r t_{12}^{d_{ij}-r}
- t_{21}^r t_{22}^{d_{ij}-r} | \ll |\b t_1-\b t_2| 
\max\{|\b{t}_1|,|\b{t}_2|\}
^{d_{ij}-1} $$ and  $h(L_l)=|\b {t}_l|^{d_{ij}}$ for $l=1,2$.
 Hence, Lemma \ref{lem:linear_forms_volume} shows that
 the set of all $\b F=(F_{ij})$ with $h(\b F)\leq H$, such that $F_{ij}(\b t_1)$ and  $F_{ij}(\b t_2)$ have a different 
sign for some $i,j$, has volume bounded by $$\ll \frac{H^{d+m}}{\max\{\abs{\vt_1},\abs{\vt_2}\}}\abs{\vt_1-\vt_2}.$$
The analogous bound holds for 
the volume of all $\b F=(F_{ij})$ with $h(\b F)\leq H$, such that some 
$F_{ij}(\b t'_1)$ and  $F_{ij}(\b t'_2)$ have a different 
sign. 

In the remaining set of $\b F$ we therefore have 
$ F_{ij}(\b t_1) F_{ij}(\b t_2) \geq 0$ and $F_{ij}(\b t'_1) 
F_{ij}(\b t'_2) \geq 0$ for all $i,j$. This property implies that 
\begin{equation}\label{eq:whaatt???}
\Phi_{i}(\b t_1) \Phi_{i}(\b t_2) \geq 0 
\ \textrm{ and } \ 
\Phi_{i}(\b t'_1) \Phi_{i}(\b t'_2) \geq 0 \end{equation}
for all $ i=1,2$. Restricting the set of $\b F$ 
measured by $V(\vt_1,\vt_1';H)$ to those that satisfy 
\eqref{eq:whaatt???} gives the same set as when 
we  restrict the set measured by $V(\vt_2,\vt_2';H)$. 
This is sufficient for
the proof.\end{proof}

\begin{proof}[Proof of Proposition~\ref{lem:spacanapoli_molto_Knoblauch}] 
We will use Lemma \ref{lem:V_sum_n_appendix} and Lemma \ref{lem:V_sum_n2345appendix} from the appendix. Fix $\b s,\b r$.
By Lemma~\ref{lem:pergomassD}
we can assume that $s'=s$.
Recall the definition of $V$ in~\eqref{eq:nnaVOLUMERE}
and let $\omega(\b x, \b y):= {V(\b x,\b y;H)}/{|\cFZ(H)|}
\ll 1
$,
so that both $\omega(\b x, \cdot), \omega(\cdot,\b y)$ 
satisfy~\eqref{assumption:fflipstz} by 
Lemma~\ref{lem:V_lipschitz_estimate}
and~\eqref{assumption:rimint} as both  
$\Phi_i$ are homogeneous. Moreover, we take 
$$P(\vn,\vn'):=\frac{\c X(\b r;\b s;\b n,\b n')}{K^{d+m}},$$
so both $P(\vn,\cdot)$ and $P(\cdot,\vn)$ satisfy 
\eqref{assumption:periodK} by our choice of $K$. Therefore, Lemma~\ref{lem:V_sum_n2345appendix} shows that
$$ \Delta_1=\Delta_3+
O\left(K^3 x^3 (\log x)(\log L)\right).$$
Next, we write
$$\Delta_2 = \frac{4\phi^\dagger(s)x^2}{\zeta(2)}\int_{\Omega_{s}^\cB}\left(\sum_{\vn\sim x}\omega(\vn,\vt'_\infty)P(\vn,\vt'_0)\right)\mathrm d\vt'$$
and apply Lemma \ref{lem:V_sum_n_appendix} to evaluate the inner sum for each $\vt'$ to see that also
$$\Delta_2=\Delta_3+O(K^3 x^3 (\log x)(\log L)),$$
and thus
$$\Delta_1-2\Delta_2+\Delta_3 = O(K^3 x^3 (\log x)(\log L)).$$

Recalling~\eqref{eq:sboundpi} and 
 $K\ll [r_1,r_2][r_1',r'_2][s,s']=
[r_1,r_2][r_1',r'_2]s\ll T_0^2 3^L$, 
the sum to be bounded in the proposition becomes  
$$\ll  ( T_0^2 3^L)^3 x^3 (\log x)(\log L) 
\sum_{\substack{ s\leq 3^L  }}  \bigg(
\sum_{\substack{ [r_1,r_2] \leq T_0 
\\ p\mid r_1r_2  \Rightarrow p\mid s  }} 
1\bigg)^2.$$ Applying  Lemma~\ref{lem:easy}
with $k=2$ and sufficiently small $\eps>0$ provides the overall error term
$$ \ll  T_0^{6+\eps} 3^{5L}
x^3 (\log x)\ll  
\frac{T_0^{6}}{x^{1/2}}x^{7/2+3\eps}\ll x^{15/4},$$ 
where $3^{5 L}=
3^{O(\sqrt{\log H} )}\ll x^\eps$ follows from our 
assumption $\log H \leq (\log x)^{3/2}$.
\end{proof}
\begin{proposition}
\label{prop:pain and blood and tears} Fix $\omega\in (0,1)$ and $\lambda\in(0,\omega)$.
Assume that 
$x,T_0,T,z$ satisfy 
$$H^\omega \leq x \leq H,\  
z^4T^2\leq H^{9/10},\  3^L\leq z,\  H^\omega\leq T_0\leq 
\min\{T, x^{1/12}\}.$$
Then 
$$  \frac{1}{|\cFZ(H)|}
\sum_{\b F\in\cFZ(H)}
| \hS_\b F(x)-x^2\hSing(\b F)|^2 \ll \frac{x^4}{L^{1-\lambda}}
,$$ where the implied constant 
depends at most on $m$, the $d_{ij}$,
$\lambda$ and $\omega$.
\end{proposition}

\begin{proof} By expanding the square and applying Lemma~\ref{lem:goalofsection1} with, say, $\eta=1/20$, we can replace the sums over $\b F$ with corresponding local sums. We then use Proposition~\ref{lem:middles} with $\eps=\lambda$ to simplify the moduli. As $\log H \leq (\log x)/\omega \leq (\log x)^{3/2}$ for sufficiently large $H$, we may finally invoke Proposition~\ref{lem:spacanapoli_molto_Knoblauch} to transition from sums over $\b n$ to analogous integrals. 

In this process, we pick up an error term
\begin{equation*}
\ll \frac{x^4}{H^{1/20}} + \frac{x^4}{L^{1-\lambda}} + \frac{x^{4}}{x^{1/4}}\ll \frac{x^4}{L^{1-\lambda}}.\qedhere
\end{equation*}
\end{proof}

\subsection{Anatomy of adelic integers}
\label{s:smoothings} Recall the definitions of $\Sing$ in \eqref{eq:def_Sing_mult} and 
$\hSing$ in \eqref{deg...singularseries...}.
It now   remains to 
remove, up to an admissible error term, 
the condition $[r_1,r_2] \leq T_0$ from $\hSing$.
The main idea is that the condition $[r_1,r_2] > T_0$ forces the existence of some $\b t$
in an appropriate adelic space, such that at least one $p$-adic valuation of $F_{ij}(\b t)$
is somewhat large. We will show that this happens rarely
by adapting
anatomy-of-integers estimates of Erd\H os from~\cite{MR44565} 
to an adelic setting.

Recall again that $L=\sqrt{\log H}$ and 
define  
the ring $\AAL:=\prod_{p\leq L} \Z_p$. 
As usual,  $\ZZ$ can be embedded   diagonally in $\AAL$. 
Let us also write $\AAL^{2*}:=\prod_{p\leq L}\ZZ_p^2\smallsetminus p\ZZ_p^2=\Omega_s^0$ for $s:=\prod_{p\leq L}p$, and write elements of $\AAL^{2*}$ in the form $\vt = (\vt_p)_{p\leq L}$. Moreover, by $\pi(L)$ we denote the number of primes up to $L$.

\begin{lemma}\label{lem:sing_dddsing_L2} For any $1\leq T_0\leq H$, we have
$$\sum_{\b F\in\cFZ(H)}
 |\mathfrak S(\b F)-\hSing(\b F)|^2\ll 4^{\pi(L)}\sum_{i, j}
 \sup_{\b t\in\AAL^{2*} }\#\l\{\b F\in\cFZ(H):\prod_{p\leq L}p^{v_p(F_{ij}(\vt_p))} 
> T_0^{1/(2m)}\r\} ,$$ with an implied constant depending only on $m_1,m_2,m_3$ and the $d_{ij}$.
\end{lemma}
 
\begin{remark}
  By convention, the condition 
  $$\prod_{p\leq L}p^{v_p(F_{ij}(\vt_p))}> T_0^{1/(2m)}$$
  is satisfied in case $F_{ij}(\vt_p)=0$ for some $p\leq L$. In this case, we interpret the product on the left-hand side as $\infty$.
\end{remark}

\begin{proof}Since $\omega_\infty(\b F)\ll 1$  holds 
  uniformly in $\b F$,
 we   obtain for large enough $H$ the bound
\begin{equation}\label{eq:sum_EFS_bound}
\ll   \sum_{\b F\in\cFZ(H)}\Bigg(
\sum_{\substack{ s \textrm{\ square-free} \\P^+(s)\leq L}  }\phi^\dagger(s)
\mathrm{vol}(E_{\b F,s})  \Bigg)^2\leq 2^{\pi(L)} \sum_{\b F\in\cFZ(H)}
\sum_{\substack{ s \textrm{\ square-free} \\P^+(s)\leq L}  }\phi^\dagger(s)
\mathrm{vol}(E_{\b F,s})
,
\end{equation}
where
$\phi^\dagger(s)$ is defined in \eqref{subtrahended in a shore far away with only this paper to read} and
$E_{\b F,s}$ is the set of all $\b t_0=(\b t_p)_{p\mid s}\in\Omega_s^0$ for which 
\begin{equation*}
\prod_{p\mid s} p^{\max\{v_p(\Phi_1(\vt_p)),v_p(\Phi_2(\vt_p))\}}
> T_0, 
\end{equation*}
so in particular $\phi^\dagger(s)\vol(E_{\b F,s})\leq 1$.
If $\b t_0 \in E_{\b F,s}$, then there exists $i\in \{1,2\}$ such that  
$\prod_{p\mid s}p^{v_p(\Phi_i(\vt_p))}> \sqrt{T_0}$, and hence there are values of $i,j$ 
such that $\prod_{p\mid s}p^{v_p(F_{ij}(\vt_p))}> T_0^{1/(2m)}$. With $S:=\prod_{p\leq L} p$, this shows that $\phi^\dagger(s)\vol(E_{\b F,s})$ is  
bounded by
\begin{equation*} 
\phi^\dagger(S)\sum_{i,j}\int_{\vt \in \AAL^{2*}}\one_{\prod_{p\mid s}p^{v_p(F_{ij}(\vt_p))}>T_0^{1/(2m)}}\mathrm{d}\vt\leq\sum_{i,j}\phi^\dagger(S)\int_{\vt \in \AAL^{2*}}\one_{\prod_{p\leq L}p^{v_p(F_{ij}(\vt_p))}>T_0^{1/(2m)}}\mathrm{d}\vt.
\end{equation*} 
Hence, we estimate \eqref{eq:sum_EFS_bound} further by
$$\leq 4^{\pi(L)}\sum_{i,j}\phi^\dagger(S)\int_{\AAL^{2*} } \sum_{\b F\in\cFZ(H)} \one_{\prod_{p\leq L}p^{v_p(F_{ij}(\vt_p))}>T_0^{1/(2m)}}  \mathrm d\b t.$$
We conclude by bounding the integral over $\AAL^{2*}$ by the supremum of the integrand times the measure of $\AAL^{2*}$, which is $1/\phi^\dagger(S)$. 
\end{proof}
We next show that for fixed $\vt\in\AAL^{2*}$, the exponents
$v_p(F_{ij}(\vt_p))$ can be bounded
  individually for most of the $\b F\in\cFZ(H)$. 

\begin{lemma}\label{lem:wdywydb?} 
Fix $d\in \N$, $W>1$ and $\vt\in\AAL^{2*}$. Then the number of binary integer forms $F$
of degree $d$ with $h(F)\leq H$, such that 
there is a prime $p\leq L$ with $p^{v_p(F(\vt_p))}> W$ is
  $$\ll H^{d+1}\frac{L}{\log L}\left(\frac{1}{W} + \frac{1}{H}\right),$$
  where the implied constant depends only on $d$.
  \end{lemma} 
 
  \begin{proof}  For a 
  prime $p\leq L$, denote by $a(p)\in \N$
the least integer satisfying $p^{a(p)} > W$. We claim that the number of forms $F$ such that $p^{a(p)}$ divides $F(\b t_p)$, is $\ll H^d(H p^{-a(p)}+1)$.

Indeed, write $F(\b t_p)=\sum_{j=0}^{k}c_{j}t_1^jt_2^{k-j}$ with $c_j\in\ZZ\cap[-H,H]$. We assume that $t_2\in\ZZ_p^\times$, the other case is symmetric. Then for each fixed 
$c_1,\ldots,c_{k}$, the congruence $F(\b t_p)\equiv 0\bmod{p^{a(p)}}$
  has a unique solution $c_0$ modulo $p^{a(p)}$, which implies the claimed bound.
  
By the union bound, the number of $F$ as in the statement of the lemma is  \[\ll H^d
  \sum_{p\leq L} \left(\frac{H}{p^{a(p)}}+1\right) \ll 
 H^{d+1} \frac{L}{W \log L } +H^{d}\frac{L}{\log L}.\qedhere\]
 \end{proof}

For $x=(x_p)_{p\leq L}\in \AAL$ we define 
$$
\omega_L(x):=\#\{p\leq L: x_p\in p\Z_p\}=\#\{p\leq L: x\in p\AAL\}.$$
Given any $\vt\in\AAL^{2*}$, we show that the value of $\omega_L(F(\b t) )$  is also small for random forms $F$.
 \begin{lemma}\label{lem:Z_p} Fix $d\in \N,M>0$ and $\vt\in\AAL^{2*}$.  
Then the number of binary integer forms $F$
of degree $d$ with $h(F)\leq H$, such that  $\omega_L(F(\b t)) >  M$ 
is $\ll H^{d+1} \mathrm e^{-M}(\log L)^2$, where the implied constant depends 
only on $d$.
\end{lemma}

\begin{proof} If $\omega_L(x)>M$ then $1< \mathrm e^{-M} 3^{\omega_L(x)}$, hence, 
the number of $F$ in the lemma is at most 
$$\mathrm e^{-M} \sum_{ h(F)\leq H   }3^{\omega_L(F(\vt))},$$
where the sum is over integer binary forms $F$ of degree $d$ with $h(F)\leq H$.
Let $W:=\prod_{p\leq L}p$. For $x\in \AAL$ we have $ 3^{\omega_L(x)}=
\sum_{s\mid W} 2^{\omega(s)}\mathds 1_{s\AAL}(x)$, thus, 
\begin{equation}\label{eq:small_itervals_estimate_1}\sum_{ h(F)\leq H}3^{\omega_L(F(\vt))}
=\sum_{s\mid W}2^{\omega(s)} \sum_{ h(F)\leq H}\mathds 1_{s\AAL}(F(\vt)).
\end{equation}
By Lemma~\ref{lem:count_forms_modulo_p_m} and the Chinese remainder theorem,$$
\sum_{ h(F)\leq H}\mathds 1_{s\AAL}(F(\vt))=\sum_{\substack{ 
g \in (\Z/s\Z)[\b u]\text{ form}\\\deg(g)=d,\ g (\vt)=0}} 
\sum_{\substack{h(F)\leq H\\F\equiv g\bmod s}}1\ll \sum_{\substack{ 
g \in (\Z/s\Z)[\b u]\text{ form}\\\deg(g)=d,\ g (\vt)=0}}
 \frac{H^{d+1}}{s^{d+1}} =\frac{H^{d+1}}{s }  ,$$ as
 $s \leq W\ll H$ for large enough $H$. Injecting this 
 into~\eqref{eq:small_itervals_estimate_1}, we obtain the bound
 \[
\ll H^{d+1} \sum_{s\mid W} \frac{2^{\omega(s)}}{s}= H^{d+1} 
\prod_{p\leq L} \left(1+\frac{2}{p} \right) \ll H^{d+1} (\log L)^2.
\qedhere\]\end{proof}
Using the last two lemmas, we can bound the cardinality of 
$\b F$ in the right-hand side of Lemma~\ref{lem:sing_dddsing_L2}, obtaining the following result.

\begin{proposition}\label{lem:wdywydb?2} Fix $\psi\in (0,1)$, let $H>1$ and assume that $T_0=H^\psi$. Then
 $$\frac{1}{|\cFZ(H)|}
\sum_{\b F\in\cFZ(H)}
 |\mathfrak S(\b F)-\hSing(\b F)|^2
 \ll \frac{1}{L^2},$$  where the implied constant
  depends only on $m_1,m_2,m_3$, the $d_{ij}$ and $\psi$.
\end{proposition}

\begin{proof}By Lemma~\ref{lem:sing_dddsing_L2}  it suffices to bound  
\begin{equation}\label{eq:sing_series_diff_estimate}
\frac{4^{\pi(L)}}{|\cFZ(H)|} \#\l\{\b F\in\cFZ(H):\prod_{p\leq L}p^{v_p(F_{ij}(\vt_p))} 
> H^{\psi/(2m)}\r\}
\end{equation}
uniformly in $\b t\in\AAL^{2*}$ and $i,j$.
Each  such   $\b F$  for which $F_{ij}$
is not counted by Lemma \ref{lem:wdywydb?}, with $W>1$ to be chosen later, satisfies
$$H^{\psi/(2m)} < \prod_{p\leq L} p^{v_p(F_{ij}(\vt_p))}=
\prod_{\substack{ p\leq L\\ F_{ij}(\vt_p)\in p\ZZ_p }} p^{v_p(F_{ij}(\vt_p))}
\leq \prod_{\substack{ p\leq L\\F_{ij}(\vt_p)\in p\ZZ_p }} W  \leq
W^{\omega_L(F_{ij}(\vt))}.$$ Using Lemma~\ref{lem:Z_p} with $M:=(\psi\log H)/(2m\log W)$, the number of these $\b F$ is bounded by 
$$\ll H^{d+m-\frac{\psi}{2m\log W}}(\log L)^2.$$
Together with Lemma \ref{lem:wdywydb?}, this allows us to estimate the quantity in \eqref{eq:sing_series_diff_estimate} by
$$\ll 4^{\pi(L)}
L \left( H^{-\frac{\psi}{2m\log W}} +\frac{1}{W} + \frac{1}{H}\right). $$
We now choose  $W:=\exp(\sqrt{\psi/(2m)}L)$, so that $H^{\frac{\psi}{2m\log W}} = W$. Together with the estimate $\pi(L)\ll L/(\log L)$, this gives the crude bound 
\begin{equation*}
\ll  4^{\pi(L)} L \exp\l (-\sqrt{\frac{\psi}{2m}}L \r )\ll \frac{1}{L^2}.\qedhere
\end{equation*}
\end{proof}

\subsection{Proof of Theorem~\ref{thm:L^2theorem}}
\label{s:proffthrm L^2}
Recall that $L=\sqrt{\log H}$. We take
\begin{equation*}
   z:=H^{1/10},\quad T:=H^{2/10},\quad T_0:=H^{\alpha/(12d)}
\end{equation*}
in the definitions of $\ddet$, $\ddetf$ and $\hSing(\b F)$, see Definition \ref{def:dddddd}, \eqref{def:PPPstt} and  \eqref{deg...singularseries...}. By Cauchy's inequality 
we get $|\sum_{i=1}^3 z_i|^2 \leq 3 \sum_{i=1}^3|z_i|^2$, thus,
$$  | S_\b F(x)-x^2\mathfrak S(\b F)|^2 \leq 3
\l( 
| S_\b F(x)-\hS_\b F(x)|^2+ 
| \hS_\b F(x)-x^2\hSing(\b F)|^2+x^4
| \hSing(\b F)-\mathfrak S(\b F)|^2\r).$$ We control 
the terms on the right-hand side by bringing together 
Propositions~\ref{prop:circle_method result},
\ref{prop:pain and blood and tears}
and \ref{lem:wdywydb?2}, with parameters
\begin{equation*}
  \omega=\psi:=\alpha/(12d),\quad \lambda:=\min\{\omega,1-\beta\}.
\end{equation*}
The overall error term is  
\begin{equation*}
\ll \frac{x^4}{L^{1-\lambda}}+H^{\eps} x^{2d+4}
 \max \l\{z^{-2/9},z^{2/9}H^{-1}, 
   z^2T^{-2} \r\} \ll \frac{x^4}{L^\beta} + H^{\epsilon}x^{2d+4}H^{-2/90} \ll \frac{x^4}{L^{\beta}}.
 \end{equation*}
One easily checks that all the hypotheses of Propositions \ref{prop:circle_method result},
\ref{prop:pain and blood and tears}
and \ref{lem:wdywydb?2} are satisfied with our choice of parameters.
\qed

\section{The Hasse principle}\label{sec:chebyshev}
In this section we prove Theorems 
\ref{thm:main}-\ref{thm:chebychev} via Theorem \ref{thm:L^2theorem}. For simplicity, we write
 \begin{equation}\label{eq:conic_bundle_eqn_mult_Gi}
   G_i(\vt) := \prod_{j=1}^{m_i}F_{ij}(\vt)\quad\text{ for }1\leq i\leq 3,
 \end{equation}
 so that $G_i$ is a binary form of degree $d_i$ (with $G_3=1$ in case $m_3=d_3=0$), and we let 
 \begin{equation}\label{eq:conic_bundle_eqn_mult_G}
   G(\vt,\vx):=G_1(\vt)x^2+G_2(\vt)y^2-G_3(\vt)z^2.
 \end{equation}
 Hence, the variety $X_\vF$ defined in \eqref{eq:conic_bundle_eqn_mult}
 is given by the equation $G(\vt,\vx)=0$. 
 Recalling the definition of $\Phi_i$
 in \eqref{eq:def_phi_i} we 
 observe that  it is   a form of even degree
\begin{equation*}
  d_i+d_3 = \sum_{j=1}^{m_i}d_{ij} + \sum_{h=1}^{m_3}d_{3h}.
\end{equation*}
We shall give a lower bound for $\Sing(\vF)$ 
(defined in \eqref{eq:def_Sing_mult}) that holds for  
almost all $\vF \in \cFZ(H)$,  assuming that the variety $X_\vF$ has points everywhere locally.

We start with the archimedean factor $\omega_\infty(\vF)$. Recall that $L=\sqrt{\log H}$. 
\begin{lemma}\label{lem:local_factor_archimedean_mult}
  Let $\alpha\in (0,1)$. 
  The number of $\vF\in\cFZ(H)$ that satisfy $X_\vF(\RR)\neq \emptyset$, but $\omega_\infty(\vF)<(\log L)^{-1}$, is $\ll H^{d+m}/(\log L)^{\alpha}$,
  with the implicit constant depending only on $m_1,m_2,m_3$, the $d_{ij}$ and $\alpha$.
\end{lemma}

\begin{proof}
  We may assume throughout the proof that $H$, and thus $L$, is sufficiently large.
  For any $\vF\in\cF_\ZZ(H)$, let $\Phi_1,\Phi_2$ be as defined in \eqref{eq:def_phi_i}. Then $X_\vF(\RR)\neq\emptyset$ is equivalent to the existence of $\vt_0\in\RR^2\smallsetminus\{0\}$, such that $\Phi_1(\vt_0)\geq 0$ or $\Phi_2(\vt_0)\geq 0$. 
  
  Without loss of generality, by rescaling and possibly swapping the roles of
  the coordinates of $\vt_0$, it is enough to consider tuples $\vF\in\cFZ(H)$ such that 
  \begin{equation}\label{eq:local_factor_arch_mult_wlog}
\Phi_1(t_0,1)\geq 0\quad \text{ for some }\quad t_0  \in[-1,1].
\end{equation}
In this proof, by ``most'' $\vF$ we mean all $\vF\in\cFZ(H)$ with at most $\ll H^{d+m}/(\log L)^{\alpha}$ exceptions.

Let us first show that most $\vF$ that satisfy \eqref{eq:local_factor_arch_mult_wlog} will also do so with the additional restriction that $|t_0|\in [2(\log L)^{-\alpha},1]$. Indeed, otherwise one necessarily has
  \begin{align*}
    &\Phi_1(t_0,1)\geq 0 \text{ for some $t_0$ with } |t_0|< 2(\log L)^{-\alpha} \text{ and } \Phi_1(\pm 2(\log L)^{-\alpha},1)<0.
  \end{align*}
From \eqref{eq:def_phi_i}, there must then be a pair $(i,j)$ with $i\in\{1,3\}$ and $j\in\{1,\ldots,m_i\}$, and $\sigma\in\{\pm 1\}$, such that
  \begin{align*}
    &\sigma F_{ij}(t_0,1)\geq 0 \text{ for some $t_0$ with } |t_0|< 2(\log L)^{-\alpha} \text{ and } \sigma F_{ij}(\pm 2(\log L)^{-\alpha},1)<0.
  \end{align*}
  
By Lemma \ref{lem:linear_forms_volume}, the volume of such $\vF\in\cF(H)$ is $\ll H^{d+m}/(\log L)^{\alpha}$. The  
subset of $\cF(H)$
described by these linear conditions is sufficiently nice for lattice point counting, using e.g. Davenport's result \cite{MR43821}. Hence, the number of $\vF\in\cFZ(H)$ satisfying them is $\ll H^{d+m}/(\log L)^{\alpha}$.

Hence, we may restrict to tuples $\vF\in\cFZ(H)$ for which $\Phi_1(t_0,1)\geq 0$ for some $t_0$ with $|t_0|\in [2(\log L)^{-\alpha},1]$. Suppose that a tuple $\vF$ satisfies this, and also $\Phi_1(t_0+y,1)<0$ for some $y\in[-(\log L)^{-1},(\log L)^{-1}]$. Again, this implies that
\begin{equation*}
  \sigma F_{ij}(t_0,1)\geq 0\quad\text{ and }\quad \sigma F_{ij}(t_0+y,1)<0,
\end{equation*}
for some $(i,j)$ and $\sigma$ as above. Again by Lemma \ref{lem:linear_forms_volume}, the volume of such  $\vF\in\cF(H)$ is $\ll H^{d+m}/\log L$, and hence also the number of such $\vF \in\cFZ(H)$ is $\ll H^{d+m}/\log L$.

Hence, most tuples $\vF$ for which $X_\vF(\RR)\neq\emptyset$ satisfy, without loss of generality, that $\Phi_1(t,1)\geq 0$ for $t$ in a whole interval
\begin{equation*}
[t_0,t_0+(\log L)^{-1}]\subseteq [-1,-(\log L)^{-\alpha}]\cup [(\log L)^{-\alpha},1].
\end{equation*}
For each of these $\vF$, we see that   $\omega_{\infty}(\vF)$ equals  
\begin{align*}
&\int_{\cB}1+\hs{\Phi_1(\vt)}{\Phi_2(\vt)}'_\infty\mathrm{d}\vt \geq 2 \int_{\cB}\one_{\Phi_1(\vt)\geq 0}\mathrm d \vt \geq 2\int_{\substack{|t_2|\in[(\log L)^{-(1-\alpha)},1]\\|t_1/t_2|\in [(\log L)^{-\alpha},1]}}\one_{\Phi_1(t_1/t_2,1)\geq 0}\mathrm d \vt\\
                      &=2\int_{|u_1|\in[(\log L)^{-\alpha},1]}\one_{\Phi_1(u_1,1)\geq 0}\mathrm d u_1\int_{|t_2|\in [(\log L)^{-(1-\alpha)},1]}|t_2|\mathrm d t_2\geq 
\frac{ 2 (1-(\log L)^{-2(1-\alpha)})}{\log L}
\geq \frac{1}{\log L}.\qedhere
\end{align*}
\end{proof}

Let us next deal with all local factors $\omega_p(\vF)$ for not too small primes $p$. Throughout this section, we use the notation $\ZZ_p^{r*}:=\ZZ_p^r\smallsetminus p\ZZ_p^r$.

\begin{lemma}\label{lem:local_factors_large_primes}
  Let $\alpha>0$. Then
  \begin{equation*}
    \#\left\{\vF\in\cFZ(H)\where \prod_{(\log L)^\alpha < p\leq L} \omega_p(\vF)<(\log L)^{-(d+1)}\right\} \ll \frac{H^{d+m}}{(\log L)^\alpha},
  \end{equation*}
  where the implicit constant depends only on $m_1,m_2,m_3$, the $d_{ij}$ and $\alpha$.
\end{lemma} 

\begin{proof}
  We may assume that $H$, and thus $L$ is sufficiently large.
  Let $E(H)$ be the set of tuples $\vF\in\cFZ(H)$, such that at least one of the forms $F_{ij}$, $1\leq i\leq 3$, $1\leq j\leq m_i$, is zero modulo a prime $(\log L)^\alpha< p \leq L$. As $d_{ij}\geq 1$ for all $i,j$,
  \begin{equation*}
    \# E(H) \ll \sum_{i,j}\sum_{(\log L)^\alpha<p\leq L}\frac{H^{d+m}}{p^{d_{ij}+1}}\ll H^{d+m}\sum_{n>(\log L)^\alpha}\frac{1}{n^2}\ll \frac{H^{d+m}}{(\log L)^\alpha }. 
  \end{equation*}

  For $\vF\in\cFZ(H)\smallsetminus E(H)$ and $(\log L)^\alpha<p\leq L$, each of the forms 
  $G_i$, $i=1,2,3$, is non-zero modulo $p$ and therefore has 
  at most $\deg G_i=d_i$ roots in $\PP^1(\FF_p)$. Hence, there 
  are at most $(p-1)(d_1+d_2+d_3)=(p-1)d$ values 
  $\overline\vt\in\FF_p^2\smallsetminus\{0\}$ for which 
  $\Phi_1(\overline\vt)=0$ or $\Phi_2(\overline\vt)=0$. 
  Therefore, by definition of $\hs{\cdot}{\cdot}_p'$,
  \begin{equation*}
    \int_{\ZZ_p^{2*}}\hs{\Phi_1(\vt)}{\Phi_2(\vt)}'_p\mathrm{d}\vt \geq -\int_{\substack{\vt\in \ZZ_p^{2*}\\
    p\mid \Phi_1(\vt)\Phi_2(\vt) 
}}1\mathrm{d}\vt\geq -\frac{(p-1)d}{p^2}.
  \end{equation*}
  This shows that
  \begin{equation*}
    \omega_p(\vF)=1+\left(1-\frac{1}{p^2}\right)^{-1}\int_{\ZZ_p^{2*}}\hs{\Phi_1(\vt)}{\Phi_2(\vt)}'_p\mathrm{d}\vt\geq 1-\frac{d}{p}+O\left(\frac{1}{p^2}\right).
  \end{equation*}
  Therefore, any tuple $\vF\in\cFZ(H)\smallsetminus E(H)$ satisfies
\[
    \prod_{(\log L)^\alpha <p\leq L}\omega_p(\vF)\geq \prod_{(\log L)^\alpha<p\leq L}\left(1-\frac{d}{p}+O\left(\frac{1}{p^2}\right)\right) \gg (\log L)^{-d}.
\qedhere\]
\end{proof}

Next, we deal with $p$-adic factors $\omega_p(\vF)$ at small primes. We will ultimately use a version of Hensel's lemma, and to prepare for this we start with a simple lower bound in terms of the density of locally soluble fibres. For any point $b\in\PP^1(\QQ_p)$, let $X_{\vF,b}$ denote the fibre of $X_\vF\times_\QQ \QQ_p\to \PP^1_{\QQ_p}$, $((t_1:t_2),(x:y:z))\mapsto (t_1:t_2)$ above $b$. 

\begin{lemma}\label{lem:comparing_integrals}
  Let $\vF\in\cFZ$ such that $G_3\neq 0$ in $\ZZ[t_1,t_2]$. Then, for all primes $p$, 
  \begin{equation*}
    \omega_p(\vF)\geq \int_{\ZZ_p^{2*}}\one_{X_{\vF,(t_1:t_2)}(\QQ_p)\neq\emptyset}\mathrm d\vt.
  \end{equation*}
\end{lemma}

\begin{proof} 
  For $\vu=(u_1,u_2)\in\QQ_p^2$, let $Y_{\vu}\subseteq \PP^2_{\QQ_p}$ be the variety defined by $u_1x^2+u_2y^2=z^2$. For all $\vt=(t_1,t_2)\in \QQ_p^2\smallsetminus\{0\}$ with $G_3(\vt)\neq 0$, we have an isomorphism over $\QQ_p$,
  \begin{align*}
    X_{\vF,(t_1:t_2)}&\to Y_{\Phi_1(\vt),\Phi_2(\vt)}\\
    (x:y:z) &\mapsto (x:y:G_3(\vt)z).
  \end{align*}
   From this and the definition of $\hs{\cdot}{\cdot}_p'$, we see that
  \begin{align*}
    \omega_p(\vF)&\geq \int_{\ZZ_p^{2*}}1+\hs{\Phi_1(\vt)}{\Phi_2(\vt)}_p'\mathrm d\vt \\
    &\geq 
    \int_{\substack{\vt\in\ZZ_p^{2*}\\G_3(\vt)\neq 0}}\one_{Y_{(\Phi_1(\vt),\Phi_2(\vt))}(\QQ_p)\neq\emptyset}\mathrm d\vt = \int_{\substack{\vt\in\ZZ_p^{2*}\\G_3(\vt)\neq 0}}\one_{X_{\vF,(t_1:t_2)}(\QQ_p)\neq\emptyset}\mathrm d\vt.
  \end{align*}
  As $G_3\neq 0$, the condition $G_3(\vt)=0$ cuts out a hypersurface in $\AA^2_{\QQ_p}$, which has measure $0$. This shows the lemma's conclusion.
\end{proof}

Our central argument for $p$-adic factors at small primes relies on two applications of Hensel's lemma, which will allow us, for most tuples $\vF\in\cFZ(H)$, to bound from below the integral over $\ZZ_p^{2*}$ appearing in the previous lemma. Consider a polynomial $G$ as in \eqref{eq:conic_bundle_eqn_mult_G}, with forms $G_1,G_2,G_3\in\ZZ[t_1,t_2]$. 
Our first application of Hensel's lemma is straightforward, the second one is slightly more subtle.

\begin{lemma}\label{lem:hensel_application_prep}
 Let $p$ be prime, $\alpha\in\NN$, and assume that $(\vt_0,\vx_0)\in\ZZ_p^{2*}\times\ZZ_p^{3*}$ satisfies 
  \begin{align*}
&G(\vt_0,\vx_0)\equiv 0\bmod{p^{2\alpha}}, \text{ and }\\ &(G_x,G_y,G_z)(\vt_0,\vx_0)\not\equiv\vzero\bmod{p^\alpha}. 
  \end{align*}
  Then the equation $G(\vt,\vx)=0$ has solutions $\vx\in\ZZ_p^{3*}$ for every $\vt\in\ZZ_p^{2}$ that satisfies the congruence $\vt\equiv\vt_0\bmod{p^{2\alpha}}$. 
\end{lemma}

\begin{proof}
  Assume that $G_x(\vt_0,\vx_0)\not\equiv 0\bmod{p^\alpha}$; the argument with $G_x$ replaced by $G_y$ or $G_z$ is  analogous. We write $k:=v_p(G_x(\vt_0,\vx_0))$, so $k<\alpha$. For any $\vt\in\ZZ_p^{2}$ satisfying the congruence  $\vt\equiv \vt_0\bmod{p^{2\alpha}}$, we still have $G(\vt,\vx_0)\equiv 0\bmod{p^{2\alpha}}$ and $v_p(G_x(\vt,\vx_0))=k$. As $2\alpha>2k$, Hensel's lemma produces a value of $x\in\ZZ_p$, such that $x\equiv x_0\bmod{p^{2\alpha-k}}$ and $G(\vt,x,y_0,z_0)=0$. Hence, we have found solutions $\vx=(x,y_0,z_0)\in\ZZ_p^{3*}$ for every $\vt\equiv \vt_0\bmod{p^{2\alpha}}$.
\end{proof}

\begin{lemma}\label{lem:hensel_application}
  Let $p$ be prime and  $\alpha,\beta\in\NN$ with $\alpha\geq \beta$. Assume   $(\vt_0,\vx_0)\in\ZZ_p^{2*}\times\ZZ_p^{3*}$ satisfies 
  \begin{align*}
&(G_1,G_2,G_3)(\vt_0)\not\equiv\vzero\bmod{p^\beta},\\    
&G(\vt_0,\vx_0)\equiv 0\bmod{p^{2\alpha}}, \text{ and }\\ &(G_x,G_y,G_z,G_{t_1},G_{t_2})(\vt_0,\vx_0)\not\equiv\vzero\bmod{p^\alpha}. 
  \end{align*}
  Set $\gamma:=2\alpha+\beta+1+v_p(2)$. Then there is $\tilde{\vt}\in\ZZ_p^{2*}$, such that the equation $G(\vt,\vx)=0$ has solutions $\vx\in\ZZ_p^{3*}$ for every $\vt\in\ZZ_p^{2*}$ that satisfies the congruence $\vt\equiv\tilde{\vt}\bmod{p^{2\gamma}}$. 
\end{lemma}

\begin{proof}
If $(G_x,G_y,G_z)(\vt_0,\vx_0)\not\equiv\vzero\bmod{p^\alpha}$, then, as $\gamma\geq 2\alpha$, we may take $\tilde\vt=\vt_0$ by Lemma \ref{lem:hensel_application_prep}. Otherwise, we must have $(G_{t_1},G_{t_2})(\vt_0,\vx_0)\not\equiv\vzero\bmod{p^\alpha}$. Possibly exchanging the roles of $t_1$ and $t_2$, and also of $x,y,z$, we may assume that $k:=v_p(G_{t_1}(\vt_0,\vx_0))<\alpha$, and also that $G_1(\vt_0)\not \equiv 0\bmod{p^\beta}$. Write $\vt_0=(t_{0,1},t_{0,2})$. Let $x\in\ZZ_p$  such that  $x\equiv x_0\bmod{p^{2\alpha}}$ and $x\not\equiv 0\bmod{p^{2\alpha+1}}$. Then still $(x,y_0,z_0)\in\ZZ_p^{3*}$, $G(\vt_0,x,y_0,z_0)\equiv 0\bmod{p^{2\alpha}}$, and $v_p(G_{t_1}(\vt_0,x,y_0,z_0))=k$. As $2\alpha>2k$, Hensel's lemma yields a value of $\tilde{t_1}\in\ZZ_p$, such that $\tilde{t_1}\equiv t_{0,1}\bmod{p^{2\alpha-k}}$ and $G(\tilde{t_1},t_{0,2},x,y_0,z_0)=0$. Write $\tilde\vt:=(\tilde t_1,t_{0,2})$, 
$\tilde\vx:=(x,y_0,z_0)$.
As $2\alpha-k>\beta$, we still have $G_1(\tilde\vt)\not\equiv 0\bmod{p^\beta}$. Hence,
    \begin{equation*}
      v_p(G_x(\tilde\vt, 
      \tilde\vx
      )) = v_p(2G_1(\tilde\vt)x) \leq v_p(2)+\beta+2\alpha<\gamma,
    \end{equation*}
    so the desired conclusion follows from Lemma \ref{lem:hensel_application_prep} with $\vt_0=\tilde\vt$, 
    $\vx_0=\tilde\vx$
    and $\alpha=\gamma$.
\end{proof}

We now consider the coefficients of the forms $F_{ij}$ as indeterminate. That is, we write 
$S:=\ZZ[\vA]$ for the polynomial ring in variables $\vA=(A_{ijl})$ with $1\leq i\leq 3$, $1\leq j\leq m_i$, $0\leq l\leq d_{ij}$, and consider binary forms
\begin{align*}
  F_{ij}:=\sum_{l=0}^{d_{ij}}A_{ijl}t_1^{l}t_2^{d_{ij}-l} \in S[\vt],\quad \text{ with }\quad \vt=(t_1,t_2).
\end{align*}
Let $G_1,G_2,G_3\in S[\vt]$ and $G\in S[\vt,\vx]$ be as in \eqref{eq:conic_bundle_eqn_mult_Gi} and \eqref{eq:conic_bundle_eqn_mult_G}. For any $i\neq j\in \{1,2,3\}$ that satisfy $d_i,d_j\geq 1$, the polynomial
\begin{equation}\label{eq:def_Gij}
  G_{ij}:= G_i \frac{\partial G_j}{\partial t_1}\in S[\vt]\smallsetminus\{0\}
\end{equation}
is homogeneous in $\vt$ of degree $d_i+d_j-1\geq 1$. Note that in our setup we always have $d_1,d_2\geq 1$, but $d_3$ could be $0$, namely in case $m_3=0$.

Write $S[\vt]_{e}$ for the $S$-module of binary forms of degree $e$. It is free of rank $e+1$, with the standard binomial basis $t_1^e,t_1^{e-1}t_2,\ldots,t_2^e$. For any $i\neq j$ as above, the $S$-linear map 
\begin{equation*}
  S[\vt]_{d_i+d_j-2} \times S[\vt]_{d_i+d_j-2} \to S[\vt]_{2d_i+2d_j-3},\quad (U,V)\mapsto UG_{ij}+VG_{ji}
\end{equation*}
is represented with respect to the binomial bases by a $(2d_i+2d_j-2)\times (2d_i+2d_j-2)$-square matrix with entries in $S$, called the \emph{Sylvester matrix}. Recall that the resultant $\res_\vt(G_{ij},G_{ji})$ is defined as the determinant of this matrix. With this setup in place, we consider the polynomial
\begin{equation}\label{eq:def_R}
  R:=2\prod_{\substack{i<j\\d_j\neq 0}}\res_{\vt}(G_{ij}, G_{ji})\in S,
\end{equation}
which is just $2\res_\vt(G_{12},G_{21})$ in case $m_3=0$.
It is homogeneous in the variables $\vA$. As each $F_{ij}$ is irreducible in $\ZZ[A_{ij1},\ldots,A_{ijd_{ij}},\vt]$, the forms $G_{ij}$ and $G_{ji}$ have no common irreducible factors in $\Q(\vA)[\vt]$, and therefore $R\neq 0$.
 
\begin{lemma}\label{lem:resultant_lemma}
  Let $\va=(a_{ijl})\in\ZZ^{d+m}$ and let $p^\alpha$ be a prime power such that $R(\va)\not\equiv 0\bmod{p^\alpha}$. Then every $(\vt_0,\vx_0)\in\ZZ_p^{2*}\times\ZZ_p^{3*}$ 
  satisfies
  \begin{equation*}
  (G_x,G_y,G_z,G_{t_1})(\va,\vt_0,\vx_0)\not\equiv\vzero\bmod{p^\alpha}\quad \text{ and }\quad (G_1,G_2)(\va,\vt_0)\not\equiv\vzero\bmod{p^\alpha}.
\end{equation*}
\end{lemma}

\begin{proof}
  Suppose that $(\vt_0,\vx_0)\in\ZZ_p^{2*}\times\ZZ_p^{3*}$ 
  does not satisfy the lemma's conclusion. We will show that $R(\va)\equiv 0\bmod{p^\alpha}$. Writing $\vt_0=(t_{0,1},t_{0,2})$, fix $r\in\{1,2\}$ such that $p\nmid t_{0,r}$. By Cramer's rule, for all $i,j\in\{1,2,3\}$ with $i<j$ and $d_j\neq 0$, there are $U,V\in S[t]_{d_i+d_j-2}$, such that
  \begin{equation}\label{eq:res_linear_comb}
  t_r^{2d_i+2d_j-3}\res_{\vt}(G_{ij},G_{ji})=UG_{ij}+VG_{ji}\quad\text{ in }\quad S[\vt].
  \end{equation}

  If $(G_1,G_2)(\va,\vt_0)\equiv \vzero\bmod{p^\alpha}$, then  
  $p^\alpha$ divides $G_{12}(\va,\vt_0)$ and $G_{21}(\va,\vt_0)$. 
  By \eqref{eq:res_linear_comb} 
  and our assumption that $p\nmid t_{0,r}$, it follows that $p^\alpha\mid \res_\vt(G_{12},G_{21})(\va)$, and therefore $p^\alpha\mid R(\va)$.

  Now assume that $(G_x,G_y,G_z,G_{t_1})(\va,\vt_0,\vx_0)\equiv\vzero\bmod{p^\alpha}$. We first proceed under the assumption that $m_3=0$, so $G_3=1$. As $\vx_0\in\ZZ_p^{3*}$, at least one of $x_0,y_0,z_0$ is not divisible by $p$. If $p\nmid z_0$, then the hypothesis that $p^\alpha\mid G_z(\va,\vt_0,\vx_0)=-2z_0$ implies that $p^\alpha$ divides $2$, and thus $R(\va)$.
  
  If $p\nmid x_0$, let $k:= v_p(y_0)$. Then from $G_x=2xG_1$, we see that $v_p(2G_1(\va,\vt_0))\geq \alpha$, and thus in particular $v_p(2G_{12}(\va,\vt_0))\geq\alpha$. Similarly, we get $v_p(2G_2(\va,\vt_0))\geq \max\{0,\alpha-k\}$. Moreover, from
  \begin{equation*}
     G_{t_1} = \frac{\partial G_1}{\partial t_1}x^2+\frac{\partial G_2}{\partial t_1}y^2
  \end{equation*}
  and $p^\alpha\mid G_{t_1}(\va,\vt_0)$, we obtain $v_p(\partial G_1/\partial t_1(\va,\vt_0))\geq \min\{\alpha,2k\}$. Therefore, 
  \begin{equation}\label{eq:resultant_computations}
  v_p(2G_{21}(\va,\vt_0)) =  v_p(2G_2(\va,\vt_0))+v_p\left(\frac{\partial G_1}{\partial t_1}(\va,\vt_0)\right)\geq \max\{0,\alpha-k\}+\min\{\alpha,2k\}\geq \alpha.
\end{equation}
  By \eqref{eq:res_linear_comb}, as $p\nmid t_{0,r}$, this shows again that $p^\alpha$ divides $2\res_\vt(G_{12},G_{21})(\va)$, and thus $R(\va)$. The case where $p\nmid y_0$ is analogous, which concludes our proof under the assumption that $m_3=0$.

  Hence, we now assume that $m_3\geq 1$, and thus $d_3\geq 1$. In this case, the roles of $G_1,G_2,-G_3$ are exchangable, so we may assume without loss of generality that 
  \begin{equation*}
  0=v_p(x_0)\leq v_p(y_0)\leq v_p(z_0).
  \end{equation*}
  Write $k:=v_p(y_0)$. Similarly as above, we see that $v_p(2G_1(\va,\vt_0))\geq \alpha$, which implies that $p^\alpha\mid 2G_{12}(\va,\vt_0)$, and $v_p(2G_2(\va,\vt_0))\geq \max\{0,\alpha-k\}$. As
  \begin{equation*}
  G_{t_1}=\frac{\partial G_1}{\partial t_1}x^2+\frac{\partial G_2}{\partial t_1}y^2-\frac{\partial G_3}{\partial t_1}z^2,
\end{equation*}
we get that $v_p(\partial G_1/\partial t_1(\va,\vt_0))\geq \min\{\alpha,2k\}$, and thus again \eqref{eq:resultant_computations} holds.
With \eqref{eq:res_linear_comb} this shows again that $p^\alpha\mid R(\va)$, as desired.
\end{proof}

We will use the following result of Pierce, Schindler and Wood.

\begin{lemma}\cite[Lemma 4.10]{MR3551849}\label{lem:zeros_mod_prime_power_bound}
  Let $n\in\NN$ and $P\in\ZZ[x_1,\ldots,x_n]$ be a non-zero homogeneous polynomial with $\deg P=D$. Then, for any prime power $p^\alpha$,
  \begin{equation*}
   \# \{\vx\in (\Z/p^\alpha\Z)^n\where P(\vx)= 0\} \ll p^{\alpha(n-1/D)},
  \end{equation*}
  with the implied constant depending only on $P$.
\end{lemma}

Let $\cFZELS(H)$ be the set of all tuples $\vF\in\cFZ(H)$, such that the corresponding variety  $X_\vF$ given by $G(\vt,\vx)$ has real points and $\QQ_p$-points for every prime $p$. The latter condition means that for every prime $p$ there is a solution $(\vt_p,\vx_p)\in\ZZ_p^{2*}\times\ZZ_p^{3*}$ of the equation $G(\vt_p,\vx_p)=0$. 

\begin{lemma}\label{lem:local_factors_small_primes}
  Let the positive number $\delta$ be sufficiently large in terms of $m_1,m_2,m_3$ and the $d_{ij}$. For any $M,H\geq 1$ such that $M^{\delta/6}\leq H$, we have 
  \begin{equation*}
    \#\left\{\vF \in \cFZELS(H)\where \prod_{p\leq M}\omega_p(\vF)< \mathrm e^{-4\delta M} \right\} \ll H^{d+m}\cdot 2^{-\delta/(8D)},
  \end{equation*}
  where $D$ is the degree of the homogeneous polynomial $R\in\ZZ[\vA]$ defined in \eqref{eq:def_R}. The implied constant depends only on  
  $m_1,m_2,m_3$ and the $d_{ij}$.
\end{lemma}

\begin{proof}
  We take $\beta=\alpha := \lfloor(\delta-4)/6\rfloor$, assuming $\delta$ to be large enough so that $\alpha,\beta\geq\delta/8\geq 1$. Let $\vF\in\cFZELS(H)$ have coefficients $\va=(a_{ijl})_{ijl}\in\ZZ^{d+m}$. Suppose that $R(\va)$ is not divisible by $p^\alpha$ for any prime $p\leq M$. For each prime $p\leq M$, let $(\vt_p,\vx_p)\in\ZZ_p^{2*}\times\ZZ_p^{3*}$ be a solution to $G(\vt,\vx)=0$. By Lemma \ref{lem:resultant_lemma}, the hypotheses of Lemma \ref{lem:hensel_application} are satisfied, and thus, using Lemma \ref{lem:comparing_integrals},
  \begin{equation*}
    \omega_p(\vF)\geq \int_{\ZZ_p^{2*}}\one_{X_{\vF,(t_1:t_2)}(\QQ_p)\neq\emptyset}\mathrm d\vt \geq  p^{-4(2\alpha+\beta+1+v_p(2))}\geq p^{-2\delta}.
  \end{equation*}
  Then
  \begin{equation*}
    \prod_{p\leq M}\omega_p(\vF)=\exp\left(\sum_{p\leq M}\log \omega_p(\vF)\right)\geq \exp\left(-2\delta\sum_{p\leq M}\log p\right)\geq \mathrm e^{-4\delta M}.
  \end{equation*}
  Hence, every $\vF$ in the set under investigation must have 
  coefficients $\va\in\ZZ^{d+m}$ with $|\va|\leq H$ and 
  $R(\va)\equiv 0\bmod{p^\alpha}$ for some $p\leq M$. Using Lemma 
  \ref{lem:zeros_mod_prime_power_bound}, we see that for each 
  individual $p\leq M$, the cardinality of such $\va$ is bounded by
  \begin{equation*}
    \sum_{\substack{\vu\bmod{p^\alpha}\\R(\vu)\equiv 0\bmod{p^\alpha}}}\#\left\{\va\in\ZZ^{d+m}\where |\va|\leq H\text{ and }\va\equiv \vu\bmod{p^\alpha}\right\}\ll\frac{H^{d+m}}{p^{\alpha/D}}.
  \end{equation*}
  We assume $\delta$ to be large enough so that $\alpha/D\geq 2$. Then summing the previous result over all $p\leq M$ yields the total bound\[
    \ll H^{d+m}\sum_{p\leq M}p^{-\alpha/D} \ll H^{d+m}2^{-\alpha/D+2}\sum_{p}p^{-2} \ll H^{d+m}\cdot 2^{-\delta/(8D)}.
\qedhere\]
\end{proof}

\begin{proposition}\label{prop:singular_series_lower_bound}
  Let $\alpha\in(0,1)$  and let $\delta>1$ be sufficiently large in terms of $m_1,m_2,m_3$ and the $d_{ij}$. Let $H>1$  
  and suppose that
  $(\log L)^{\alpha\delta/6}\leq H$. Then  
  \begin{equation*}
    \#\left\{\vF \in\cFZELS(H)\where \Sing(\vF)\leq \mathrm e^{-6\delta(\log L)^\alpha} \right\} \ll H^{d+m}
    \left(2^{-\delta/(8D)}+(\log L)^{-\alpha}\right),
  \end{equation*}
  where $D$ is the degree of the polynomial $R\in\ZZ[\vA]$ in \eqref{eq:def_R}. The implied constant depends only on $m_1,m_2,m_3$, $d_{ij}$ and $\alpha$.
\end{proposition}

\begin{proof}
  By Lemmas~\ref{lem:local_factor_archimedean_mult}, 
  \ref{lem:local_factors_large_primes} and 
  \ref{lem:local_factors_small_primes}
  with $M=(\log L)^\alpha$, we see that
  \begin{equation*}
    \Sing(\vF)=\frac{\omega_\infty(\vF)}{\zeta(2)}\prod_{p\leq(\log L)^\alpha}\omega_p(\vF)\prod_{(\log L)^\alpha<p\leq L}\omega_p(\vF)\geq \mathrm e^{-4\delta(\log L)^\alpha}
    (\log L)^{-d-2}\gg \mathrm e^{-5\delta (\log L)^\alpha}
  \end{equation*}
  holds for all but $\ll H^{d+m}(2^{-\delta/(8D)}+
  (\log L)^{-\alpha})$ tuples $\vF\in \cFZELS(H)$. This implies the proposition's statement.
\end{proof}

\subsection{Proof of Theorem 
\ref{thm:chebychev}}\label{sec:chebyshev_multiple}
Recall that $L=\sqrt{\log H}$ and let $\alpha$ be as in the theorem. With quantities $\delta,\eta>0$ to be chosen later and $x:=H^{1/(100d)}$, we consider the exceptional sets 
\begin{align*}
&\c E_0 := \{\b F\in \cFZ(H)\where X_{\vF}\text{ is not a conic bundle surface}\},\\
&\c E_1 :=\{\b F \in \cFZ(H)\where
|S_\b F(x)-x^2\Sing(\vF)|\geq \eta x^2  \},
\\
&\c E_2 :=\{\vF \in\cFZELS(H)\where \Sing(\vF)\leq \mathrm e^{-6\delta(\log L)^\alpha}\},
\end{align*} 
and $\c E:=\c E_0\cup\c E_1\cup\c E_2$. For $\vF\in\cFZ(H)$ to lie in $\c E_0$, the binary form $\Phi:=\prod_{i=1}^3\prod_{j=1}^{m_i}F_{ij}$ has to be equal to zero or have multiple irrducible factors. If this holds, then $\Phi$ is either divisible by $t_2$, or the resultant $\res_\vt(\Phi,\partial\Phi/\partial t_1)$ is zero. The former condition is clearly satisfied by $\ll H^{d+m-1}$ tuples $\vF\in\cFZ(H)$, as then at least one of the $F_{ij}$ has to be divisible by $t_2$. For the latter condition, we consider the coefficients of $\vF$ again as indeterminates $\vA=(A_{ijl})$, as we did earlier in this section. As the form $\Phi(\vA,\vt)$ is separable in $\Q(\vA)[\vt]$, the resultant is a non-zero polynomial in $\ZZ[\vA]$. Hence, there are at most $\ll H^{d+m-1}$ tuples $\vF\in\cF_Z(H)$ for which it evaluates to zero. We have thus shown that $|\c E_0|\ll H^{d+m-1}\asymp\cFZ(H)/H$.

If  $\b F\in \c E_1$ then   
$1\leq \eta^{-2} x^{-4}|S_\b F(x)-x^2\Sing(\vF)|^2$, thus,
 by Theorem \ref{thm:L^2theorem} (applied with, e.g., $\beta=1/2$, $\alpha=1/200$),
$$\frac{|\c E_1|}{|\c F_\Z(H)|}
\leq  \frac{1}{\eta^2x^4}
\sum_{\b F \in \c F_\Z(H)} \frac{ 
| S_\b F(x)-x^2\mathfrak S(\b F)|^2 }{|\c F_\Z(H)|}
\ll  \frac{1}{\eta^2(\log H)^{1/4}}
=\frac{1}{\eta^2 L^{1/2}}.$$

Finally, for sufficiently large $\delta$ 
with $(\log L)^{\alpha\delta/6}\leq H$,
Proposition \ref{prop:singular_series_lower_bound}
shows that $$\frac{|\c E_2|}{|\c F(H)|} \ll 
(\log L)^{-\alpha}+2^{-\delta/(8D)},$$ 
and thus in total
\begin{equation}\label{eq:unionboundddd} |\c E|\ll |\cFZ(H)|\left(\frac{1}{H}+\frac{1}{\eta^2 \sqrt{L}}
+\frac{1}{(\log L)^{\alpha}}
+\frac{1}{2^{\delta/(8D)}}\right).
\end{equation} 

For $\b F\in\cFZ(H)\smallsetminus \c E$ the hypersurface $X_\b F$ is a conic bundle surface, and whenever it is everywhere locally soluble we have 
\begin{equation}\label{eq:unionboundddd123} 
  S_\b F(x) >x^2(\mathfrak S(\b F)-\eta)
>x^2(\mathrm e^{-6\delta(\log L)^\alpha}-\eta)=x^2\eta,
\end{equation}
where for the last equality we have now specified our choice of
$$ \eta := \frac{1}{2}
\mathrm e^{-6\delta(\log L)^\alpha}. $$

Now we choose $\delta$ so that the two middle summands in the bound in \eqref{eq:unionboundddd} agree, i.e.
$$ \delta:=\frac{\log\l(\frac{L^{1/2}}{4(\log L)^\alpha}\r)}{12(\log L)^\alpha}.$$
In light of the above definition of $\eta$, this is indeed equivalent to 
$\eta^2 \sqrt{L}=(\log L)^\alpha$, and moreover $\delta$ grows with $L$, so it will be sufficiently large for the above application of Proposition \ref{prop:singular_series_lower_bound} if only $H$ is sufficiently large. It is then easily verified that $2^{\delta/8D} \geq (\log L)^\alpha$
and  $(\log L)^{\alpha\delta/6}\leq H$. Hence, from
\eqref{eq:unionboundddd} we get that
\begin{equation}\label{eq:exceptional_set_final_estimate}
|\c E|\ll\frac{ |\cFZ(H)|}{(\log L)^\alpha}
\ll \frac{ |\cFZ(H)|}{(\log \log H)^\alpha}
.
\end{equation}

Let $\vF\in\cFZ(H)\smallsetminus\c E$ and assume that the conic bundle surface $X_\vF$ is everywhere locally soluble.
Since $\eta =(\log L)^{\alpha/2} L^{-1/4}$, we see from 
\eqref{eq:unionboundddd123} and our choice of $x=H^{1/(100d)}$ that
$$S_\b F(x)>
\frac{x^2(\log L)^{\alpha/2}}{ L^{1/4} }
\gg \frac{x^2(\log \log H)^{\alpha/2}}{ (\log H)^{1/8} }
\gg \frac{x^2(\log \log x)^{\alpha/2}}{ (\log x)^{1/8} }\gg_\epsilon x^{2-\epsilon},$$
for arbitrarily small $\epsilon>0$. On the other hand, as $ \updelta(\b t) \ll_\eps |t_1 t_2|^\eps 
+1
$,
one has 
$$S_\b F(x) \ll_\eps H^\eps \#\{\vt\in \P^1(\Q)\where  H(\vt)\leq x,\ 
(X_{\b F})_\vt \textrm{ has a }\Q\textrm{-point}\}, $$
where $H$ is the standard Weil height. Hence, we conclude that
$$  \#\{\vt\in \P^1(\Q)\where H(\vt)\leq x,\ 
(X_{\b F})_{\vt} \textrm{ has a }\Q\textrm{-point}\}
\gg_\epsilon x^{2}H^{-\epsilon}\geq H^{\gamma/d},$$
if only $\epsilon$ was chosen small enough in terms of $\gamma$ and $d$.

Finally, in order to remove the implicit constants in $\gg_\eps$ above and in \eqref{eq:exceptional_set_final_estimate}, we apply the proof with slightly larger values of $\alpha$ and $\gamma$, e.g. $\alpha':=(1+\alpha)/2$ and $\gamma':=(1/50+\gamma)/2$, and choose $H$ sufficiently large. 
\qed 
\subsection{Proof of Theorem~\ref{thm:main_affine}}\label{sec:proof_main_affine}
For $i=1,2,3$, let $d_i:=\sum_{j=1}^{m_i}d_{ij}$. We assume first that not all of $d_1,d_2,d_3$ have the same parity. In this case, it is easy to exhibit the existence of rational points, and hence the Hasse principle, directly. Let us assume that $d_1\equiv d_2\equiv d_3+1\bmod 2$, the other cases working analogously. Using resultants, similarly as in \S\ref{sec:chebyshev_multiple}, one easily sees that for $100\%$ of tuples $(f_{ij})_{i,j}$, their product $\prod_{i,j}f_{ij}$ is  
separable, 
and moreover $\deg f_{ij}=d_{ij}$ for all $i,j$.
We claim that then every smooth projective model of \eqref{eq:general_affine} has rational points. By Lang-Nishimura, it suffices to consider a specific model. For this, we write $d_i=2a_i+e-\one_{i=3}$ with $e\in\{0,1\}$ and take the conic bundle surface $X_\vG$ in $\FF(a_1,a_2,a_3)$ defined in \S\ref{sec:def_conic_bundle} with
$G_i$ the homogenisation of $\prod_{j=1}^{m_1}f_{ij}$ for $i=1,2$ and $G_3$ equal to $t_2$ times the corresponding homogenisation.  
Note that $G_1G_2G_3$ is separable, so $X_\vG$ is indeed a conic bundle. 
Now we simply observe that the fibre of $X_{\vG}$ over $(1:0)$ is the degenerate conic given in $\PP^2_\QQ$ by $G_1(1,0)x^2+G_2(1,0)y^2=0$, which has the rational point $(0:0:1)$.

Now let $d_1\equiv d_2\equiv d_3\bmod 2$. For each tuple $(f_{ij})_{i,j}$, we let $(F_{ij})_{i,j}$ consist of the corresponding homogenisations $F_{ij}(t_1,t_2):=t_2^{d_{ij}}f_{ij}(t_1/t_2)$. When the $(f_{ij})$ run through tuples of integer polynomials with degrees bounded by $d_{ij}$ and coefficients bounded by $H$ in absolute value, then the $(F_{ij})$ run exactly through the elements of $\cFZ(H)$.  
Whenever the hypersurface $X_\vF$ is a conic bundle surface, it is a smooth projective model of  \eqref{eq:general_affine}.
Hence, the conclusion of Theorem \ref{thm:main} implies that of Theorem \ref{thm:main_affine}.\qed

\appendix

\section{Counting weighted lattice points}\label{app:aaa}
 
In this appendix we collect a few rather standard results regarding 
volumes,
lattice point counting and comparing sums to integrals.

Our first lemma says that if two 
linear forms have almost equal 
corresponding coefficients then they should take different sign with low 
probability. Recall that for a form $L$ in $d$ variables with coefficients 
in $\RR$ we denote by $h(L)$ the maximum modulus of its coefficients. 
\begin{lemma}\label{lem:linear_forms_volume}
  Let $L_1,L_2$ be nonzero linear forms on $\RR^d$ and $H>0$. Then
  \begin{equation*}
    \vol\{\vx\in[-H,H]^d\where L_1(\vx)\geq 0 \text{ and }L_2(\vx)\leq 0\}\ll \frac{H^d}{\max\{h(L_1),h(L_2)\}}h(L_1-L_2), 
  \end{equation*}
  with the implied constant depending only on $d$.
\end{lemma} 
\begin{proof}
  Renormalising $\vx$ and the forms $L_i$, we may assume without loss of
  generality that $H=1$ and $\max\{h(L_1),h(L_2)\}=h(L_1)=1$. 
The set under consideration is contained in the set of 
$\vx\in [-1,1]^d$ where $ 0\leq L_1(\vx)\leq (d+1) h(L_1-L_2)$,
because  $$0\leq L_1(\b x)=L_1(\b x)-L_2(\b x)+L_2(\b x) \leq
 L_1(\b x)-L_2(\b x) \leq (d+1) h(L_1-L_2).$$ As $h(L_1)=1$, 
 the volume of this set is at most $ (d+1) h(L_1-L_2)$.
\end{proof}

The proof of Davenport's lattice point counting theorem \cite{MR43821} can be modified to allow lattice points weighted by Lipschitz functions. Below, we do so in a simple case.

\begin{lemma}\label{lem:weighted_davenport_lemma}
  Let $d,h\in\NN$, $c>0$, $H\geq 1$ and $\B\subseteq[-H,H]^d$ a compact domain such that every line parallel to one of
  the coordinate axes in $\RR^d$ intersects $\B$ in at most $h$ intervals.
  
  Let $\vu\in\RR^d$, and let $\omega:\RR^d\to [-1,1]$ satisfy $\abs{\omega(\vx)-\omega(\vy)}\leq c\abs{\vx-\vy}$ for all $\vx, \vy\in\c B$. Then
  \begin{equation*}
    \sum_{\vn\in(\vu+\ZZ^d)\cap\B}\omega(\vn) = \int_\B \omega(\vx)\mathrm{d}\vx +
    O(H^{d-1}(cH+1)),
  \end{equation*}
  with the implicit constant depending only on $d,h$.
\end{lemma}

\begin{proof}  
  When $d=1$, the domain $\B$ is by hypothesis a union of at most $h$ intervals in $[-H,H]$. For each such interval $I$,
  \begin{equation*}
    \int_I\omega(x)\mathrm{d}x = \sum_{n\in
      I\cap(u+\ZZ)}\int_{n}^{n+1}\left(\omega(n)+O(c)\right)\mathrm{d}x + O(1) =
    \sum_{n\in I\cap(u+\ZZ)}\omega(n) + O(cH+1).
  \end{equation*}
  Summing both sides over at most $h$ intervals proves the base case. 
  Now suppose the lemma holds for $d-1$ and write $\vx = (\vx',x)$ 
  with $\vx'\in\RR^{d-1}$. Then, similarly writing  $\vu=(\vu',u)$,
  \begin{align*}
    \int_\B \omega(\vx)\mathrm{d}\vx = \int_{-H}^H\left(\int_{\cB_{x}}\omega(\vx',x)\mathrm{d}\vx'\right)\mathrm{d}x
    =\int_{-H}^H\left(\sum_{\vn'\in
                            (\vu'+\ZZ^{d-1})\cap \cB_x}\hspace{-0.8cm}\omega(\vn',x)
                            + O(H^{d-2}(cH+1))\right)\mathrm{d}x,
  \end{align*}
  where the sections $\cB_{x}:=\{\vx'\in\RR^{d-1}\where (\vx',x)\in\cB\}$ still intersect every line parallel to one of the coordinate axes in at most $h$ intervals. Integrating the error term gives an acceptable bound. Exchanging sum and integral in the main term gives
  \begin{align*}
         \sum_{\vn'\in(\vu'+\ZZ^{d-1})\cap [-H,H]^{d-1}}\int_{\cB_{\vn'}}\omega(\vn',x)\mathrm d x,
   \end{align*}  
   where the sections $\cB_{\vn'}:=\{x\in\RR\where (\vn',x)\in\cB\} \subseteq[-H,H]$ again satisfy the Lemma's hypotheses in case $d=1$. Hence, we conclude by applying the base case to each integral over $\cB_{\vn'}$, turning it into a sum over $(u+\ZZ)\cap\cB_{\vn'}$ plus an error term that we can sum trivially.  
\end{proof}

We now use this lemma to estimate certain arithmetic sums by real and $p$-adic integrals.
For $K\in\NN$, let $\phi^\dagger(K):=\prod_{p\mid K}(1-p^{-2})^{-1}$ and
$$(\ZZ/K\ZZ)^{2*}:=\{\vt\in(\ZZ/K\ZZ)^2\where \gcd(t_1,t_2,K)=1\}.$$ 
\begin{lemma}\label{lem:V_sum_n_appendix} Let $K,h\in\NN$, $\gamma\in(0,1]$
and $\c B\subset ([-1,1]\smallsetminus(-\gamma,\gamma))^2$  
be a compact set such that every line parallel to one of the coordinate axes in $\RR^2$ intersects $\c B$ in 
at most $h$ intervals.

Let $P:\Z^2\to [-1,1]$ be a function satisfying
\begin{equation}\label{assumption:periodK}\b n \equiv \b t \md K
\Rightarrow P(\b n)=P(\b t).\end{equation} 
Assume that $\omega:\R^2\to[-1,1]$  satisfies the conditions
\begin{equation}\label{assumption:rimint}
\omega(a\vu)=\omega(\vu)\text{ for all }a>0\text{ and }\vu\in\RR^2
,\end{equation}
\begin{equation}\label{assumption:fflipstz}
|\omega(\b u)-\omega(\b v)|\ll\frac{\abs{\vu-\vv}}{\max\{\abs{\b u},\abs{\b v}\}}
\text{ for all }\b u,\b v\in\RR^2\smallsetminus\{\vzero\}.
\end{equation}  
Then for $x\geq 1$ we have 
  $$ \sum_{\substack{\vn\in  \Z^2 \cap x\c B 
  \\ \gcd(n_1,n_2)=1 }}\omega(\vn) P(\b n)
  = \frac{x^2\phi^\dagger(K)} {\zeta(2)K^2}
\int_{\c B} \omega(\vu)\mathrm d\vu  
\sum_{\b t \in (\Z/K\Z)^{2*} } P(\b t)+ O\left(\frac{K^3 x (\log x)}{\gamma}\right),
$$ where the implied constant depends only on 
$h$ and the implied constant in~\eqref{assumption:fflipstz}.
\end{lemma}

\begin{proof}
By assumption \eqref{assumption:periodK} and inclusion-exclusion, the sum on 
the left-hand side is equal to \begin{equation*} 
  \sum_{\b t \in (\Z/K\Z)^{2*} } P(\b t)
  \sum_{\substack{d   \leq
  x \\ \gcd(d,K)=1}}\mu(d)
\sum_{\substack{\vn\in  x\c B, d\mid \vn\\
\vn\equiv \b t\md K}} \omega(\vn).
  \end{equation*} 
For each such $\vt,d$, the Chinese remainder theorem yields 
$\vn_0\in\ZZ^2$ with $\vn_0\equiv \b t \md K$ and 
$\vn_0 \equiv \vzero \md d$, so the sum over $\b n $ becomes  $$
\sum_{\substack{\vn \in x\c B\\\vn\equiv \vn_0\md{Kd}}}\omega(\vn)=
\sum_{\vm \in (\ZZ^2+\frac{\vn_0}{Kd})\cap 
\frac{x}{Kd}\c B}\omega(\vm) $$by~\eqref{assumption:rimint}.
Note that if $\b u \in \frac{x}{Kd} \c B$ then 
$|\b u| \geq \gamma x/(Kd)$, hence ~\eqref{assumption:fflipstz}
yields$$ |\omega(\b u)-\omega(\b v)| \ll \frac{|\b u-\b v|}{\gamma x/(Kd)}.$$ Thus, by
Lemma~\ref{lem:weighted_davenport_lemma}
with $c\asymp Kd/(\gamma x)$ and 
$H=1 + x/(Kd)$, we get
\begin{equation*}
\sum_{\vm \in (\ZZ^2+\frac{\vn_0}{Kd})\cap 
\frac{x}{Kd}\c B}\omega(\vm) =
\frac{x^2}{K^2d^2}    \int_{\c B}\omega(\vu)\mathrm d \vu 
    + O\left( 
    \frac{x}{Kd\gamma} + \frac{Kd}{\gamma x}\right).
  \end{equation*}
Summing the error term over all $\vt$ and $d\ll x$ yields the total bound
$$\ll \frac{K x (\log x)}{\gamma}+\frac{K^3 x}{\gamma }  
\ll \frac{K^3 x (\log x)}{\gamma} 
.$$
 The sum of the main term over 
$d$ yields\begin{align*}
  &\frac{x^2}{K^2}\Big(\sum_{\substack{d 
  \leq
  x\\ \gcd(d,K)=1}}\frac{\mu(d)}
  {d^2}\Big)\int_{\c B}\omega(\vu)
  \mathrm d\vu.
\end{align*} Completing the sum over $d$ can be done at a cost 
of an insignificant error term of size $O(x/K^2)$.
\end{proof}

\begin{lemma}\label{lem:V_sum_n2345appendix} 
Let $K,h,\gamma,\c B$  be as in Lemma~\ref{lem:V_sum_n_appendix}.
Let $P:\Z^4\to [-1,1]$ such that for each $\b n \in \Z^2$ both  functions  $P(\b n, \cdot)$ and $P(\cdot,\b n)$ 
satisfy~\eqref{assumption:periodK}. 
Let $\omega:\R^4\to[-1,1]$ be such that for all $\b x\in \R^2$
both  functions $\omega(\b x, \cdot)$ and $\omega(\cdot,\b x)$ 
satisfy~\eqref{assumption:rimint}-\eqref{assumption:fflipstz}. 
Then for  $x\geq 1 $ we have 
  \begin{align*} \sum_{\substack{\vn,\vn'\in  x\c B\cap\ZZ^2 
  \\ \gcd(n_1,n_2)=1\\ \gcd(n'_1,n'_2)=1}}\omega(\vn,\vn')P(\b n,\b n')
    &= \frac{x^4\phi^\dagger(K)^2}{\zeta(2)^2K^4}
    \int_{\c B^2} \omega(\vu,\vu')\mathrm d\vu\mathrm d\vu' 
    \sum_{\b t,\b t'\in (\Z/K\Z)^{2*} } P(\b t,\b t')
    \\&+O\left(\frac{K^3 x^3 (\log x)}{\gamma}\right).
  \end{align*}
\end{lemma}

\begin{proof} We fix $\b n'$ and use Lemma~\ref{lem:V_sum_n_appendix}
to sum over $\b n$. The main term equals
\begin{equation}\label{eq:V_sum_asymp}\frac{x^2\phi^\dagger(K)} {\zeta(2)K^2}
  \int_{\c B}\sum_{\b t \in (\Z/K\Z)^{2*} } 
 \sum_{\substack{\vn'\in  x\c B\cap\ZZ^2 \\ \gcd(n'_1,n'_2)=1}}
  \omega(\vu,\vn') P(\b t, \b n')\mathrm d\vu  
  \end{equation}and the error term is admissible.
Using Lemma~\ref{lem:V_sum_n_appendix} for the sum over $\b n'$ yields
an acceptable error term while  the main term is as claimed. 
\end{proof}

\bibliographystyle{plain}
\bibliography{bibliography}
\end{document}